\documentclass{amsart}
\usepackage[utf8]{inputenc}
\usepackage{amsmath}
\usepackage{amssymb}
\usepackage{amsthm}
\usepackage[style=alphabetic,maxbibnames=8,safeinputenc,backend=biber]{biblatex}

\usepackage[hyperfootnotes=false]{hyperref}
\usepackage{enumerate}
\usepackage{graphicx}
\usepackage{tikz-cd}
\usepackage[labelfont=bf]{caption}
\usepackage{accents}

\usepackage{todonotes}
\usepackage{fancyhdr}
 
\newcommand{\complex}{\mathbb{C}}
\newcommand{\real}{\mathbb{R}}
\newcommand{\ints}{\mathbb{Z}}

\newcommand{\nats}{\mathbb{N}}

\newcommand{\proj}{\mathbf{P}}
\newcommand{\Gr}{\mathrm{Gr}}
\newcommand{\HH}{\mathbb{H}}

\DeclareMathOperator{\PGL}{PGL}
\DeclareMathOperator{\GL}{GL}
\DeclareMathOperator{\SL}{SL}
\DeclareMathOperator{\PSL}{PSL}

\DeclareMathOperator{\SO}{SO}
\DeclareMathOperator{\PSO}{PSO}

\newcommand{\eps}{\epsilon}
\newcommand{\del}{\partial}

\DeclareMathOperator{\id}{id}

\DeclareMathOperator{\diam}{diam}
\DeclareMathOperator{\Cay}{Cay}
\DeclareMathOperator{\Stab}{Stab}
\newcommand{\actson}{\curvearrowright}

\DeclareMathOperator{\Aut}{Aut}

\newcommand{\Nth}{^{\mathrm{th}}}

\newcommand{\spacer}{}

\newcommand{\ubar}[1]{\underaccent{\bar}{#1}}

\theoremstyle{plain}
\newtheorem{thm}{Theorem}[section]
\newtheorem{lem}[thm]{Lemma}
\newtheorem{prop}[thm]{Proposition}
\newtheorem{cor}[thm]{Corollary}
\newtheorem{defn}[thm]{Definition}

\theoremstyle{definition}
\newtheorem{eg}[thm]{Example}
\newtheorem{rmk}[thm]{Remark}

\title{Relatively dominated representations}
\author{Feng Zhu}
\date{}

\begin{document}

\begin{abstract}
%Anosov representations give a higher-rank analogue of convex cocompactness in a rank-one Lie group which shares many of its good geometric and dynamical properties; geometric finiteness in rank one may be seen as a controlled weakening of convex cocompactness to allow for isolated failures of hyperbolicity. 
We introduce relatively dominated representations as a relativization of Anosov representations, or in other words a higher-rank analogue of geometric finiteness. We prove that groups admitting relatively dominated representations must be relatively hyperbolic, that these representations induce limit maps with good properties, provide examples, and draw connections to work of Kapovich--Leeb which also introduces higher-rank analogues of geometric finiteness.
\end{abstract}

\maketitle

\section{Introduction}

Given a rank-one semisimple Lie group $G$ such as $\SL(2,\real)$ or $\SL(2,\complex) \cong \SO(1,3)$, the notion of convex cocompactness, first introduced in the setting of Kleinian groups acting on $\HH^3$, gives us a stable class of subgroups with good geometric and dynamical properties. 

When $G$ is instead a higher-rank semisimple Lie group, such as $\SL(d,\real)$ with $d \geq 3$, Anosov subgroups  are, at present, the best analogue of convex cocompact ones. These were originally defined in \cite{Labourie}, as a tool to study the dynamics and geometry of individual Hitchin representations, and further developed 
in \cite{GW}. There have subsequently been many other equivalent characterizations: see for instance \cite{KLP}, \cite{GGKW}, and \cite{BPS}.

In rank one, the class of convex cocompact subgroups form part of the strictly larger class of geometrically finite subgroups, which may be understood as convex cocompactness with the possible addition of certain degenerate ``cuspidal'' ends with controlled geometry. Geometrically finite groups continue to have many of the good properties of convex cocompact groups, modulo mild degeneracy at the cusps which may need controlled by additional hypotheses.

In prior work \cite{KL}, Kapovich and Leeb proposed relativized versions of the Anosov condition, which may be considered to be higher-rank analogues of  geometric finiteness.
In this paper we propose another, inspired by the characterization in \cite{BPS} and making use of the theory of relatively hyperbolic groups.

Below, all of our groups $\Gamma$ will be finitely-generated, and, to avoid unnecessary additional technicalities, torsion-free.

The condition on representations which we wish to define is given in terms of singular values and subspaces, and in terms of a modified word-length: 
given a matrix $A \in \GL(d,\real)$, let $\sigma_i(A)$ (for $1 \leq i \leq d$) denote the $i\Nth$ singular value of $A$.

Fix $\Gamma$ a finitely-generated torsion-free group
and a finite collection $\mathcal{P}$ of finitely-generated subgroups satisfying certain conditions (RH) (described in Definition \ref{defn:peri_preconds}) which are automatic if $\Gamma$ is hyperbolic relative to $\mathcal{P}$. We will designate the subgroups in $\mathcal{P}$ and their conjugates ``peripheral''. 

Given $\Gamma$ and $\mathcal{P}$ as above, we will say that the images of peripheral subgroups under a representation $\rho: \Gamma \to \GL(d,\real)$ are well-behaved if they satisfy certain conditions which essentially ensure their images are parabolic, plus mild technical conditions governing the behaviors of limits of Cartan projections. All of these conditions are described precisely in Definition \ref{defn:peri_conds}.

Let $X$ be a cusped space for $(\Gamma,\mathcal{P})$ as constructed in \cite{GrovesManning} (see \S2 for definitions.) Write $d_c$ to denote the metric on $X$, and $|\cdot|_c := d_c(\id, \cdot)$. These are defined in \cite{GrovesManning} in the case where $\Gamma$ is hyperbolic relative to $\mathcal{P}$, but the same construction can be done and continues to make sense in the more general case of $\Gamma$ a torsion-free finitely-generated group and $\mathcal{P}$ a malnormal finite collection of finitely-generated subgroups.

Given $\Gamma$ a finitely-generated torsion-free subgroup and a collection  $\mathcal{P}$ of finitely-generated subgroups satisfying (RH), we will say a representation $\rho: \Gamma \to \GL(d,\real)$ is {\bf 1-dominated relative to $\mathcal{P}$} (Definition \ref{defn:reldomrep}), if there exists constants $C, \mu > 0$ such that 
(D\textsuperscript{-}) for all $\gamma \in \Gamma$, $\frac{\sigma_1}{\sigma_{2}}(\rho(\gamma)) \geq C e^{\mu|\gamma|_c}$,
and the images of peripheral subgroups under $\rho$ are well-behaved.

Examples of relatively-dominated representations include geometrically-finite hyperbolic holonomies and geometrically-finite convex projective holonomies in the sense of \cite{CM12}; we also remark that in the case $\mathcal{P} = \varnothing$, we recover the \cite{BPS} definition of dominated reprsentations.

1-relatively dominated representations are discrete and faithful, and send  non-peripheral elements to proximal images. Their orbit maps are quasi-isometric embeddings of the relative Cayley graph, i.e. the Cayley graph with the metric induced from the cusped space $X \supset \Cay(\Gamma)$. 

In the setting of Anosov representations, \cite{KLP_Morse} proved that if $\Gamma$ is finitely-generated and $\rho: \Gamma \to \GL(d,\real)$ is such that there exist constants $C, \mu > 0$ so that $\frac{\sigma_1}{\sigma_2}(\rho(\gamma)) \geq C e^{\mu|\gamma|}$ for all $\gamma \in \Gamma$, then $\rho$ is ($P_1$)-Anosov, and in particular $\Gamma$ must be word-hyperbolic. An alternative proof of this appears in \cite{BPS} and was the original inspiration for this work. 
Here we can prove a relative analogue to this hyperbolicity theorem: 
\spacer \begin{thm} [Theorem \ref{thm:reldom_relhyp}]
If $\rho: \Gamma \to \GL(d,\real)$ is 
1-dominated relative to $\mathcal{P}$, and $\Gamma$ contains non-peripheral elements, then $\Gamma$ must be hyperbolic relative to $\mathcal{P}$.
\end{thm}

Moreover, given a 1-relatively dominated representation, we have limit maps from the Bowditch boundary $\del(\Gamma,\mathcal{P})$ with many of the good properties of Anosov limit maps:
\spacer \begin{thm}[Theorem \ref{thm:limitmaps}]
Given $\rho: \Gamma \to \GL(d,\real)$ 1-dominated relative to $\mathcal{P}$, we have well-defined, $\Gamma$-equivariant, continuous maps $\xi: \del(\Gamma, \mathcal{P}) \to \proj(\real^d)$ and $\xi^*: \del(\Gamma, \mathcal{P}) \to \proj(\real^d)^*)$ which are dynamics-preserving, compatible and transverse.
\end{thm}

A key technical input into the proofs of these theorems is a powerful generalization of the Oseledets theorem recently formulated in \cite{QTZ}; we will use a slightly modified version of this result, whose proof is discussed in Appendix \ref{app:QTZ}.

Our approach is different from that of \cite{KL}---the latter really focuses on the geometry of the symmetric space whereas we look more at the intrinsic geometry associated to the relatively hyperbolic group---but we show that the resulting notions are closely related:
\spacer 
\begin{thm}[Theorems \ref{thm:reldom_rRCA} and \ref{thm:rRCA_reldom}]
(a) If $\rho: \Gamma \to \SL(d,\real)$ is relatively dominated, then $\rho(\Gamma)$ is relatively RCA (in the sense of \cite{KL}) with uniformly regular peripherals.

(b) If $\rho: \Gamma \to \SL(d,\real)$ is such that $\rho(\Gamma)$ is relatively RCA with uniformly regular and undistorted peripherals satisfying an additional technical condition, then $\rho$ is relatively dominated.
\end{thm}

The rest of this paper is organized as follows: we start by reviewing relevant background facts on relatively hyperbolic groups in \S\ref{sec:relhyp} and on singular value decompositions in \S3. We then give the definition of relatively dominated representations, as well as noting some immediate properties, in \S\ref{sec:reldom}. \S\ref{sec:limtrans} proves a key transversality property, \S\ref{sec:reldom_relhyp} the relative hyperbolicity theorem, and \S\ref{sec:limitmaps} the existence of the limit maps. \S\ref{sec:eg} briefly discusses examples. \S\ref{sec:KL} describes links between the notion of relatively dominated representations introduced here and notions in \cite{KL}; finally, \S10 discusses extending the definition in \S\ref{sec:reldom} to more general semisimple Lie groups and parabolic subgroups.

Appendix \ref{app:linalg} collects various linear algebra lemmas which are used throughout, especially in the later sections; Appendix \ref{app:QTZ} contains a proof of the generalization of the Oseledets theorem alluded to above.

\subsection*{Acknowledgements}
%{\bf Acknowledgements.} 
The author wishes to thank Dick Canary for suggesting this question, and for many helpful discussions. We also thank Jean-Philippe Burelle, Jeff Danciger, Matt Durham, Ilya Gekhtman, Fanny Kassel, Wouter van Limbeek, Sara Maloni, Jason Manning, Max Riestenberg, Andr\'es Sambarino and Kostas Tsouvalas for their insights, advice, and encouragement, and Jairo Bochi for pointing out (via Andr\'es Sambarino) the work of \cite{QTZ}.

The author was partially supported by U.S. National Science Foundation (NSF) grant DMS 1564362 ``FRG: Geometric Structures on Higher Teichm\"uller Spaces'', and acknowledges support from NSF grants DMS 1107452, 1107263, 1107367 ``RNMS: Geometric Structures and Representation Varieties'' (the GEAR Network).
This project has also received funding from the European Research Council (ERC) under the European Union’s Horizon 2020 research and innovation programme (ERC starting grant DiGGeS, grant agreement No 715982).

\section{Relatively hyperbolic groups} \label{sec:relhyp}
Relative hyperbolicity is a group-theoretic notion---originally suggested by Gromov, and further developed by Bowditch, Farb, Groves--Manning, and others---of non-positive curvature inspired by the geometry of cusped hyperbolic manifolds and free products.  

The geometry of a relatively hyperbolic group is akin to the geometry of a cusped hyperbolic manifold in that it is negatively-curved outside of certain regions, which, like the cusps in a cusped hyperbolic manifold, can be more or less separated from each other. 

There are various ways to make this intuition precise, resulting in various equivalent characterizations of relatively hyperbolic groups. We will use a definition of Bowditch, in the tradition of Gromov: 

Consider a finite-volume cusped hyperbolic manifold with an open neighborhood of each cusp removed: call the resulting truncated manifold $M$. The universal cover $\tilde{M}$ of such a $M$ is hyperbolic space with a countable set of horoballs removed. The universal cover $\tilde{M}$ is not Gromov-hyperbolic; distances along horospheres that bound removed horoballs are distorted. If we glue the removed horoballs back in to the universal cover, however, the resulting space will again be hyperbolic space.

We can do a similar thing from a group-theoretic perspective: the Cayley graph of the fundamental group $\pi_1 M$ is not word-hyperbolic, because the cusp subgroups fail to quasi-isometrically embed into hyperbolic space. However, we can glue in metric graphs quasi-isometric to horoballs (``combinatorial horoballs'') along the subgraphs of the Cayley graph corresponding to these cusp subgroups, and the resulting space (a ``cusped space'' or ``augmented space'') will again be quasi-isometric to hyperbolic space. We then say that $\pi_1 M$ is hyperbolic {\it relative to} its cusp subgroups.

More precisely (and more generally), let $\Gamma$ be a finitely generated group and $S = S^{-1}$ a finite generating set. We consider the following construction:

\spacer \begin{defn}[\cite{GrovesManning}, Definition 3.1] \label{defn:combhoroball}
Given a subgraph $\Lambda$ of the Cayley graph $\Cay(\Gamma,S)$, the {\bf combinatorial horoball} based on $\Lambda$, denoted $\mathcal{H} = \mathcal{H}(\Lambda)$, is the 1-complex\footnote{Groves-Manning combinatorial horoballs are actually defined as 2-complexes; the definition here is really of a 1-skeleton of a Groves-Manning horoball. For metric purposes only the 1-skeleton matters.} formed as follows:
\begin{itemize}
\item the vertex set $\mathcal{H}^{(0)}$ is given by $\Lambda^{(0)} \times \ints_{\geq 0}$
\item the edge set $\mathcal{H}^{(1)}$ consists of the following two types of edges: \begin{enumerate}[(1)]
\item If $k \geq 0$ and $0 < d_\Lambda(v, w) \leq 2^k$, then there is a (``horizontal'') edge connecting
$(v, k)$ to $(w, k)$ 
\item If $k \geq 0$ and $v \in \Lambda^{(0)}$, there is a (``vertical'') edge joining $(v, k)$ to $(v, k + 1)$.
\end{enumerate} \end{itemize}
$\mathcal{H}$ is metrized by assigning length 1 to all edges.
\end{defn}

\spacer \begin{eg} \label{eg:horoball_qi}
The combinatorial horoball over $\ints^d$ is quasi-isometric to a horoball in $\HH^{d+1}$, via the map sending $(\vec{v}, n)$ in $\ints^d \times \nats$ to $(\vec{v}, e^n)$ in the upper half-space.
% (The proof is a relatively straightforward computation.)
\end{eg}
% Note also \cite{GrovesManning}, Theorem 3.8: combinatorial horoballs are uniformly hyperbolic (independent of underlying metric graph!)

Next let $\mathcal{P}$ be a finite collection of finitely-generated subgroups of $\Gamma$, and suppose $S$ is a {\bf compatible generating set}, i.e. for each $P \in \mathcal{P}$, $S \cap P$ generates $P$. 

\spacer \begin{defn}[\cite{GrovesManning}, Definition 3.12] \label{defn:cuspedspace}
Given $\Gamma, \mathcal{P}, S$ as above, the {\bf cusped space} $X(\Gamma, \mathcal{P}, S)$ is the simplicial metric space 
\[ \Cay(\Gamma,S) \cup \bigcup \mathcal{H}(\gamma P) \]
where the union is taken over all left cosets of elements of $\mathcal{P}$, i.e. over $P \in \mathcal{P}$ and (for each $P$) $\gamma P$ in a collection of representatives for left cosets of $P$. 

Here the induced subgraph of $\mathcal{H}(tP)$ on the $tP \times \{0\}$ vertices is identified with (the induced subgraph of) $tP \subset \Cay(\Gamma,S)$ in the natural way.
\end{defn}

\spacer \begin{defn} \label{defn:relhyp}
$\Gamma$ is hyperbolic relative to $\mathcal{P}$ if and only if the cusped space $X(\Gamma,\mathcal{P},S)$ is $\delta$-hyperbolic (for any compatible generating set $S$.) 

We will also call $(\Gamma, \mathcal{P})$ a {\bf relatively hyperbolic structure}.
\end{defn}

We remark that cusped spaces are quasi-isometry invariant for relatively hyperbolic groups (\cite{Groff}, Theorem 6.3):
in particular, the notion above is well-defined independent of the choice of generating set $S$.
There is a natural action of $\Gamma$ on the cusped space $X = X(\Gamma,\mathcal{P},S)$;
with respect to this action, the quasi-isometry between two cusped spaces $X(\Gamma,\mathcal{P},S_i)$ ($i=1,2$) is $\Gamma$-equivariant. 

% Re paragraph in 3.3.1 of \cite{KL} about Gromov model being not-canonical: yes, this is just one possible Gromov model. But once we've fixed the horoballs thus, then we have a ``canonical'' space---see also \cite{KL} Remark 3.17

In particular, this gives us a notion of a boundary associated to the data of a relatively hyperbolic group $\Gamma$ and its peripheral subgroup $\mathcal{P}$:
\spacer \begin{defn} \label{defn:bowditch_bdy}
For  $\Gamma$ hyperbolic relative to $\mathcal{P}$, the {\bf Bowditch boundary} $\del (\Gamma, \mathcal{P})$ is defined as the Gromov boundary $\del_\infty X$ of any cusped space $X = X(\Gamma,\mathcal{P},S)$.
\end{defn}
By the remarks above, this is well-defined up to homeomorphism, independent of the choice of compatible generating set $S$ (\cite{Bowditch}, \S9.)

The following terminology will be useful further below:
\spacer \begin{defn}
 $\Cay(\Gamma, S)$ considered as a subspace of $X(\Gamma,\mathcal{P},S)$---i.e. with the metric inherited from $X(\Gamma,\mathcal{P},S)$---will be called the {\bf relative Cayley graph}.
\end{defn}

Below, with a fixed choice of $\Gamma$, $\mathcal{P}$ and $S$ as above, for $\gamma, \gamma' \in \Gamma$, $d(\gamma, \gamma')$ will denotes the distance between $\gamma$ and $\gamma'$ in the Cayley graph with the word metric, and $|\gamma| := d(\id, \gamma)$ denotes word length in this metric. Similarly, $d_c(\gamma, \gamma')$ denotes distance in the corresponding cusped space and $|\gamma|_c := d_c(\id,\gamma)$ denotes cusped word-length.

\subsection{A Bowditch--Yaman criterion for relative hyperbolicity}

The Bowditch criterion \cite{Bowditch_crit} states, roughly speaking, that we can show a group $\Gamma$ is hyperbolic by exhibiting an action of $\Gamma$ on a metric space satisfying certain properties which are characteristic of the action of a hyperbolic group on its Gromov boundary. Moreover, if the hypotheses are satisfied, the space (and action) we produce is naturally identified with the Gromov boundary of the group (and the action of the group thereon.)

% There is a relative version of this: Bowditch showed that a group $\Gamma$ that is hyperbolic relative to a collection $\mathcal{P}$ of peripheral subgroups has a well-defined boundary $\del(\Gamma,\mathcal{P})$ 
% In short, $\del(\Gamma,\mathcal{P})$ can be taken to be the Gromov boundary of the cusped space of $X(\Gamma,\mathcal{P},S)$ (for any compatible generating set $S$.) `
Using the Bowditch boundary and generalizing Bowditch's arguments, Asli Yaman proved an analogue of Bowditch's criterion for relatively hyperbolic groups:

\spacer \begin{defn}
If $M$ is a compact metric space, $\Gamma \actson M$ {\bf as a convergence group} if the induced action on the space $M^{(3)}$ of distinct triples is properly discontinuous. 

$\Gamma \actson M$ as a {\bf geometrically-finite} convergence group if every point in $M$ is either a conical limit point or a bounded parabolic point.

($x \in M$ is a {\bf conical limit point} if there exists a sequence $(g_i) \subset \Gamma$ and $a, b \in M$ ($a \neq b$) such that $g_ix \to a$ and $g_i y \to b$ for any $y \in M \setminus \{x\}$. 

$H \leq \Gamma$ is {\bf parabolic} if it is infinite, fixes some point of $M$, and contains no infinite-order element with fixed locus of size 2. Such $H$ have unique fixed points in $M$, called {\bf parabolic points}.
A parabolic point $x \in M$ is {\bf bounded} if $(M \setminus \{x\}) / \Stab_\Gamma(x)$ is compact.)
\end{defn}

\spacer \begin{thm}[\cite{Yaman}, Theorem 0.1] \label{thm:Yaman}
Suppose that $M$ is a non-empty, perfect, compact metric space, and $\Gamma \actson M$ as a geometrically-finite convergence group.

Suppose also that the stabiliser of each bounded parabolic point is finitely generated.

Then $\Gamma$ is hyperbolic relative to the collection $\mathcal{P}$ of its maximal parabolic subgroups, and $M$ is equivariantly homeomorphic to $\del (\Gamma,\mathcal{P})$.
\end{thm}

Gerasimov in \cite{Gerasimov2009} has shown that geometric finiteness, as well as the finite generation of the parabolic stabilisers, can be characterized using the induced group action on the space of distinct pairs. Putting these together, we obtain
\spacer \begin{thm} \label{thm:Gerasimov}
Suppose $\Gamma$ is finitely-generated and $M$ is a non-empty, perfect, compact metrizable space, and $\Gamma \actson M$ is such that the induced action on $M^{(3)}$ is properly discontinuous and the induced action on $M^{(2)}$ is cocompact.

Then $\Gamma$ is hyperbolic relative to the maximal parabolic subgroups of the action $\Gamma \actson M$.
\end{thm}
% Gerasimov also shows that $\Gamma$ f.g. implies parabolic subgroups f.g. and undistorted* *in the Cayley graph---parabolic subgroups ARE distorted in the cusped space.

\subsection{Geodesics in the cusped space} \label{sub:hat_unhat}

Let $\Gamma$ be a finitely-generated group, $\mathcal{P}$ be a malnormal finite collection of finitely-generated subgroups, and let $S = S^{-1}$ be a compatible finite generating set as above. Let $X = X(\Gamma, \mathcal{P}, S)$ be the cusped space, and $\Cay(\Gamma) = \Cay(\Gamma,S)$ the Cayley graph.

We emphasize that none of the results in this or the next subsection requires $\Gamma$ to be relatively hyperbolic, although the motivation for the constructions involved comes from relative hyperbolicity. This will be useful below, in the proof of the relative hyperbolicity theorem (Theorem \ref{thm:reldom_relhyp}.)

We start by pointing out a family of preferred geodesics
% \footnote{These are not the same as, but close to, the preferred paths defined by Groves--Manning (\cite{GrovesManning}, Definition 5.4); we will not use Groves--Manning preferred paths and the mild clash in terminology will hopefully not be an issue.} 
in the combinatorial horoballs:
\spacer \begin{lem}[\cite{GrovesManning}, Lemma 3.10] \label{lem:gm310}
Let $\mathcal{H}(\Gamma)$ be a combinatorial horoball. Suppose that $x,y \in \mathcal{H}(\Gamma)$ are distinct vertices. Then there is a geodesic $\gamma(x,y) = \gamma(y,x)$ between $x$ and $y$ which consists of at most two vertical segments and a single horizontal segment of length at most 3. 
\end{lem}

We will call any such geodesic a {\bf preferred geodesic}.

We have the following estimate going between uncusped and cusped lengths:
\spacer \begin{prop} \label{prop:wordlengths_bilip_log}
Suppose $\gamma$ is a word contained in a single peripheral subgroup. 

Then
$\frac{2}{\log 2} \log |\gamma| \leq |\gamma|_c \leq \frac{2}{\log 2} \log |\gamma| + 1$, or equivalently
$\frac{1}{\sqrt{2}} \sqrt{2}^{|\gamma|_c} \leq |\gamma| \leq \sqrt{2}^{|\gamma|_c}$

\begin{proof}
% Note the second set of inequalities is equivalent to the first. 
Let $\gamma$ be an peripheral element of $\Gamma$ which can be written as a word of word-length $L$.

There is always a path in the cusped space $X$ from $\id$ to $\gamma$ which consists of going up $\left \lfloor \log_2 L \right\rfloor$, going across 1, and then going down $\left\lfloor \log_2 L \right\rfloor$, and so the cusped word-length is certainly bounded from above by $2 \log_2 L + 1 = \frac{2}{\log 2} \log L + 1$.

Conversely, any path in $X$ of cusped length at most $2 \log_2 L-1$ with a single horizontal segment of (cusped) length $\ell$ can correspond to a word of word-length at most $\ell \cdot 2^{\log_2 L - \frac{\ell+1}2} = 2^{-\frac{\ell+1}2}\ell L < L$ whenever $\ell\geq 1$

Note that any path in $X$ which has two distinct endpoints in $\Cay(\Gamma) \subset X$ must contain at least one horizontal edge. 
By Lemma \ref{lem:gm310}, there is always a geodesic in the cusped space from $\id$ to $\gamma$ consisting of at most two vertical segments and a single horizontal segment.

Hence the cusped word-length is bounded from below by $2 \log_2 L = \frac{2}{\log 2} \log L$, as desired.
\end{proof}
\end{prop}

Given a path $\gamma: I \to \Cay(\Gamma)$ in the Cayley graph 
% (e.g. a quasigeodesic or geodesic path)
such that $\gamma(I \cap \ints) \subset \Gamma$, we can consider $\gamma$ as a {\bf relative path}
% \footnote{Borrowing terminology used in \cite{Bowditch}, although with some mild differences in the details: in that setting ``path'' is used in the graph theoretic sense of ``sequence of adjacent vertices.''}
$(\gamma, H)$, where $H$ is a subset of $I$ consisting of a disjoint union of finitely many subintervals $H_1, \dots, H_n$  occurring in this order along $I$, such that each $\eta_i := \gamma|_{H_i}$ is a maximal subpath lying in a closed combinatorial horoball $B_i$, 
and $\gamma|_{I \smallsetminus H}$ contains no edges of $\Cay(\Gamma)$ labelled by a peripheral generator.

Similarly, a path $\hat\gamma: \hat{I} \to X$ in the cusped space with endpoints in $\Cay(\Gamma) \subset X$ may be considered as a relative path $(\hat\gamma, \hat{H})$, where $\hat{H} = \coprod_{i=1}^n \hat{H}_i$, $\hat{H}_1, \dots, \hat{H}_n$ occur in this order along $\hat{I}$, each $\hat\eta_i := \hat\gamma|_{\hat{H}_i}$ is a maximal subpath in a closed combinatorial horoball $B_i$, and $\hat\gamma|_{\hat{I} \smallsetminus \hat{H}}$ lies inside the Cayley graph. Below, we will consider only geodesics and quasigeodesic paths $\hat\gamma: \hat{I} \to X$ where all of the $\hat\eta_i$ are preferred geodesics (in the sense of Lemma \ref{lem:gm310}.)

We will refer to the $\eta_i$ and $\hat\eta_i$ as {\bf peripheral excursions}. We remark that the $\eta_i$, or any other subpath of $\gamma$ in the Cayley graph, may be considered as a word and hence a group element in $\Gamma$; this will be used without further comment below.

Given a path $\hat\gamma: \hat{I} \to X$ whose peripheral excursions are all preferred geodesics, we may replace each excursion $\hat\eta_i = \hat\gamma|_{\hat{H}_i}$ into a combinatorial horoball with a geodesic path (or, more precisely, a path with geodesic image) $\eta_i = \pi \circ \hat\eta_i$ in the Cayley (sub)graph of the corresponding peripheral subgroup connecting the same endpoints, by omitting the vertical segments of the preferred geodesic $\hat\eta_i$ and replacing the horizontal segment with the corresponding segment at level 0, i.e. in the Cayley graph.\footnote{As a parametrized path this has constant image on the subintervals of $\hat{H}_i$ corresponding to the vertical segments, and travels along the projected horizontal segment at constant speed.} We call this the ``project'' operation, since it involves ``projecting'' paths inside combinatorial horoballs onto the boundaries of those horoballs. This produces a path $\gamma = \pi\circ\hat\gamma: \hat{I} \to \Cay(\Gamma)$. 

% More precisely, the parametrization $\gamma$ coincides with $\hat\gamma$ off the $\hat{H}_i$, and may be defined on the $\hat{H}_i$ as follows: let $H_i \subset \hat{H}_i$ be such that $\hat\gamma(H_i)$ is the horizontal segment of the preferred geodesic $\hat\eta_i = \hat\gamma|_{\hat{H}_i}$, define $\eta_i: H_i \to \Cay(\Gamma)$ to be the arc-length parametrization of the corresponding horizontal path at level 0 of the combinatorial horoball that matches at the endpoints. $\hat{H}_i \smallsetminus H_i$ consists of two connected components; we have $\eta_i$ send each of these to the corresponding endpoint of $\hat\eta_i(\hat{H}_i)$.

Below, given any path $\alpha$ in the Cayley graph with endpoints $g, h \in \Gamma$, or any path $\hat\alpha$ in the cusped space with endpoints in $g, h \in X$, we write $\ell(\alpha)$ to denote $d(g,h)$ i.e. distance measured according to the word metric in $\Cay(\Gamma)$, and $\ell_c(\hat\alpha)$) to denote $d_c(g,h)$, where $d_c$ denotes distance in the cusped space.
% the metric described at the beginning of this section.

The following observation will be used many times below:
\spacer \begin{prop} \label{prop:unhat_distance}
Given a geodesic $\hat\gamma: \hat{J} \to X$ with endpoints in $\Cay(\Gamma) \subset X$ and whose peripheral excursions are all preferred geodesics, let $\gamma = \pi \circ \hat\gamma: \hat{J} \to \Cay(\Gamma)$ be its projected image. 

Given any subinterval $[a,b] \subset \hat{J}$, consider the subpath $\gamma|_{[a,b]}$ as a relative path $(\gamma|_{[a,b]}, H)$ where $H = (H_1, \dots, H_n)$, and write $\eta_i := \gamma|_{H_i}$; then we have the biLipschitz equivalence
\[ \frac13 \leq \frac{d_c(\gamma(a), \gamma(b))}{\ell(\gamma|_{[a,b]}) - \sum_{i=1}^n \ell(\eta_i) + \sum_{i=1}^n \hat\ell(\eta_i)} \leq \frac{2}{\log 2} + 1 < 4 \]
where $\hat\ell(\eta_i) := \max\{\log(\ell(\eta_i)), 1\}$.
\begin{proof}
If $\gamma|_{[a,b]}$ lies in a single peripheral excursion, then this follows from the fact that the projection operation replaces excursions with geodesic paths in the Cayley graph and from Proposition \ref{prop:wordlengths_bilip_log}.

More generally, since we start with a geodesic in the cusped space, we have
\begin{align}
d_c(\gamma(a), \gamma(b)) \leq \ell_c(\gamma|_{[a,b] \setminus H}) + \sum_{i=1}^n \ell_c(\eta_i). \label{ineq:dc_geod_upper}
\end{align}
Here $\gamma|_{[a,b] \setminus H}$ is a disjoint union of subpaths $\gamma_1, \dots, \gamma_k$ of $\gamma$ with endpoints in $\Gamma$, and $\ell_c(\gamma|_{[a,b] \setminus H}) := \sum_{i=1}^k \ell_c(\gamma_i)$, where $\ell_c(\gamma_i)$ denotes cusped distance between the endpoints of the subpath $\gamma_i$.

If the endpoints of our subpath do not lie in the middle of a (projected) peripheral excursion, we can promote the inequality (\ref{ineq:dc_geod_upper}) to an equality
\begin{align}
d_c(\gamma(a), \gamma(b)) = \ell_c(\gamma|_{[a,b] \setminus H}) + \sum_{i=1}^n \ell_c(\eta_i). \tag{1'} \label{eq:dc_geod}
\end{align}

Now suppose one of our endpoints, say $b$, does lie in the middle of a projected peripheral excursion, say $\eta_n$. (The case where $a$ lies in the middle of an excursion will be similar.) This is the special case which will take the remaining time:

Let $b^-$ be such that $\hat\gamma(b^-)$ is the endpoint of $\eta_n$ between $\gamma(a)$ and $\gamma(b)$. The infinite vertical ray into the combinatorial horoball from $\gamma(b)$ hits the image of $\hat\gamma$ at the point $\hat\gamma(b)$. We remark that, by the properties of the project operation, $\gamma(a) = \hat\gamma(a)$ and $\gamma(b^-) = \hat\gamma(b^-)$.

Note $\hat\gamma|_{[a,b]}$ is a geodesic, so 
by the triangle inequality
\begin{align} 
d_c(\gamma(a), \gamma(b)) + d_c(\hat\gamma(b), \gamma(b)) & \geq d_c( \gamma(a),\hat\gamma(b) ) \notag \\ 
 & = d_c( \gamma(a), \gamma(b^-) ) + d_c( \gamma(b^-), \hat\gamma(b) ) 
\label{ineq:part_exc_triang}
\end{align}

Moreover, $[\gamma(b), \hat\gamma(b)]$ consists of a single vertical segment, (an isometric translate of) which is a subpath of $\hat\gamma|_{[b^-, b]}$, so $d_c(\gamma(b), \hat\gamma(b)) \leq d_c(\gamma(b^-), \hat\gamma(b))$.
Combining these observations with (\ref{ineq:part_exc_triang}), we obtain
\begin{align*}
d_c(\gamma(a), \gamma(b)) + d_c(\hat\gamma(b), \gamma(b)) & \geq  d_c( \gamma(a), \gamma(b^-) ) + d_c(\gamma(b^-), \hat\gamma(b)) \end{align*}
% \qquad\mbox{ so}\\
so
\begin{align*}
d_c(\gamma(a), \gamma(b)) & \geq d_c( \gamma(a), \gamma(b^-) ) + d_c(\gamma(b^-), \hat\gamma(b)) -  d_c(\hat\gamma(b), \gamma(b)) \\
 & \geq d_c(\gamma(a), \gamma(b^-) ) .
\end{align*}

\begin{figure}[ht]
    \centering
    \includegraphics[width=0.4\textwidth]{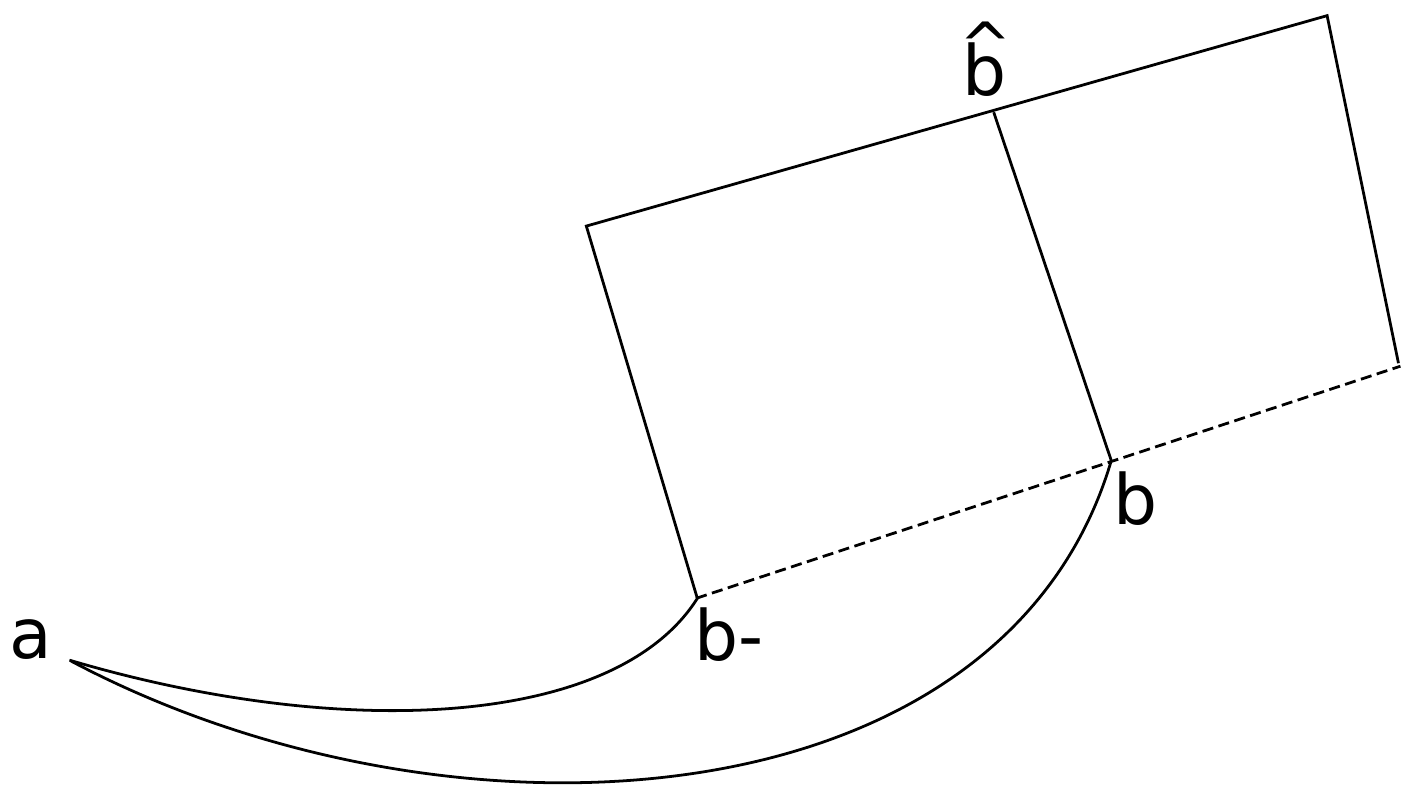}
    \caption{Solid lines here indicated geodesics in $X$, dotted lines indicate projected geodesics}
    \label{fig:prop:unhat_distance}
\end{figure}

On the other hand, again applying the triangle equality 
(and multiplying both sides by $\frac12$) we have
\[ \frac12 \left( d_c(\gamma(b^-),\gamma(b)) - d_c(\gamma(a), \gamma(b^-)) \right) \leq \frac12 d_c(\gamma(a), \gamma(b)) .\]

Adding together these inequalities, we obtain
\[ \frac32 d_c(\gamma(a), \gamma(b)) \geq \frac12 \left( d_c(\gamma(a), \gamma(b^-)) + d_c(\gamma(b^-), \gamma(b)) \right)
.\]

Now apply (\ref{ineq:dc_geod_upper}) to $\gamma|_{[a,b^-]}$, where we have equality, and remark that $d_c(\gamma(b^-), \gamma(b)) = \ell_c(\eta_n)$ by the properties of the project operation, so that we may rewrite this inequality as 
\begin{align*}
d_c(\gamma(a), \gamma(b)) & \geq \frac13 \left( d_c(\gamma(a), \gamma(b^-)) + d_c(\gamma(b^-), \gamma(b)) \right)  \\
 & = \frac13 \left( \ell(\gamma|_{[a,b] \setminus H}) + \sum_{i=1}^{n-1} \ell_c(\eta_i) + \ell_c(\eta_n) \right) 
\end{align*}
and so we have $d_c(\gamma(a), \gamma(b)) \geq \frac13 \left( \ell_c(\gamma|_{[a,b] \setminus H}) + \sum_{i=1}^n \ell_c(\eta_i) \right)$.

By the definition of the cusped metric and of a relative path, 
\[ \ell_c(\gamma|_{[a,b] \setminus H}) = \ell(\gamma|_{[a,b] \setminus H}) = \ell(\gamma|_{[a,b]}) - \sum_{i=1}^n \ell(\eta_i) .\] 
By Proposition \ref{prop:wordlengths_bilip_log}, for each $i$ between 1 and $n$, 
\[ \frac{2}{\log 2} \log \ell(\eta_i) \leq \ell_c(\eta_i) \leq \frac{2}{\log 2} \log \ell(\eta_i) + 1 .\]

Hence, writing $L := \ell(\gamma|_{[a,b]}) - \sum_{i=1}^n \ell(\eta_i) + \sum_{i=1}^n  \hat\ell(\eta_i)$, we have 
\[ \frac13 L \leq d_c(\gamma(a), \gamma(b)) \leq \frac{2}{\log 2} L + n \leq \left(\frac{2}{\log 2} + 1\right) L \] as desired.
\end{proof} \end{prop}

In particular, we note the following very coarse equivalence statement:
\spacer \begin{cor} \label{cor:wordlengths_comp}
For any sequence of elements $(\gamma_n) \subset \Gamma$, $|\gamma_n|_c \to \infty$ if and only if $|\gamma_n| \to \infty$.
\end{cor} 

\subsection{Reparametrizing projected geodesics} \label{sub:reparam_proj}

Given a geodesic segment $\hat\gamma$ in the cusped space with endpoints in $\Cay(\Gamma)$, we can take its projection $\gamma = \pi \circ \hat\gamma: \hat{I} \to \Cay(\Gamma)$ and then reparametrize it in such a way that the increments correspond, approximately, to linear increments in cusped distance. Slightly more generally we will find it useful to consider paths in $\Cay(\Gamma)$ that ``behave metrically like quasi-geodesics in the relative Cayley graph'', in the following sense: 
\spacer \begin{defn} \label{defn:projgeod_depth}
Given any path $\gamma: I \to \Cay(\Gamma)$ such that $I$ has integer endpoints and $\gamma(I \cap \ints) \subset \Gamma$,
% and with at least one end not inside a peripheral coset, 
define the {\bf depth} $\delta(n) = \delta_\gamma(n)$ of a point $\gamma(n)$ (for any $n \in I \cap \ints$) as \begin{enumerate}[(a)]
\item the smallest integer $d$ such that at least one of $\gamma_r(n-d)$, $\gamma_r(n+d)$ is well-defined (i.e. $\{n-d, n+d\} \cap I \neq \varnothing$) and not in the same peripheral coset as $\gamma(n)$, {\bf or}
\item if no such integer exists, $\min\{\sup I - n, n - \inf I\}$.
\end{enumerate}
\end{defn}

\spacer \begin{defn} \label{defn:metric_proj_qgeod}
Given constants $\ubar\upsilon, \bar\upsilon > 0$, an {\bf $(\ubar\upsilon,\bar\upsilon)$-metric quasigeodesic path} is a path $\gamma: I \to \Cay(\Gamma)$ 
with $\gamma(I \cap \ints) \subset \Gamma$
such that
for all integers $m, n \in I$, 
\begin{enumerate}[(i)]
\item $ |\gamma(n)^{-1}\gamma(m)|_c \geq \ubar\upsilon^{-1} |m-n| - \ubar\upsilon$, 
\item $ |\gamma(n)^{-1}\gamma(m)|_c \leq \bar\upsilon(|m-n| + \min\{\delta(m), \delta(n)\} ) + \bar\upsilon$, and
\item if $\gamma(n)^{-1} \gamma(n+1) \in P$ for some $P \in \mathcal{P}$, we have $\gamma(n)^{-1} \gamma(n+1) = p_{n,1} \cdots p_{n,\ell(n)}$ where each $p_{n,i}$ is a peripheral generator of $P$, and $$2^{\delta(n)-1} \leq \ell(n) = |\gamma(n)^{-1}\gamma(n+1)| \leq 2^{\delta(n)+1}.$$ 
\end{enumerate}
\end{defn}

We can now make more precise our assertion about reparametrizing projected geodesic segments: 

\spacer \begin{prop} \label{prop:reparam_proj_geod}
Given a cusped space $X = X(\Gamma,\mathcal{P},S)$, 
for any projected geodesic $\gamma = \pi \circ \hat\gamma: I \to \Cay(\Gamma)$ with at least one end not inside a peripheral coset, we have a reparametrization of its image $\gamma_r: I_r \to \Cay(\Gamma)$ which is a $(6,20)$-metric quasigeodesic path. (In fact, we can improve the inequalities slightly so that for all integers $m, n \in I_r$, 
\begin{enumerate}[(i)]
\item $ |\gamma_r(n)^{-1}\gamma_r(m)|_c \geq \frac16 |m-n|$, and
\item $ |\gamma_r(n)^{-1}\gamma_r(m)|_c \leq 8(|m-n| + \min\{\delta(m), \delta(n)\}) + 20$.)
\end{enumerate}
\begin{proof}
We define the reparametrization as follows: 
\begin{itemize}
\item Outside of the peripheral excursions, parametrize by arc-length in $\Cay(\Gamma)$. 

\item Within a infinite but not bi-infinite peripheral excursion, the first letter is left alone, the next two are multiplied together, then the next four multiplied together, and so on.

\item Within a finite peripheral excursion of cusped length $E$, do this from both ends simultaneously, and do some rounding as necessary. More precisely, to each natural number $n$ we associate an ordered partition of positive integers as follows:
\begin{itemize}
\item If $n = 1 + 2 + \dots + 2^{k-1} + 2^k + 2^{k-1} + \dots + 1$ for some $k \in \ints_{\geq0}$, that is the associated ordered partition (e.g. $22 = 1 + 2 + 4 + 8 + 4 + 2 + 1$, so $(1,2,4,8,4,2,1)$ is the ordered partition associated to 22.) Call these numbers $n_k$. Note $n_k = 3 \cdot 2^k - 2$. 

\item If $n \in (n_k, n_{k+1})$, associate to $n$ the ordered partition $(1, 2, \dots, 2^k + (n-n_k), 2^{k-1}, \dots, 1)$. Note the middle term will be between $2^k+1$ and $2^k+(n_{k+1}-n_k-1) = 2^k + 3 \cdot 2^k - 1 = 2^{k+2}-1$ in this case.

e.g. $n=17 \in (n_2, n_3) = (10, 22)$, and so the ordered partition for 17 is given by $(1, 2, 4+7, 2, 1)=(1,2,11,2,1)$
\end{itemize}
Then take the ordered partition $(a_1, \dots, a_l)$ associated to $E$, and if $\gamma(s) = \gamma_r(s_r)$ is the start of the peripheral excursion, define $\gamma_r(s_r+j) = \gamma(m+\sum_{i=1}^j a_i)$ for $1 \leq j \leq l$.
\end{itemize}

To verify that this satisfies the desired criteria, we remark that the reparametrization does not modify cusped length outside of the peripheral excursions; inside a peripheral excursion of length $E$,
the sum of any $j$ consecutive numbers inside the partition associated to $E$ is at least 
\[ 1 + \dots + 2^{j-1} = 2^{j}-1 \] 
if $j$ is no more than half the length of the partition; if $j$ is greater than this threshold, this sum is still bounded below by
\[ 1 + \dots + 2^{\ell_2-1} = 2^{\ell_2}-1 \geq 2^{j/2}-1 ,\]
where $\ell_2$ is the floor of half the length of the partition, since the sum must contain a sum of $\ell_2$ consecutive numbers inside the partition.

Thus, by Proposition \ref{prop:wordlengths_bilip_log}, the cusped length of the part of the peripheral excursion associated to this part of the reparametrization is no less than $2 \log_2 (2^{j/2} - 1) \geq j-1$. Considering separately what happens for small values of $j$, we may further replace this lower bound with $j/2$.

Proposition \ref{prop:unhat_distance} then gives us
\[ d_c(\gamma_r(n), \gamma_r(m)) \geq \frac13 \left( \ell(\gamma_r|_{[n,m] \setminus I_\eta}) + \ell\left(\gamma_r|_{I_\eta} \right) \right) \geq \frac16|m-n| \]
This suffices to verify (i).

To verify (ii), we recall that, if $w_{m,n} := \gamma_r(m)^{-1} \gamma_r(n)$ 
is a peripheral word of length $\ell(w_{m,n})$, its cusped length is between $2 \log_2 \ell(w_{m,n})$ and $2 \log_2 \ell(w_{m,n})+1$ (see Proposition \ref{prop:wordlengths_bilip_log}.)

By construction $\ell(w_{m,m+1}) \leq 2^{\delta(m)+1}$, so $|w_{m,m+1}|_c \leq 2\delta(m)+3$, and more generally, 
\[ |w_{m,n}|_c \leq 2 \log_2 (2^{\delta(m)+1} + \dots + 2^{\delta(n)+1}) + 1 \]
and, writing $\delta = \min \{ \delta(m), \delta(n) \}$, this latter is bounded above by
\begin{align*}
2 \log_2 \left( 2^{\delta+1} + \dots + 2^{\delta+1+|m-n|} \right) + 1 & \leq 2 \log_2 \left( 2^{\delta+1} \cdot (2^{|m-n|+1}-1) \right) + 1 \\
 & \leq 2 (\delta + |m-n|) + 5  .
\end{align*}  
This, again in conjunction with Proposition \ref{prop:unhat_distance}, which yields
\[ d_c(\gamma_r(n), \gamma_r(m)) \leq 4 \left( \ell(\gamma_r|_{[n,m] \setminus I_\eta}) + \ell(\gamma_r|_{I_\eta}) \right) \leq 8\left( |m-n| + \min\{\delta(m), \delta(n)\} \right) + 20 ,\]
suffices to prove the Proposition.
\end{proof} 
\end{prop}
\section{Singular value decompositions} \label{sec:SVD}

The condition on representations which we will define is given in terms of singular values and subspaces: given a matrix $g \in \GL(d,\real)$, let $\sigma_i(g)$ (for $1 \leq i \leq d$) denote its $i\Nth$ singular value.

Measuring these requires specifying a norm on $\real^d$, although the conditions below are independent (up to possibly changing the constants) of this choice of norm. Below we will assume we have fixed a norm coming from an inner product on $\real^d$; by viewing the symmetric space $\SL(d,\real) / \SO(d)$ as a space of (homothety classes of) inner products on $\real^d$, this is equivalent to choosing a basepoint $o \in \SL(d,\real) / \SO(d)$ (and then arbitrarily fixing a scaling).

Furthermore, write $U_i(g)$ to denote the span of the $i$ largest axes in the image of the unit sphere in $\real^d$ under $g$, and $S_i(g) := U_i(g^{-1})$ (the letters come from ``Unstable'' and ``Stable''; these names are inspired by ideas from dynamics.)
Note $U_i(g)$ is well-defined if and only if we have a singular-value gap $\sigma_i(g) > \sigma_{i+1}(g)$.

More precisely, given any $g \in \GL(d,\real)$, we may write $g = KAL$, where $K$ and $L$ are orthogonal matrices and $A$ is a diagonal matrix with nonincreasing entries down the diagonal. $A$ is uniquely determined, and we may define $\sigma_i(g) = A_{ii}$. $U_i(g)$ is given by the span of the first $i$ columns of $K$, which is well-defined as long as $\sigma_i(g) > \sigma_{i+1}(g)$. 

We remark that, for $g \in \SL(d,\real)$, this singular-value decomposition is a (particular choice of) Cartan decomposition. We will occasionally write (given $g = KAL$ as above) 
\[ a(g) := (\log A_{11}, \dots, \log A_{dd}) = (\log \sigma_1(g), \dots, \log \sigma_d(g)) ;\] 
we note that the norm $\|a(g)\| = \sqrt{(\log \sigma_1(g))^2 + \dots + (\log \sigma_d(g))^2}$ is equal to the distance $d(o, g \cdot o)$ in the associated symmetric space $\SL(d,\real)/\SO(d)$ (see e.g. formula (7.3) in \cite{BPS}.)

\section{Relatively dominated representations} \label{sec:reldom}
Recall that $\Gamma$ is a finitely-generated group, which we assume to be torsion-free.

Let $\mathcal{P}$ be a finite collection of finitely-generated proper infinite subgroups; call all conjugates of these subgroups {\bf peripheral}. A element of $\Gamma$ is called peripheral if it belongs to any peripheral subgroup, and non-peripheral otherwise. Below we will write $\mathcal{P}^\Gamma$ to denote the set of all conjugates of groups in $\mathcal{P}$, $\bigcup \mathcal{P} := \bigcup_{P\in\mathcal{P}} P$ and $\bigcup \mathcal{P}^\Gamma := \bigcup_{Q\in\mathcal{P}^\Gamma} Q$ to denote the set of peripheral elements.

Let $S$ be a compatible generating set, and let $X = X(\Gamma, \mathcal{P}, S)$ be the corresponding cusped space
(see Definitions \ref{defn:combhoroball} and \ref{defn:cuspedspace} above.) As above, let $d_c$ denote the metric on $X$, and $|\cdot|_c := d_c(\id, \cdot)$ denote the cusped word-length.

For most of the arguments below we will also impose further conditions on $\mathcal{P}$:
\spacer \begin{defn} \label{defn:peri_preconds}
We say that a finite collection $\mathcal{P}$ of finitely-generated proper infinite subgroups satisfies (RH) if 
\begin{itemize}
\item (malnormality) $\mathcal{P}$ is malnormal, i.e. for all $\gamma \in \Gamma$ and $P, P' \in \mathcal{P}$, $\gamma P \gamma^{-1} \cap P' = 1$ unless $\gamma \in P = P'$; % (MN)
\item (non-distortion) there exists $\nu > 0$ such that for any infinite-order non-peripheral element $\gamma \in \Gamma$, $|\gamma^n|_c \geq \nu |n|$; % (UD)
\item (local-to-global) 
there exist $\ubar\upsilon, \bar\upsilon > 0$ and a constant $L > 0$ so that if $p=p_1....p_n$ is a geodesic word in $P\in\mathcal P$, $n>L$ and $\gamma p_1 \cdots p_L$ is a projected geodesic, then $\gamma p$ is an $(\ubar\upsilon, \bar\upsilon)$-metric projected quasigeodesic. % (LG)
\end{itemize}
\end{defn}

We remark that all of these conditions hold automatically if $\Gamma$ is hyperbolic relative to $\mathcal{P}$: malnormality follows for torsion-free $\Gamma$ from \cite{osinRH}, Theorem 1.4; non-distortion follows from \cite{osinRH}, Theorem 1.14; the local-to-global condition is a particular case of the much more general local-to-global properties that hold due to the hyperbolicity of the cusped space $X$ when $\Gamma$ is relatively hyperbolic. 

We introduce first a few technical conditions controlling what happens on the images of peripheral subgroups, and then the main notion we are defining:

\spacer \begin{defn} \label{defn:peri_conds}
Given $\Gamma$ and $\mathcal{P}$ as above, and a representation $\rho: \Gamma \to \GL(d,\real)$, we say that the peripheral subgroups have {\bf well-behaved images under $\rho$} if the following conditions are satisfied: 
\begin{itemize}%[(i)]
\item (upper domination) there exist constants $C_1, \mu_1 > 0$ such that $\sigma_1(\rho(\eta)) \leq C_1 e^{\mu_1|\eta|_c}$ for every peripheral element $\eta \in \bigcup\mathcal{P}$ 
\item (unique limits) for each $P \in \mathcal{P}$, there exists $\xi_\rho(P) \in \proj(\real^d)$ and  $\xi^*_\rho(P) \in \Gr_{d-1}(\real^d)$ such that for every sequence $(\eta_n) \subset P$ with $\eta_n \to \infty$, we have $\lim_{n\to\infty} U_1(\rho(\eta_n)) = \xi_\rho(P)$ and $\lim_{n\to\infty} U_{d-1}(\rho(\eta_n)) = \xi^*_\rho(P)$.
% (L)
\item (quadratic gaps) for every $\ubar\upsilon, \bar\upsilon > 0$, there exists
% $L>0$ and 
$C'\geq 0$ such that if $\eta \in P$ for some $P \in \mathcal{P}$, then, for any $\gamma\in\Gamma$, 
if $\gamma\eta$ ($\eta\gamma$, respectively) is 
an $(\ubar\upsilon,\bar\upsilon)$-metric quasigeodesic path
then $\frac{\sigma_1}{\sigma_2}(\rho(\gamma\eta)) \geq C'|\eta|^2 = C'e^{|\eta|_c}$ 
($\frac{\sigma_1}{\sigma_2}(\rho(\eta\gamma)) \geq C'|\eta|^2$, resp.);

\item (uniform transversality) %[(PT)\textsubscript{u}] 
for every $P, P' \in \mathcal{P}$ and $\gamma \in \Gamma$, $\xi(P) \neq \xi(\gamma P'\gamma^{-1})$. 
Moreover,
for every $\ubar\upsilon,\bar\upsilon>0$, 
there exists $\delta_0 > 0$ 
such that for all $P, P' \in \mathcal{P}$ and $g, h \in \Gamma$
such that there exists a bi-infinite $(\ubar\upsilon,\bar\upsilon)$-metric quasigeodesic path $\eta gh \eta'$ where 
$\eta'$ is in $P'$ and $\eta$ is in $P$, 
we have 
$\sin \angle (g^{-1} \xi(P), h\, \xi^*(P')) > \delta_0$.
\end{itemize}
\end{defn}

We remark that the unique limits condition corresponds to the ``tied-up horoballs'' condition in \cite{KL}, and the quadratic gaps condition is analogous to the uniform gap summation property that appears in \cite{GGKW}.

\spacer \begin{defn} \label{defn:reldomrep}
Fix $\Gamma$ and $\mathcal{P}$ as above, with $\mathcal{P}$ satisfying (RH), and fix constants $\ubar{C}, \ubar{\mu} > 0$. 
A representation $\rho: \Gamma \to \GL(d,\real)$ is {\bf $1$-almost dominated relative to $\mathcal{P}$} with lower domination constants $(\ubar{C},\ubar\mu)$, if it satisfies \begin{itemize}
\item[(D\textsuperscript{-})] for all $\gamma \in \Gamma$, $\frac{\sigma_1}{\sigma_{2}}(\rho(\gamma)) \geq \ubar{C} e^{\ubar\mu|\gamma|_c}$.
\end{itemize}

A $1$-almost dominated representation $\rho$ is {\bf $1$-dominated relative to $\mathcal{P}$} with lower domination constants $(\ubar{C},\ubar\mu)$ if in addition 
the images of peripheral subgroups under $\rho$ are well-behaved.
\end{defn}

Below we will sometimes refer to (D\textsuperscript{-}) as the lower domination inequality.
We will sometimes 
suppress $\mathcal{P}$ and refer to 1-relatively dominated representations. 

We further remark that many of the conditions in Definition \ref{defn:peri_conds} can be weakened or omitted if we assume relative hyperbolicity of the source group, together with the existence and transversality of limit maps: see Theorem \ref{thm:rRCA_reldom}, and associated definitions in that section, for a precise statement. We conjecture that it may further be possible that the uniform transversality hypothesis in Definition \ref{defn:peri_conds} can be made to follow from relative hyperbolicity and (D\textsuperscript{-}) as well.

\subsection{Dual representations} \label{sub:dual_rep}

Given $\rho: \Gamma \to \GL(V)$ with $V = \real^d$ as above (and the implicit choice of the standard basis, which fixes an identification $V \cong V^*$), we may define the dual representation $\rho^*: \Gamma \to \GL(V^*) \cong \GL(V)$ by $\rho^*(\gamma) = \rho(\gamma^{-1})^T$.

The following observations will be useful later:

\spacer \begin{prop} \label{prop:dualrep_reldom}
If $\rho: \Gamma \to \GL(V)$ is $1$-dominated relative to $\mathcal{P}$ with lower domination constants $(\ubar{C},\ubar\mu)$, then so is $\rho^*: \Gamma \to \GL(V)$.

Furthermore, for all $\gamma \in \Gamma$, $U_1(\rho^*(\gamma)) = (U_{d-1}(\rho(\gamma)))^\perp$ and $U_{d-1}(\rho^*(\gamma)) = (U_{1}(\rho(\gamma)))^\perp$.
\begin{proof}
We have (D\textsuperscript{-}) since 
$\frac{\sigma_1}{\sigma_2}(\rho^*(\gamma)) = \frac{\sigma_1}{\sigma_2}(\rho(\gamma^{-1})) \geq \ubar{C} e^{-\ubar\mu |\gamma^{-1}|_c} = \ubar{C} e^{-\ubar\mu |\gamma|_c}$.

We can similarly get the quadratic gaps condition, since $\frac{\sigma_1}{\sigma_2}(\rho^*(\gamma\eta)) = \frac{\sigma_1}{\sigma_2}(\rho(\eta^{-1}\gamma^{-1})$ and $\frac{\sigma_1}{\sigma_2}(\rho^*(\eta\gamma)) = \frac{\sigma_1}{\sigma_2}(\rho(\gamma^{-1}\eta^{-1})$

Now if write the singular value decomposition $\rho(\gamma) = KAL$,
then $\rho^*(\gamma) = (K^{-1})^T (A^{-1})^T (L^{-1})^T = KA^{-1}L$. 

Recalling $A$ has diagonal entries in non-increasing order, $A^{-1}$ has diagonal entries in non-decreasing order; hence $U_1(\rho^*(\gamma))$ is the line spanned by the last column of $K$, 
which is $(U_{d-1}(\rho(\gamma)))^\perp$.
Similarly, $U_{d-1}(\rho^*(\gamma))$ is the hyperplane spanned by the all but the first column of $K$;
this is $(U_1(\rho(\gamma)))^\perp$.

Now the unique limits condition for $\rho^*$ follows from the unique limits condition for $\rho$, since 
\[ \lim_{n\to\infty} U_1(\rho^*(\eta_n)) = \lim_{n\to\infty} (U_{d-1}(\rho(\eta_n)))^\perp = \xi^*_\rho(P)^\perp \] 
% (notationally this is slightly annoying, although actually if you remember that there's an identification involved between $V^*$ and $V$, and what that identification is, it actually makes perfect sense and there's (after passing through the identification) no ambiguity about what $\xi^*$ is.)
and similarly
\[ \lim_{n\to\infty} U_{d-1}(\rho^*(\eta_n)) = \lim_{n\to\infty} (U_1(\rho(\eta_n)))^\perp = \xi_\rho(P)^\perp \] 

Similarly, the uniform transversality condition for $\rho^*$ follows from the uniform transversality condition for $\rho$, due to the above identifications.
\end{proof}
\end{prop}

\subsection{Discreteness, faithfulness, proximal elements} \label{sub:df_prox}

Discreteness and faithfulness are straightforward consequences of the singular value gap growing coarsely with cusped word-length:
% together with the continuity of the function $\GL(d,\real) \to \real^d$ given by taking a matrix to its singular values: 
\spacer \begin{prop} \label{prop:discrete_faithful}
If $\rho: \Gamma \to \GL(d,\real)$ is 1-almost relatively dominated, then $\rho$ is discrete and faithful.
\begin{proof}
Given any sequence of distinct elements$(\gamma_n) \subset \Gamma$, we must have $|\gamma_n|_c \to \infty$ since there are finitely many group elements $\gamma$ satisfying $|\gamma|_c \leq N$ for each $N$. 

(D\textsuperscript{-}) then gives $\log \frac{\sigma_1}{\sigma_2}(\rho(\gamma_n)) \geq \log \ubar{C} + \ubar\mu |\gamma_n|_c \to \infty$ for a $(1,\ubar{C},\ubar\mu)$-relatively almost dominated representation. Hence we cannot have $\rho(\gamma_n) \to \id$, which proves that $\rho$ is discrete and has finite kernel. Since by assumption $\Gamma$ is torsion-free, we may further conclude that $\rho$ is faithful.
\end{proof}
\end{prop}

Using in addition the property that our peripheral subgroups $\mathcal{P}$ satisfy (RH)---or, in particular, non-distortion---, we further obtain
\spacer \begin{prop} \label{lem:nonperi_proximal}
Suppose $\rho: \Gamma \to \GL(d,\real)$ is 1-almost relatively dominated. For any non-peripheral $\gamma \in \Gamma$, $\rho(\gamma)$ must be proximal.

\begin{proof}
Recall the relation between the eigenvalues and singular values given by 
\[ \log |\lambda_i (\rho(\gamma))| = \lim_{n \to \infty} \frac 1n \log \sigma_i(\rho(\gamma^n)) \]
(see e.g. \cite{Benoist1997}, \S2.5.) Suppose $\rho: \Gamma \to G$ is $(1,\ubar{C},\ubar\mu)$-almost relatively dominated.

Non-distortion implies there exists $\nu > 0$ such that $|\gamma^n|_c \geq \nu n$ for any non-peripheral $\gamma$, and (D\textsuperscript{-}) then implies $\log \frac{\sigma_1}{\sigma_2}(\gamma^n) \geq \ubar\mu \nu n + \log \ubar{C}$; hence we obtain 
\[ \log \left| \frac{\lambda_1}{\lambda_2} \right| (\rho(\gamma)) = \lim_{n \to \infty} \frac 1n \log \frac{\sigma_1}{\sigma_2}(\gamma^n) \geq \ubar\mu \nu > 0 .\]

Hence $\rho(\gamma)$ is proximal, as desired.
\end{proof} 
\end{prop}

\subsection{Relative quasi-isometric embedding} \label{sub:rel_qi_embed}

We can extend the upper domination hypothesis on the peripherals to a more general upper domination inequality (D\textsuperscript{+}). Using the upper and lower domination inequalities (D$^\pm$), we can then demonstrate that orbit maps are quasi-isometric embeddings of the relative Cayley graph, that is the Cayley graph with the extrinsic metric from the cusped space.

\spacer \begin{prop} \label{prop:1D+}
Suppose $\rho: \Gamma \to \GL(d,\real)$ is $1$-dominated relative to $\mathcal{P}$ with lower domination constants $(\ubar{C},\ubar\mu)$. Then there exists $\bar{C} > 1$ and $\bar\mu \geq \ubar\mu$ such that for all $\gamma \in \Gamma$, 
\[ \sigma_1(\rho(\gamma)) \leq  \bar{C}^{\frac12} e^{\frac12 \bar\mu |\gamma|_c} .\]
\end{prop}

Since $\frac{\sigma_1}{\sigma_d}(\rho(\gamma)) = \sigma_1(\rho(\gamma)) \cdot \sigma_1(\rho(\gamma^{-1}))$, this immediately yields
\spacer 
\begin{cor}[D\textsuperscript{+}] \label{cor:D+}
For $\rho: \Gamma \to \GL(d,\real)$ a 1-relatively dominated representation, let $\bar{C}$, and $\bar\mu$ be as in Proposition \ref{prop:1D+}. We have 
\[ \frac{\sigma_1}{\sigma_d}(\rho(\gamma)) \leq \bar{C} e^{\bar\mu |\gamma|_c} \]
for all $\gamma \in \Gamma$.
\end{cor}

We will sometimes refer to (D\textsuperscript{+}) as the upper domination inequality. Below, we will speak of relatively dominated representations with domination constants $(\ubar{C},\ubar\mu,\bar{C},\bar\mu)$.

\begin{proof}[Proof of Proposition \ref{prop:1D+}]
We already know the related but weaker inequality $\sigma_1(\rho(\gamma)) \leq e^{\mu_2 |\gamma|}$ from $\Gamma$ being finitely-generated, where we may take $e^{\mu_2} = \max_{s\in S} \|s\|$ where $S$ the finite generating set we used to build our cusped space.

More generally, given a word $\gamma$, we consider it as a relative path $(\gamma, H)$ (see \S\ref{sub:hat_unhat}) where $H = H_1 \coprod \dots \coprod H_n$, and suppose $\eta = (\eta_1, \dots, \eta_n)$ where $\eta_i = \gamma|_{H_i}$ are the maximal peripheral excursions.
Then we have 
\begin{align*}
\|\rho(\gamma)\| & \leq \|\rho(\gamma \setminus \eta)\| \cdot \prod_{i=1}^n \|\rho(\eta_i)\| \\
 & \leq e^{\mu_2 \cdot \ell(\gamma \setminus \eta)} \cdot C_1^n e^{\mu_1 \sum_{i=1}^n|\eta_i|_c} \\
 & \leq C_1^{|\gamma|_c} e^{\max\{\mu_2,\mu_1 \} \cdot |\gamma|_c}
\end{align*}
where $\|\rho(\gamma \setminus \eta)\|$ is to be interpreted as a product of $\|\rho(\gamma_i))\|$, where each $\gamma_i$ is a maximal connected component of $\gamma \setminus \eta$ as a path; 
% $\# (\gamma \setminus \eta)$ denotes the number of such components, and 
$\ell(\gamma \setminus \eta)$ is the (sum of) length(s) of these paths (see \S\ref{sub:hat_unhat}.) $C_1$ and $\mu_1$ here are the constants from the upper domination condition in Definition \ref{defn:peri_conds}.

Here the second inequality follows from the first paragraph of the proof for individual non-peripheral pieces, and the upper domination hypothesis in Definition \ref{defn:peri_conds} for peripheral pieces, together with the equality (\ref{eq:dc_geod}) (from the proof of Proposition \ref{prop:unhat_distance}.)

In particular, writing $\bar{C}^{\frac12} = C_1$ and $\frac12 \bar\mu = \max\{\mu_2,\mu_1\}$, we have the Proposition.
\end{proof}

\spacer \begin{prop} \label{prop:qiembed_rel}
Let $\rho: \Gamma \to \SL(d,\real)$ be a representation which is $1$-dominated relative to $\mathcal{P}$ with lower domination constants $(\ubar{C},\ubar\mu)$.

Then the orbit maps $\gamma \mapsto \rho(\gamma) \cdot o$  are equivariant quasi-isometric embeddings of the relative Cayley graph $\Cay(\Gamma, S) \subset X(\Gamma, \mathcal{P}, S)$ into the symmetric space $G/K = \SL(d,\real) / \SO(d)$.

\begin{proof}
By construction, the orbit map is equivariant, i.e. $\rho(\gamma_2\gamma_1) \cdot o = \rho(\gamma_2) \cdot (\rho(\gamma_1) \cdot o)$.

Viewing $G/K$ as a space of inner products on $\real^d$, we recall the distance formula at the end of \S\ref{sec:SVD}:
\[ d_{G/K}(o, g\cdot o) = \sqrt{\sum (\log \sigma_i(g))^2 } \]
for any $g \in \SL(d,\real)$, where the $o$ denotes the basepoint corresponding to our choice of inner product (see the beginning of this section.)

Now Proposition \ref{prop:1D+} implies $\left( \log \sigma_i(\rho(\gamma)) \right)^2 \leq \left( \log \sigma_1(\rho(\gamma)) \right)^2 \leq  \frac14 \left( \log \bar{C} + \bar\mu |\gamma|_c \right)^2$ for $1 \leq i \leq d$, and so \[ \sqrt{\sum_{i=1}^d (\log \sigma_i(\rho(\gamma)))^2} \leq \frac{\sqrt{d}}2 \left( \log \bar{C} + \bar\mu |\gamma|_c \right) \]
for all $\gamma \in \Gamma$.
On the other hand, we have
\begin{align*} 
\sqrt{ \sum_{i=1}^d (\log \sigma_i(\rho(\gamma)))^2} & \geq 
\frac12 \left( |\log \sigma_1(\rho(\gamma))| + |\log \sigma_d(\rho(\gamma))| \right) \\
 & \geq \frac12 \log \frac{\sigma_1}{\sigma_2}(\rho(\gamma)) \geq \frac12 \log\ubar{C} + \frac{\ubar\mu}2 |\gamma|_c .
\end{align*}

Combining the two immediately yields that
the orbit map into $G/K$ is a quasi-isometric embedding {\it with respect to the cusped metric}.
\end{proof} \end{prop}

\section{Existence and transversality of limits} \label{sec:limtrans}

For the rest of this paper, let $\Gamma$ be a finitely generated group, $\mathcal{P}$ be a finite collection of subgroups of $\Gamma$ satisfying (RH), and $S = S^{-1}$ be a compatible finite generating set. For the next three sections (\S\S5, 6, and 7), fix $\rho: \Gamma \to \GL(d,\real)$ a representation which is $1$-dominated relative to $\mathcal{P}$ with domination constants $(\ubar{C},\ubar\mu,\bar{C},\bar\mu)$.

The goal of this section is to establish the following existence and transversality result, which will be very useful in the following sections:
\spacer \begin{defn} \label{defn:x_gamma}
Let $\alpha: I \to \Cay(\Gamma)$ be a path with $\alpha(I \cap \ints) \subset \Gamma$.

We define the sequence
\begin{align*} 
x_\alpha  & = ( \dots A_{a-1}, \dots, A_{-1}, A_0, \dots, A_{b-1}, \dots) \\
 & := \resizebox{.96\hsize}{!}{$([\cdots], \rho(\alpha(a)^{-1} \alpha(a-1)), \dots, \rho(\alpha(0)^{-1} \alpha(-1)), \rho(\alpha(1)^{-1} \alpha(0)), \dots, \rho(\alpha(b)^{-1} \alpha(b-1)), [\cdots] ) $}
 \end{align*}
and call this the {\bf matrix sequence associated to $\alpha$}. 

We say that $\alpha$ (or $x_\alpha$) is {\bf based at $\id$} if $I \ni 0$ and $\alpha(0) = \id$.
\end{defn}

\spacer \begin{prop}
% [cf. \cite{BPS}, Theorem 2.2 and Proposition 2.4] 
\label{prop:limits_exist_trans}
Let $\gamma = \pi \circ \hat\gamma$ be a bi-infinite $(\ubar\upsilon, \bar\upsilon)$-metric quasigeodesic path $\gamma$ based at $\id$, 
and let $x = x_{\gamma} = (A_k)_{k\in\ints}$ be the matrix sequence associated to $\gamma$.
Then \begin{enumerate}[(i)]
\item the following limits 
\begin{align*}
E^{u}(x) & := \lim_{n\to\infty} U_1(A_{-1} \cdots A_{-n}) %\label{eqn:Ecu} 
\\
E^{s}(x) & := \lim_{n\to\infty} S_{d-1}(A_{n-1} \cdots A_0) %\label{eqn:Ecs} 
\end{align*}
exist and form a 
% continuous and equivariant
splitting $E^{u}(x) \oplus E^{s}(x)$ of $\real^d$, and

\item there is a uniform bound $s_{\min}$ (depending only on the quasigeodesic and domination constants) on the minimal separation $s(E^{u}(x), E^{s}(x)) := \sin \angle (E^{u}(x),E^{s}(x))$ between these linear subspaces.

\end{enumerate}
\end{prop}

To prove this we will use the following theorem, which is a mild modification of a recent result of Quas--Thieullen--Zarrabi \cite{QTZ}, which in turn is a vast generalization of the characterization of linear cocycles with dominated splittings given in Bochi--Gourmelon \cite{BG}:
\spacer \begin{thm}
\label{thm:QTZ}
Let $(A_k)_{k\in\ints} \subset \GL(d,\real)$ be a sequence of matrices such that 
there exists constants $C \geq 1$ and $\mu, \mu' \geq 0$, with $\frac1\mu \log 3C > 1$, such that the following axioms are satisfied:
\begin{itemize}
\item (SVG-BG) 
for all $k \in \ints$ and all $n \geq 0$, 
\begin{align*}
\frac{\sigma_2}{\sigma_1} (A_{k+n-1} \cdots A_k) & \leq C e^{-n\mu} 
\end{align*}

\item (EC) 
for all $k \in \ints$ and all $n \geq 0$, 
\begin{align*}
d(S_{d-1}(A_{k+n-1} \cdots A_k), S_{d-1}(A_{k+n} \cdots A_k)) & \leq C e^{-n\mu} ,\\
d(U_1(A_{k-1} \cdots A_{k-n}), U_1(A_{k-1} \cdots A_{k-(n+1)})) & \leq C e^{-n\mu} .
\end{align*}

\item (FI)\textsubscript{back}: 
for all $k \leq 0$ and $n, m \geq 0$
\begin{align*}
\frac{\sigma_1(A_{k+n-1}  \cdots A_{k-m})}{\sigma_1(A_{k+n-1} \cdots A_k) \cdot \sigma_1(A_{k-1} \cdots A_{k-m})} & \geq C^{-1} e^{-m \mu'}
\end{align*}
\end{itemize}

Then \begin{enumerate}[(i)]
\item for each $k \in \ints$ in the sequence we have a splitting $E^{u} \oplus E^{s}$ of $\real^d$ given by
\begin{align*}
E^{u}(k) & := \lim_{n\to\infty} U_1(A_{k-1} \cdots A_{k-n}) %\label{eqn:Ecu} 
\\
E^{s}(k) & := \lim_{n\to\infty} S_{d-1}(A_{k+n-1} \cdots A_k) %\label{eqn:Ecs}
\end{align*}
which is equivariant in the sense that $A_k E^*(k) = E^*(k+1)$ for all $k \in \ints$ and $* \in \{u,s\}$;
\item moreover, for all $k \leq 0$, we have a uniform lower bound $s_{\min} = s_{\min}(C,\mu,\mu')$ on the gap $s(E^{u}(k), E^{s}(k)) := \sin \angle (E^{u}(k),E^{s}(k))$ given by
\[ s(E^{u}(k), E^{s}(k)) \geq s_{\min} := \frac23 (3e)^{-2r} \exp\left( -\frac{3/2}{1-e^{-\mu}} \right) C^{-(1+2r)} ,\]
where $r := \frac{\mu'}\mu$.

\end{enumerate}
\end{thm}

We will defer the proof of this result to Appendix \ref{app:QTZ} and focus on showing how to obtain Proposition \ref{prop:limits_exist_trans} given the Theorem. We remark that we may assume, without loss of generality, that our constants are such that the additional hypothesis $\frac1\mu \log 3C > 1$ specified in Theorem \ref{thm:QTZ} is satisfied; if they are not, we can make $C$ larger or $\mu$ smaller and the other required axioms will continue to hold with these adjusted constants.

Before beginning the argument, we remark that a number of linear algebra results, which will be used throughout this and subsequent proofs, are collected in Appendix \ref{app:linalg}. We note that Lemma \ref{lem:BPSA4A5}, in particular, will be used many times below to control unstable spaces of products of matrices. 

We start by establishing the following
\spacer \begin{lem} \label{lem:EC}
Given $\ubar\upsilon, \bar\upsilon > 0$, there exist constants $C \geq 1$ and $\mu >0$, depending only on the representation and $\ubar\upsilon, \bar\upsilon$, such that for any matrix sequence $x = x_{\gamma}$ associated to a 
bi-infinite $(\ubar\upsilon,\bar\upsilon)$-metric quasigeodesic path $\gamma$ based at $\id$,
\begin{align*}
d(U_1(A_{k-1} \cdots A_{k-n}), U_1(A_{k-1} \cdots A_{k-(n+1)})) & \leq C e^{-n\mu} \\
d(S_{d-1}(A_{k+n-1} \cdots A_k), S_{d-1}(A_{k+n} \cdots A_k)) & \leq C e^{-n\mu} .
\end{align*}
\end{lem}
In other words, such sequences $x_{\gamma}$ satisfy (EC), with constants depending only on the representation and the quasigeodesic constants. It then follows, using the triangle inequality, that the limits exist, and in fact convergence to the limits is uniform:
\spacer \begin{cor} \label{cor:EC}
Given $x = x_{\gamma} = (A_k)$ a matrix sequence associated to 
bi-infinite $(\ubar\upsilon,\bar\upsilon)$-metric quasigeodesic path $\gamma$ based at $\id$, the limits 
\[ E^u(x) := \lim_{n\to\infty} U_1(A_{-1} \cdots A_{-n})
\quad\mbox{ and }\quad 
E^s(x) := \lim_{n\to\infty} S_{d-1}(A_{n-1} \cdots A_0)\] 
exist, and
\begin{align*}
d(U_1(A_{-1} \cdots A_{-n}), E^{u}(x)) & \leq 
\frac{C}{1-e^{-\mu}} \cdot e^{-n\mu} \\
d(E^{s}(x), S_{d-1}(A_{n} \cdots A_0)) & \leq 
\frac{C}{1-e^{-\mu}} \cdot e^{-n\mu}
\end{align*}
where $C, \mu$ are the constants from Lemma \ref{lem:EC}.
\end{cor}

To prove Lemma \ref{lem:EC} it will be useful to 
% establish some language and notation allowing us to 
more closely examine the parts of matrix sequences inside the peripheral subgroups. For this purpose, we recall the notions of peripheral excursion and depth from \S\ref{sec:relhyp}, now used for matrix sequences coming from paths in $\Gamma$:
\spacer \begin{defn} \label{defn:matseq_periexcur}
Given $I$ an interval in $\ints$ and a sequence $x = x_\alpha = (A_k) \in \GL(d,\real)^I$ associated to some path $\gamma: I \to 
\Cay(\Gamma)$, a {\bf peripheral excursion} in $x$ is a subsequence $(A_k) \in \GL(d,\real)^J$ where $\gamma|_J$ is a peripheral excursion in the sense of \S\ref{sub:hat_unhat}.

The {\bf depth} of a matrix $A_k = \rho(\gamma(k)^{-1} \gamma(k-1))$ inside a peripheral excursion is the depth of $\gamma(k)^{-1} \gamma(k-1)$ in the sense of Definition \ref{defn:projgeod_depth}.
\end{defn}

\begin{proof}[Proof of Lemma \ref{lem:EC}]
We presently restrict our attention to $(A_{k-n})_{n>0}$, in order to study more carefully the limit giving $E^{u}(k)$. 

We now derive two inequalities, each of which works to give us the bound we want in a different case.
On the one hand, we have
\begin{align*}
& d(U_1(A_{k-1} \cdots A_{k-n}), U_1(A_{k-1} \cdots A_{k-n-1})) \\
 & \quad\leq \frac{\sigma_1}{\sigma_d}(A_{k-n-1}) \cdot \frac{\sigma_2}{\sigma_1}(A_{k-1} \cdots A_{k-n}) \\
 & \quad \leq \frac{\sigma_1}{\sigma_d}(\rho(\gamma(k-n)^{-1} \gamma(k-n-1))) \cdot \frac{\sigma_2}{\sigma_1}(\rho(\gamma(k)^{-1} \gamma(k-n))) 
\end{align*}
by Lemma \ref{lem:BPSA4A5}. By 
Corollary \ref{cor:D+} and Definition \ref{defn:metric_proj_qgeod},
\[ \frac{\sigma_1}{\sigma_d}(\rho(\gamma(k-n)^{-1} \gamma(k-n-1))) \leq 
e^{\bar\mu \cdot \bar\upsilon (\delta(A_{k-n-1}) + 6)} = \bar{C} e^{6 \bar\mu \bar\upsilon} e^{\bar\mu \bar\upsilon \cdot \delta(A_{k-n-1})} ;\]
by Definition \ref{defn:metric_proj_qgeod} and the lower domination inequality (D\textsuperscript{-}),
\[ \frac{\sigma_2}{\sigma_1}(\rho(\gamma(k)^{-1} \gamma(k-n))) \leq \ubar{C}^{-1} e^{-\ubar\mu \ubar\upsilon} e^{\ubar\mu \ubar\upsilon n} \] 
where $\ubar{C}$ and $\ubar\mu$ are the domination constants.
Hence, writing $C_2 = \bar{C} \ubar{C}^{-1} e^{6 \bar\mu \bar\upsilon + \ubar\mu \ubar\upsilon}$, $\mu_2 = \bar\mu \bar\upsilon$, and $\mu_0 = \ubar\mu \ubar\upsilon$,
\begin{align}
d(U_1(A_{k-1} \cdots A_{k-n}), U_1(A_{k-1} \cdots A_{k-n-1}))&  \leq C_2 e^{ \mu_2 \cdot \delta(A_{k-n-1})} \cdot e^{-\mu_0 n} \label{ineq:lodepth}
\end{align}
This will turn out to give us the inequality we want when the depth $\delta(A_{k-n-1})$ is relatively small compared to $n$.

Alternatively, suppose a matrix lies in a peripheral excursion starting at $k-n_0$. Write $D := A_{k-1} \cdots A_{k-n_0}$ to denote the word prior to the excursion, and, for any integer $n$ with $A_{k-n}$ belonging to the peripheral excursion, $E(n-n_0) := A_{k-n_0-1} \cdots A_{k-n}$, so that we have the decomposition $A_{k-1} \cdots A_{k-n} = DE(n-n_0)$. 

We break $E(n-n_0)^{-1} E(n+1-n_0) = A_{k-n-1}$ up into smaller chunks
\[ A_{k-n-1} = A_{k-n-1,1} \cdots A_{k-n-1,\ell(k-n-1)} = \rho\left( p_{k-n-1,1} \cdots p_{k-n-1,\ell(k-n-1)} \right) \]
corresponding to single unbunched peripheral generators (as in property (iii) of Definition \ref{defn:metric_proj_qgeod}.)

For brevity, we write $F_j := A_{k-n-1,j}$ in the next inequality, and also adopt the convention $F_0 = \id$.
Now we have
\begin{align*}
& d(U_1(DE(n-n_0)), U_1(DE(n+1-n_0))) \\
 & \quad \leq \sum_{j=1}^{\ell(k-n-1)} d(U_1(DEF_0 \cdots F_{j-1}), U_1(DE(n-n_0) F_0 \cdots F_j) )  \\
 & \quad \leq \sum_{j=1}^{\ell(k-n-1)} \frac{\sigma_1}{\sigma_d}(F_j) \cdot \frac{\sigma_2}{\sigma_1}(DE(n-n_0) F_0 \cdots F_{j-1})  \\
 & \quad \leq \bar{C} e^{\bar\mu} \sum_{j=1}^{\ell(k-n-1)} \frac{\sigma_2}{\sigma_1} (DE(n-n_0) F_0 \cdots F_{j-1}) =: RHS_1
\end{align*}
where we have used the triangle inequality $\ell(k-n-1)$ times, applied Lemma \ref{lem:BPSA4A5} to each of the resulting terms, and then used Corollary \ref{cor:D+} with the bound on the size of single generators; then, using the quadratic gaps condition (which bounds from below the first singular value gap for images of words ending in peripheral excursions) 
\begin{align}
& d(U_1(DE(n-n_0)), U_1(DE(n+1-n_0))) \leq RHS_1 \notag \\ & \quad \leq \sum_{j=1}^{\ell(k-n-1)} \frac{\sigma_1}{\sigma_d}(F_j) \cdot \frac{\sigma_2}{\sigma_1}(DE(n-n_0) F_0 \cdots F_{j-1}) \notag \\
 & \quad \leq \bar{C} e^{\bar\mu} \sum_{j=1}^{\ell(k-n-1)} \frac{\sigma_2}{\sigma_1} (\rho(\gamma(k)^{-1} \gamma(k-n_0) \cdot 
 \gamma(k-n_0)^{-1} \gamma(k-n) p_{k-n-1,1} \cdots p_{k-n-1,j}))
 \notag \\
 & \quad \leq \bar{C} e^{\bar\mu}  \cdot \frac1{C'} \sum_{j=0}^{\ell(k-n-1)} |\gamma(k-n_0)^{-1} \gamma(k-n) p_{k-n-1,1} \cdots p_{k-n-1,j})|^{-2} =: RHS_2 \notag 
\end{align}
and finally using the metric quasigeodesic lower bound and Proposition \ref{prop:wordlengths_bilip_log}, we obtain
\begin{align}
& d(U_1(DE(n-n_0)), U_1(DE(n+1-n_0))) \leq RHS_2 \notag \\
 & \quad \leq \bar{C} e^{\bar\mu} \cdot \frac1{C'} \sum_{j=0}^{\ell(k-n-1)} \left( 2^{\ubar\upsilon^{-1}(n-n_0) - \ubar\upsilon} + j \right)^{-2} \notag \\
 & \quad \leq \frac{2^{1+\ubar\upsilon} \bar{C} e^{\bar\mu}}{C'} \exp \left( -\frac{\log 2}{\ubar\upsilon} (n-n_0) \right) 
 \leq C_3 \exp \left( \frac{\log 2}{\ubar\upsilon} \cdot \delta(A_{k-n}) \right) \label{ineq:hidepth}
\end{align}
where $C_3 := \frac{2^{1+\ubar\upsilon} \bar{C} e^{\bar\mu}} {C'}$; at the end we have used the general inequality
\[ \sum_{j=0}^b (M+j)^{-2} = \sum_{j=M}^{M+b} j^{-2} \leq \int_{M-1}^{M+b} x^{-2} \,dx 
= \frac{1}{M-1} - \frac{1}{M+b} \leq \frac2M .\]
This second inequality will serve us when the depth $\delta(A_{k-n})$ is relatively large compared to $n$.

For $n > 0$ where the depth $\delta(A_{k-n-1}) \leq \frac{\mu_0}{2\mu_2} n$ (including all $n$ where $\delta(A_{k-n}) = 0$, i.e. $A_n$ is nonperipheral), it follows from (\ref{ineq:lodepth}) that
\begin{align*}
d\left( U_1(A_{k-1} \cdots A_{k-n}), U_1(A_{k-1} \cdots A_{k-(n+1)}) \right) & \leq C_2 e^{\mu_2 \left( \delta(A_{k-n-1}) \right)} \cdot C_0^{-1} e^{-\mu_0 n} \\
  & \leq C_2 e^{\mu_2 \cdot \frac{\mu_0}{2\mu_2}n} e^{-\mu_0 n} = C_2 e^{-\frac{\mu_0}2n}
\end{align*}
For $n > 0$ where the depth $\delta(A_{k-n}) > \frac{\mu_0}{2\mu_2} n$, we have,
from (\ref{ineq:hidepth}),
\begin{align*} 
d\left( U_1(A_{k-1} \cdots A_{k-n}), U_1(A_{k-1} \cdots A_{k-n-1}) \right) & \leq C_3 \exp\left( -\frac{\mu_0 \log 2}{2\ubar\upsilon \mu_2} n \right)
\end{align*}
and so {\bf we have the desired inequalities for our Lemma, with $C = \max\left\{ C_2, C_3 \right\}$
and $\mu = \frac{\mu_0}2 \cdot \min\{1, \frac{\log 2}{\ubar\upsilon \mu_2} \}$.} 

For $(A_{k+n})_{n \geq 0}$ and the limit giving $E^{s}(k)$, we may argue similarly, or alternatively we may consider the reversed dual sequence $^\iota x^* = (B_k)_{k\in\ints}$ given by 
\begin{equation}
B_k := \rho^*(\gamma(-k-1)^{-1} \gamma(-k-2)) = (A_{-k-1}^{-1})^T 
\label{eqn:rev_dual_seq}
\end{equation} 
where $\rho^*$ is the dual representation, which is also 1-relatively dominated (Proposition \ref{prop:dualrep_reldom}.)

By Proposition \ref{prop:dualrep_reldom}, we have
\begin{align*}
U_1(B_{-k-1} \cdots B_{-k-n}) & = U_1(\rho^*(\gamma_r(-k)^{-1} \gamma_r(n-k-1))) \\
 & = \left( U_{d-1}(\rho(\gamma_r(k)^{-1} \gamma_r(k+n-2))) \right)^\perp \\
 & = \left( S_{d-1}(\rho(\gamma_r(k+n-2)^{-1} \gamma_r(k) )) \right)^\perp  \\ & = \left( S_{d-1} (A_{k+n-1} \cdots A_k) \right)^\perp
\end{align*}
Then we have
\begin{align*}
d(S_{d-1}(A_{k+n-1} \cdots A_k), S_{d-1}(A_{k+n} \cdots A_0)) & = d(U_1(B_{-k} \cdots B_{-k-n}), U_1(B_{-k} \cdots B_{-k-n-1})) \\ & \leq C e^{-\mu n} .
\end{align*}
where in the last step we have used the argument above for the $E^u(-k)$ limit for $^\iota x^*$,
\end{proof}

\begin{proof}[Proof of Proposition \ref{prop:limits_exist_trans}]
By Corollary \ref{cor:EC}, the limits $E^{u}(x)$ and $E^{s}(x)$ exist, and the sequence $x = x_\gamma$ satisfies axiom (EC) in the statement of Theorem \ref{thm:QTZ}, with constants depending only on the domination and quasigeodesic constants.

From the upper and lower domination inequalities (D\textsuperscript{-}) and the metric quasigeodesic properties in Definition \ref{defn:metric_proj_qgeod}), $x = x_\gamma$ satisfies axiom (SVG-BG) in the statement of Theorem \ref{thm:QTZ}, with constants $C = \ubar{C} e^{-\ubar\mu \ubar\upsilon}$ and $\mu = \ubar\mu \ubar\upsilon$.

{\bf Step 1: bounded-depth sequences.} 

\begin{defn}
We say a sequence $x = (A_k)_{k\in\ints}$ has {\bf bounded depth $\Delta$} in the backward direction (in the forward direction, respectively) if  $\delta(A_k) \leq \Delta$ for all $k \leq 0$ (for all $k \geq 0$, resp.)
\end{defn}

Equivalently, for $x_{\gamma_r}$, our (sub)path $\gamma|_{\ints_{\leq0}}$ (or $\gamma|_{\ints_{\geq0}}$, respectively) has peripheral excursions of bounded cusped length.

\begin{prop} \label{prop:lim_trans_bdd_dep}
Given $\Delta \in \ints_{\geq 0}$, there exists $s_{\min}(\Delta)$ (which also depends on the quasigeodesic and domination constants) such that for any $x = x_{\gamma_r}$ with bounded depth $\Delta$ in the backward direction or in the forward direction, $s(E^{u}(x), E^{s}(x)) \geq s_{\min}(\Delta)$

\begin{proof}
If $x = x_{\gamma_r}$ has bounded depth $\Delta$ in the backward direction, then $x$ satisfies the axiom (FI)\textsubscript{back} from the inequalities
\begin{align*} 
\frac{\sigma_1(A_{k+n-1} \cdots A_{k-m})}{\sigma_1(A_{k+n-1} \cdots A_k) \cdot \sigma_1(A_{k-1} \cdots A_{k-m})} & \geq \frac{\sigma_1(A_{k+n-1} \cdots A_k) \cdot \sigma_d(A_{k-1} \cdots A_{k-m})}{\sigma_1(A_{k+n-1} \cdots A_k) \cdot \sigma_1(A_{k-1} \cdots A_{k-m})} \\
 & \geq \frac1{C_2} e^{-\mu_2(\Delta+m)}
= \frac{e^{-\mu_2 \Delta}}{C_2} e^{-\mu_2 m} ;
\end{align*}
these inequalities follow from the general inequalities $\sigma_1(A) \cdot \sigma_1(B) \geq \sigma_1(AB) \geq \sigma_1(A) \cdot \sigma_d(B)$ and Corollary \ref{cor:D+} and Definition \ref{defn:metric_proj_qgeod}, with $C_2,\mu_2$ as in the proof of Lemma \ref{lem:EC}. 

Thus if $x = x_{\gamma_r}$ has bounded depth $\Delta$ in the backward direction, it satisfies (FI)\textsubscript{back} with $D = C_2 e^{\mu_2 \Delta}$ and $\mu'=\mu_2$.  
In particular, Theorem \ref{thm:QTZ} gives us $s(E^u(x), E^s(x)) \geq s_{\min}(\Delta)$ for some $s_{\min}(\Delta)$ depending also on the quasigeodesic and domination constants, and we obtain the Proposition for such sequences.
%\footnote{The quasigeodesic and domination constants feed into the (SVG-BG) and (EC) constants}

If $x = x_{\gamma_r} = (A_k)_{k\in\ints}$ has bounded depth $\Delta$ in the forward direction but not the backward direction, consider again the reversed dual sequence $^\iota x^* = (B_k)_{k\in\ints}$ defined above in (\ref{eqn:rev_dual_seq}).

The sequence $^\iota x^*$ has bounded depth in the backward direction, hence Proposition \ref{prop:lim_trans_bdd_dep} we have $s(E^u(^\iota x^*), E^s(^\iota x^*)) \geq s_{\min}(\Delta)$. 

But now, by Proposition \ref{prop:dualrep_reldom}, $E^u(^\iota x^*) = E^s(x)^\perp$ since
\begin{align*}
E^u(^\iota x^*) & = \lim_{n\to\infty} U_1(B_{-1} \cdots B_{-n}) \\
 & = \lim_{n\to\infty} U_1(\rho^*(\gamma_r(0)^{-1} \gamma_r(n-2))) \\
 & = \lim_{n\to\infty} (U_{d-1}(\rho(\gamma_r(0)^{-1} \gamma_r(n-2))) )^\perp \\
 & = \left( \lim_{n\to\infty} S_{d-1}(\rho(\gamma_r(n-2)^{-1} \gamma_r(0) )) \right)^\perp = E^s(x)^\perp
\end{align*}
and similarly $E^s(^\iota x^*) = E^u(x)^\perp$. Hence we have $s(E^u(x), E^s(x)) \geq s_{\min}(\Delta)$ as desired.
\end{proof} \end{prop}

{\bf Step 2: unbounded-depth sequences.} If our sequence $x = x_{\gamma_r}$ does not have bounded depth in either the backward or forward directions, then the subpaths in both directions (i.e. both $\gamma|_{\ints_{\leq0}}$ and $\gamma|_{\ints_{\geq0}}$) contain arbitrarily long peripheral excursions. 

Define $P^\pm \in \mathcal{P}$ and infinite peripheral excursions $p^\pm_\infty$ as follows: \begin{itemize} 
\item if $\gamma$ is eventually peripheral in the forward (backward, respectively) direction, let $p^+_\infty$ ($p^-_\infty$, resp.) be the maximal infinite peripheral excursion of the form $\gamma|_{\geq N}$ for some $N \in \ints_{\geq 0}$ ($\gamma|_{\geq N}$ for some $N \in \ints_{\leq 0}$, resp.), and let $P^+$ ($P^-$, resp.) be the peripheral subgroup in which $p^+_\infty$ ($p^-_\infty$, resp.) lies.
\item If $\gamma$ is not eventually peripheral in the forward (backward, resp.) direction: by the finiteness of $|\mathcal{P}|$ and since the peripheral subgroups are finitely-generated, in this direction we can find $P^+ \in \mathcal{P}$ ($P^-$, resp.) and a sequence of increasingly longer peripheral excursions $p_n^\pm$ in $P^\pm$. By a diagonal argument these converge to an infinite peripheral excursion $p^\pm_\infty$ into $P^\pm$ (respectively.)
\end{itemize}

Let $L$ be the constant from the local-to-global condition in Definition \ref{defn:peri_preconds}
and $T_2$ be the threshold such that \[ \frac{C}{1-e^{-\mu}} e^{-\mu T_2} \leq \frac{\delta_0}8 \] where $C, \mu>0$ are the constants from Lemma \ref{lem:EC}, $\delta_0$ is the constant from the uniform transversality condition, and define 
$ T := \max\left \{ T_2, L \right\} $.

Consider, in each direction, the first peripheral excursions into $P^\pm$ of depth at least $T$ which (i.e. whose reparametrized projections) agree with $p^\pm_\infty$ up to length $T$. Take a sequence $x'$ where we replace these peripheral excursions with $p^\pm_\infty$ (resp.) By construction and by the local-to-global condition, these are uniform metric projected quasigeodesics in both directions (starting from 0.) 
From the uniform transversality condition, 
we have $s(E^{u}(x'), E^{s}(x')) \geq \delta_0$.

Next we wish to use (EC) (more precisely, Corollary \ref{cor:EC}) and the choice of $T$ to say that 
\[ d(E^{u}(x), E^{u}(x')) \leq \frac{\delta_0}4 \quad\quad\quad\quad
d(E^{s}(x), E^{s}(x')) \leq \frac{\delta_0}4 .\] 
To verify (EC) for $x'$, remark that our construction---in particular the choice of $T$---together with the local-to-global condition give us that we have geodesic rays in both directions, and hence 
(EC) still follows from Lemma \ref{lem:EC}.

Hence
$ s(E^{u}(x), E^{s}(x)) \geq \frac{\delta_0}2 >0 $
and we have a splitting.

{\bf To obtain the minimum gap:} from Proposition \ref{prop:lim_trans_bdd_dep} (i.e. step 1 above), we have a minimum gap $s(N)$ for any sequence of bounded depth $N$ in either direction; from step 2, we have a minimum gap $\delta_0/2$ for sequences of unbounded depth. Suppose $s(N) \to 0$ as $N \to \infty$. Then we may choose an infinite sequence of matrix sequences $x^{(m)}$, each associated to a (reparametrized) $(\ubar\upsilon,\bar\upsilon)$-metric projected quasigeodesic of bounded depth $d_m$, with $d_m \to \infty$, such that the gap between $E^{u}(x^{(m)})$ and $E^{s}(x^{(m)})$ is bounded above by $\frac 1m$. 

Up to subsequence, these converge to some infinite sequence $x$ which is associated to a reparametrized $(\ubar\upsilon,\bar\upsilon)$-metric projected quasigeodesic with zero gap between $E^{u}(x)$ and $E^{s}(x)$; but this is a contradiction whether $x$ has unbounded or bounded depth.

Hence, by our compactness argument, we may choose our minimum gap to be 
\[ \min \{ \delta_0/2, \inf_{N \in \nats} s(N) \} > 0 .\]
\end{proof}
\section{Relative domination implies relative hyperbolicity} \label{sec:reldom_relhyp}
Recall that $\Gamma$ is a torsion-free finitely-generated group. 
We will presently prove the following
\spacer \begin{thm} \label{thm:reldom_relhyp}
If $\rho: \Gamma \to \GL(d,\real)$ is 
$1$-dominated relative to $\mathcal{P} \neq \varnothing$ with domination constants $(\ubar{C},\ubar\mu,\bar{C},\bar\mu)$, 
and $\Gamma \neq \bigcup \mathcal{P}^\Gamma$
(i.e. $\Gamma$ contains non-peripheral elements), 
then $\Gamma$ must be hyperbolic relative to $\mathcal{P}$.
\end{thm}

We remark that the statement is still true if $\mathcal{P} = \varnothing$---that is precisely the result from \cite{BPS}.

The proof of Theorem \ref{thm:reldom_relhyp} will use the criterion for relative hyperbolicity given in Theorem \ref{thm:Gerasimov}.
To do so we will find a compact, perfect metric space on which $\Gamma$ acts as a geometrically finite convergence group, and verify that the maximal parabolic subgroups are precisely the peripheral subgroups. Below, we construct such a space $\Lambda_{rel}$, verify it has the required properties, check that the action of $\Gamma$ on the space of distinct triples $\Lambda_{rel}^{(3)}$ is properly discontinuous and the action of $\Gamma$ on the space of distinct pairs $\Lambda_{rel}^{(2)}$ is cocompact, and finally characterize the maximal parabolic subgroups.

We remark that the outline of the argument is adapted from that of \cite{BPS}, \S3. In particular, a statement describing north-south dynamics (Lemma 3.13 in \cite{BPS}, Lemma \ref{lem:ns_dyn_all} here), resulting from a quantitative transversality result (Corollary \ref{cor:lim_transverse}), is a key intermediate proposition. Here the geodesics we consider are located not in the group but in the associated cusped space, and this necessitates the new tools introduced in the previous section
% , beyond the dominated splittings and multiplicative ergodic theorem used in \cite{BPS}, 
for the proof of the transversality result. There are also differences in the proofs due to the convergence action of the group being geometrically-finite rather than uniform; among other things, this, through our assumption that $\Gamma$ contains both peripheral and non-peripheral elements, simplifies the proof of perfectness (Proposition \ref{prop:Mrel_perfect}.)

We fix some notation for the below. Fix $\ell_0 \in \nats$ such that $\ubar{C} e^{-\ubar\mu \ell_0} < 1$.
We will write, for brevity, $\Xi_\rho(\gamma) := U_1(\rho(\gamma))$ and $\Xi^*_\rho(\gamma) := S_{d-1}(\rho(\gamma)^{-1}) = U_{d-1}(\rho(\gamma))$, for $\gamma \in \Gamma$. We recall that these were defined in \S\ref{sec:SVD}. Given $\xi, \zeta \in \proj(\real^d)$ or $\Gr_{d-1}(\real^d)$, $d(\xi, \zeta)$ will denote distance between $\xi$ and $\zeta$ in the relevant Grassmannian.

\subsection{The limit set}
We will construct a candidate space $\Lambda_{rel}$ for the compact metric space $M$ required in Theorem \ref{thm:Yaman}, as follows:
\[ \Lambda_{rel} := \bigcap_{n \geq \ell_0} \overline{ \left\{ \Xi_\rho(\gamma) : |\gamma|_c \geq n \right\} } .\]
We remark that any $\xi \in \Lambda_{rel}$ can be written as a limit $\displaystyle \lim_{n \to \infty} \Xi_\rho(\gamma_n)$ where $|\gamma_n|_c \to \infty$.

\spacer \begin{rmk}
$\Lambda_{rel}$ is closely related to
% (essentially is) 
Benoist's limit set from \cite{Benoist1997}: 
at least in the case where $\rho(\Gamma)$ is Zariski-dense, $\Lambda_{rel}$ is the natural projection of Benoist's limit set to the projective space.
\end{rmk}

It is fairly immediate that
\spacer \begin{prop} \label{prop:BPS311}
$\Lambda_{rel}$ is compact, non-empty, and $\rho(\Gamma)$-invariant.

\begin{proof}
$\Lambda_{rel}$ is compact and non-empty since it is a decreasing intersection of non-empty closed subsets of a Grassmannian, which is a compact space. 

To show $\Lambda_{rel}$ is $\rho(\Gamma)$-invariant, we fix $\eta \in \Gamma$ and $\xi \in \Lambda_{rel}$, and choose a sequence $(\gamma_n) \subset \Gamma$ such that $|\gamma_n|_c \to \infty$ and $\Xi(\gamma_n) \to \xi$. $\Xi(\eta\gamma_n)$ is well-defined whenever $|\gamma_n| \geq \ell_0 - |\eta|$, and by (D\textsuperscript{-}) and Lemma \ref{lem:BPSA4A5}(\ref{eqn:A5}) we have
\[ d(\rho(\eta) \, \Xi_\rho(\gamma_n), \Xi_\rho(\eta\gamma_n)) \leq \frac{\sigma_1}{\sigma_d}(\rho(\eta)) \cdot \ubar{C} e^{-\ubar\mu |\gamma_n|_c} \to 0 \]
as $n \to \infty$, and so $\Xi_\rho(\eta\gamma_n) \to \rho(\eta) \xi$ as $n \to \infty$, and in particular $\rho(\eta) \xi \in \Lambda_{rel}$.
\end{proof}
\end{prop}

\subsection{Dynamics on the limit set} \label{sub:dyn_bdy}

Recall that $\rho: \Gamma \to \GL(d,\real)$ is a $1$-relatively dominated representation with domination constants $(\ubar{C},\ubar\mu,\bar{C},\bar\mu)$.

We start this section with the following comparability lemma, which follows from Corollary \ref{cor:D+} and related estimates:
% {\footnotesize (In Sambarino's description, ``the only place where the representation and group interact with each other.'')}
\spacer \begin{lem}
% [cf. \cite{BPS}, Lemma 3.9] 
\label{lem:wordsumcomp}
There exist constants $\nu \in (0,1)$, $c_0 > 1$ and $c_1 > 1$, depending only on the domination constants $\ubar{C}, \ubar\mu>0$, such that 
for any $\gamma, \eta \in \Gamma$ satisfying
$|\gamma|_c, |\eta|_c \geq \ell_0$ (with $\ell_0$ as above), then
\[ d_c(\gamma,\eta) \geq \nu(|\gamma|_c + |\eta|_c) - c_0 - c_1|\log d(\Xi_\rho(\gamma), \Xi_\rho(\eta))| .\]
\begin{proof}
Consider $\gamma, \eta \in \Gamma$ with cusped word length at least $\ell_0$. Assume without loss of generality that $|\gamma|_c \leq |\eta|_c$. Applying Lemma \ref{lem:BPSA4A5}(\ref{eqn:A4}) to $A = \rho(\eta)$ and $B = \rho(\eta^{-1}\gamma)$, and using the relatively dominated condition and Corollary \ref{cor:D+}, we obtain
\begin{align*}
d\left( \Xi_\rho(\eta), \Xi_\rho(\gamma) \right) & \leq \frac{\sigma_1}{\sigma_d} \left( \rho(\eta^{-1}\gamma) \right) \cdot \frac{\sigma_2}{\sigma_1}(\rho(\eta)) \\
 & \leq \bar{C} e^{\bar\mu |\eta^{-1}\gamma|_c} \cdot \ubar{C} e^{-\ubar\mu |\eta|_c} 
\end{align*} 
where $\bar{C}, \bar\mu$ are the constants from Corollary \ref{cor:D+}. Equivalently, after taking logarithms and isolating the $d_c(\gamma,\eta)$ term,
\begin{align*}
d_c(\gamma, \eta) = |\eta^{-1}\gamma|_c & \geq {\bar\mu}^{-1} \left( \ubar\mu|\eta|_c - \log \bar{C} - \log \ubar{C} + \log d(\Xi_\rho(\eta), \Xi_\rho(\gamma)) \right) \\
& \geq \ubar\mu \bar\mu^{-1} |\eta|_c - \bar\mu^{-1} \left(\log \bar{C} + \log \ubar{C} \right) - \bar\mu^{-1} \left|\log d(\Xi_\rho(\eta), \Xi_\rho(\gamma)) \right|
\end{align*} 
and since $|\eta|_c \geq (|\gamma|_c + |\eta|_c)/2$, we obtain the lemma. \end{proof}
\end{lem}

In particular, applying this to projected geodesic rays, we obtain
\spacer \begin{lem} \label{lem:geodrays_qg}
If $(\gamma_n)_{n=0}^\infty, (\eta_n)_{n=0}^\infty$ are two projected geodesic sequences in $\Gamma$ with $\gamma_0 = \eta_0 = \id$ such that 
$\displaystyle \lim_{n\to\infty} \Xi_\rho(\gamma_n) \neq \lim_{n\to\infty} \Xi_\rho(\eta_n)$, then $(\dots, \eta_2, \eta_1, \id, \gamma_1, \gamma_2, \dots)$ is a metric quasigeodesic, with quasigeodesic constants depending only on 
\[  d \left( \lim_{n\to\infty} \Xi_\rho(\gamma_n), \lim_{n\to\infty} \Xi_\rho(\eta_n)\right) .\]
\begin{proof}  %KOALA
Given the hypotheses, it follows from Corollary \ref{cor:EC} that the limits $\displaystyle \xi_\rho(\gamma) := \lim_{n\to\infty} \Xi_\rho(\gamma_n)$, $\displaystyle \xi_\rho(\eta) := \lim_{n\to\infty} \Xi_\rho(\eta_n)$ and $\displaystyle \xi^*_\rho(\eta) := \lim_{n\to\infty} \Xi^*_\rho(\eta_n)$ exist.

The previous Lemma applied to the pairs of elements $(\gamma_n, \eta_n)$, together with Proposition \ref{prop:reparam_proj_geod}, yields that the sequence $(\dots, \eta_n, \dots \eta_0, \id, \gamma_0, \dots, \gamma_n, \dots )$ is a metric quasigeodesic path, with constants depending on 
$\eps := d(\xi_\rho(\gamma_n), \xi_\rho(\eta_n))$, $\nu \in (0,1), c_0, c_1$ from Lemma \ref{lem:wordsumcomp}, and $\ell_0$ from above.

More precisely, Proposition  \ref{prop:reparam_proj_geod} verifies the metric quasigeodesic inequalities for any subpath restricted to one side of $\id$, i.e. containing only elements $\gamma_i$ or $\eta_j$. 

For subpaths containing both some $\eta_l$ and some $\gamma_k$, we have
\begin{align*}
d_c(\gamma_k, \eta_l) \leq d_c(\gamma_k,\id) + d_c(\id,\eta_l) \leq 8(k+l)+40
\end{align*}
from the triangle inequality and Proposition \ref{prop:reparam_proj_geod}. For the lower bound here: write $$ c := \max\{2\ell_0, c_0+c_1 \log(3/\eps)\} ,$$
and note that we have
\begin{align*}
d_c(\gamma_k, \eta_l) = |\eta_l^{-1}\gamma_k|_c & \geq \nu(|\eta_l|_c+|\gamma_k|_c) - c
 \geq \frac\nu6 (l + k) - c
\end{align*}  
% \geq r\nu(i+j) - c \]
from Lemma \ref{lem:wordsumcomp} and Proposition \ref{prop:reparam_proj_geod} when both $|\gamma_k|_c,|\eta_l|_c > \ell_0$. In the case $|\eta_l|_c \leq \ell_0$ we have
\begin{align*}
d_c(\gamma_k, \eta_l) & \geq d_c(\gamma_k,\id) - d_c(\eta_l,\id) 
 \geq |\gamma_k|_c - \ell_0 \\
 & \geq (|\gamma_k|_c + |\eta_l|_c) - 2\ell_0 
\end{align*}
and an analogous argument produces the same lower bound when $|\gamma_k|_c \leq \ell_0$. 
\end{proof}
\end{lem}

We may combine this with Proposition \ref{prop:limits_exist_trans} to obtain
\spacer \begin{cor} \label{cor:lim_transverse}
If $(\gamma_n)_{n=0}^\infty, (\eta_n)_{n=0}^\infty$ are two projected geodesic sequences in $\Gamma$ with $\gamma_0 = \eta_0 = \id$ such that 
$\displaystyle \lim_{n\to\infty} \Xi_\rho(\gamma_n) \neq \lim_{n\to\infty} \Xi_\rho(\eta_n)$, then $\displaystyle \lim_{n\to\infty} \Xi_\rho(\gamma_n)$ is transverse to $\displaystyle \lim_{n\to\infty} \Xi^*_\rho(\eta_n)$.

\begin{proof}
Given the hypotheses, it follows from Corollary \ref{cor:EC} that the limits $\displaystyle \xi_\rho(\gamma) := \lim_{n\to\infty} \Xi_\rho(\gamma_n)$, $\displaystyle \xi_\rho(\eta) := \lim_{n\to\infty} \Xi_\rho(\eta_n)$ and $\displaystyle \xi^*_\rho(\eta) := \lim_{n\to\infty} \Xi^*_\rho(\eta_n)$ exist.
Since $\gamma_n$ and $\eta_n$ piece together
to form a metric quasigeodesic path (Lemma \ref{lem:geodrays_qg}), Proposition \ref{prop:limits_exist_trans} then yields the desired conclusion.
\end{proof}
\end{cor}

Using this together with a compactness argument, we may then prove the following quantitative / finite transversality result.

\spacer \begin{lem}
% [cf. \cite{BPS}, Lemma 3.10 and Lemma 2.5]
\label{lem:biinf_trans_quant}
For every $\eps > 0$, there exist $\ell_1 \geq \ell_0$ and $\delta > 0$ such that for all $\gamma, \eta \in \Gamma$ with \begin{enumerate}[(i)]
\item $|\gamma|_c, |\eta|_c > \ell_1$, and
\item $d(\Xi_\rho(\gamma), \Xi_\rho(\eta)) > \eps$,
\end{enumerate} 
we have $\angle (\Xi_\rho(\gamma), \Xi^*_\rho(\eta)) > \delta$. 
\begin{proof}
The proof will proceed by contradiction. Assume there exist $\eps > 0$ and sequences $\ell_j \to \infty$, $\delta_j \to 0$ such that for each $j$ there exist $\gamma_j, \eta_j \in \Gamma$ with $|\gamma_j|_c, |\eta_j|_c > \ell_j$ and $d\left( \Xi_\rho(\gamma_j), \Xi_\rho(\eta_j) \right) > \eps$, but $\angle (\Xi_\rho(\gamma_j), \Xi^*_\rho(\eta_j)) \leq \delta_j$. 

Consider the $\gamma_j$ and $\eta_j$ as projected geodesics.
By a diagonal argument, these converge, up to subsequence, to some (infinite words) $\gamma := g_1 \cdots g_n \cdots$ and $\eta := h_1 \cdots h_m \cdots$. 
Reparametrizing as needed, we may assume that these are $(6,20)$-metric quasigeodesic paths (these constants being the ones obtained in Proposition \ref{prop:reparam_proj_geod}.)

By Corollary \ref{cor:lim_transverse}, the limits $\xi_\rho(x_\gamma)$ and $\xi^*_\rho(x_\eta)$ exist, and $ \angle (\xi_\rho(x_\gamma), \xi^*_\rho(x_\eta)) > 0 $.

This gives us a contradiction, since, by construction, $\angle (\xi_\rho(x_\gamma), \xi^*_\rho(x_\eta)) = 0$.
\end{proof} \end{lem}

Using this last version of transversality, we then have the following statement describing a sort of North-South dynamics: 
\spacer
\begin{lem}
% [cf. \cite{BPS}, Lemma 3.13] 
\label{lem:ns_dyn_all}
Given $\eps, \eps' > 0$, there exists $\ell > \ell_0$ such that for any 
% $r$-non-peripheral 
$\eta \in \Gamma$ with $|\eta|_c > \ell$ and any $\xi \in \Lambda_{rel}$ with $d(\xi, \Xi_\rho(\eta^{-1}) ) > \eps$, we have
\[ d(\rho(\eta) \xi, \Xi_\rho(\eta) ) \leq \eps' .\]

\begin{proof}
Let $\ell_1 \geq \ell_0$ and $\delta > 0$ be given by Lemma \ref{lem:biinf_trans_quant}, with our given $\eps > 0$. Choose $\ell > \ell_1$ such that $\ubar{C} e^{-\ubar\mu \ell} < \eps' \sin \delta$. 

Fix $\eta \in \Gamma$ and $\xi \in \Lambda_{rel}$ such that $|\eta|_c > \ell$ and $d(\xi, \Xi_\rho(\eta^{-1}) ) > \eps$. Choose a sequence $(\gamma_n) \subset \Gamma$ such that $|\gamma_n|_c \to \infty$ and $\Xi_\rho(\gamma_n) \to \xi$. Without loss of generality assume for each $n$ we have $|\gamma_n|_c > \ell_1$ and 
\[ d(\Xi_\rho(\gamma_n), \Xi_\rho(\eta^{-1}) ) > \eps .\]
It then follows from Lemma \ref{lem:biinf_trans_quant} that
\[ \angle (\Xi_\rho(\gamma_n), \Xi^*_\rho(\eta^{-1})) > \delta \]
and then, by Lemma \ref{lem:BPSA6} with $A = \rho(\eta)$ and $P = \Xi(\gamma_n)$, we obtain
\begin{align*} 
d(\rho(\eta) \, \Xi_\rho(\gamma_n), \Xi_\rho(\eta) ) & \leq \frac{\sigma_2}{\sigma_1}(\rho(\eta)) \frac{1}{\sin \angle \left( \Xi_\rho(\gamma_n), \Xi^*_\rho(\eta^{-1}) \right) } 
\leq \frac{\sigma_2}{\sigma_1}(\rho(\eta)) \frac{1}{\sin \delta} \\
 & \leq \frac{\ubar{C} e^{-\ubar\mu\ell}}{\sin \delta}  < \eps'
\end{align*}
and letting $n \to \infty$ we have $d(\rho(\eta) \xi, \Xi_\rho(\eta) ) \leq \eps'$ as desired.
\end{proof} \end{lem}

\subsection{Perfectness}

\spacer \begin{prop} \label{prop:Mrel_perfect}
$\Lambda_{rel}$ is perfect, 
that is every point in $\Lambda_{rel}$ is an accumulation point of other points in $\Lambda_{rel}$.

\begin{proof}
We first claim that $|\Lambda_{rel}| \geq 3$. By assumption we have non-peripheral and hence (by Lemma \ref{lem:nonperi_proximal}) biproximal elements, and also peripheral elements. The proximal elements give us at least two distinct points $\xi^\pm$ in $\Lambda_{rel}$; the peripheral elements give us at least one point $\xi_P$ in $\Lambda_{rel}$. 

We claim that the peripheral point $\xi_P$ is not fixed by any non-peripheral element of $\Gamma$, and in particular is distinct from the proximal limit points $\xi^\pm$. To see this, suppose $\gamma \in \Gamma$ is non-peripheral and fixes $\xi_P$. Then $\xi_{\gamma P \gamma^{-1}} = \xi_P$, which violates the transversality hypothesis in Definition \ref{defn:peri_conds}. 

Hence $|\Lambda_{rel}| \geq 3$.

Now let $b_1$ be a point in $\Lambda_{rel}$, and let $\eps' > 0$. We will show that the $2\eps'$-neighborhood of $b_1$ contains another element of $\Lambda_{rel}$.

Choose $b_2, b_3$ to be two distinct points of $\Lambda_{rel} \setminus \{b_1\}$. Let $\eps := \frac 12 \min_{i\neq j} d(b_i,b_j)$. Let $\ell > \ell_0$ be given by  Lemma \ref{lem:ns_dyn_all}, depending on $\eps$ and $\eps'$. Choose $\eta \in \Gamma$ such that $|\eta|_c > \ell$ and $d(\Xi_\rho(\eta), b_1) < \eps'$. Consider $\Xi_\rho(\eta^{-1})$ as a linear subspace of $\real^d$; it can be $\eps$-close to at most one of the spaces $b_1, b_2, b_3$. In other words, there are different indices $i, j\in \{1,2,3\}$ 
such that 
\[ d(b_i, \Xi_\rho(\eta^{-1})) > \eps \]
and similarly for $b_j$. In particular, by Lemma \ref{lem:ns_dyn_all}, 
\[ d(\rho(\eta)b_i, b_1) \leq d(\rho(\eta) b_i, \Xi_\rho(\eta)) + \eps' < 2 \eps' .\]
By $\Gamma$-invariance, the spaces $\rho(\eta) b_i$ and $\rho(\eta) b_j$ are in $\Lambda_{rel}$; but at most one of them can be equal to $b_1$.
\end{proof}
\end{prop}

\subsection{Geometrically-finite convergence group action}

We first prove that $\Gamma$ acts on $\Lambda_{rel}$ as a convergence group, that is to say
\spacer 
\begin{prop}
% [cf. \cite{BPS}, Proposition 3.18] 
\label{prop:new318}
The natural induced action of $\Gamma$ on the space $\Lambda_{rel}^{(3)}$ of distinct triples is properly discontinuous.

\begin{proof}
We will pick out a distinguished family of compact sets of $\Lambda_{rel}^{(3)}$, and use these to prove proper discontinuity of the action.
Given $T = (P_1, P_2, P_3) \in \Lambda_{rel}^{(3)}$ a triple of distinct points, define $|T| = |(P_1,P_2,P_3)| := \min_{i \neq j} d(P_i,P_j)$, where $d$ is a(ny) Riemannian metric on the Grassmannian. 
For every $\delta > 0$, $\left\{T \in \Lambda_{rel}^{(3)} : |T| \geq \delta \right\}$ is a compact subset of $\Lambda_{rel}^{(3)}$, and conversely every compact subset of $\Lambda_{rel}^{(3)}$ is contained in a subset of that form. 

We will now establish that, given $\delta > 0$, there exists $\ell \in \nats$ such that if $T \in \Lambda_{rel}^{(3)}$ satisfies $|T| > \delta$ and $\eta \in \Gamma$ 
%is $r$-non-peripheral with 
satisfies $|\eta|_c > \ell$, then $|\rho(\eta)T| < \delta$. 
This will suffice to establish the proposition, since it implies that given any compact subset $\Lambda_{rel}^{(3)}$, all but finitely many words (those of length at most $\ell$) must move the compact subset off itself.

Given $\delta > 0$, let $\ell$ be given by Lemma \ref{lem:ns_dyn_all} with $\eps = \eps' = \frac\delta2$. 

Now consider $(\xi_1,\xi_2,\xi_3) \in \Lambda_{rel}^{(3)}$ such that $|T| > \delta$, and $\eta \in \Gamma$ such that $|\eta|_c > \ell$. Note that $d(\Xi(\eta^{-1}) , \xi_i) > \frac\delta2$ for at least two of the lines $\xi_1,\xi_2,\xi_3$---say, without loss of generality, $\xi_1$ and $\xi_2$.

Lemma \ref{lem:ns_dyn_all} yields $d(\rho(\eta) \xi_i, \Xi_\rho(\eta) ) < \frac\delta2$ for $i=1,2$, and so $$|\rho(\eta) T| \leq d(\rho(\eta) \xi_1, \rho(\eta) \xi_2) < \delta,$$ as desired.
\end{proof}
\end{prop}

We then prove that $\Gamma$ in fact acts on $\Lambda_{rel}$ as a geometrically finite convergence group. By Theorem \ref{thm:Gerasimov}, to demonstrate geometric finiteness it suffices to show cocompactness on the space of distinct pairs. For this we will use an expansivity argument:
\spacer \begin{prop} \label{prop:geomfin}
The natural induced action of $\Gamma$ on the space $\Lambda_{rel}^{(2)}$ of distinct pairs is cocompact.

\begin{proof}
As with the case of distinct triples above, for every $\delta >0$, $\left\{T \in \Lambda_{rel}^{(2)} : |T| \geq \delta \right\}$ is compact subset of $\Lambda_{rel}^{(2)}$, and conversely every compact subset of $\Lambda_{rel}^{(2)}$ is contained in a subset of that form. Here, analogously to above, $|T| := d(\xi_1, \xi_2)$.

We will now prove the following statement: 
there exists $\eps > 0$ such that for every $T = (\xi_1, \xi_2) \in \Lambda_{rel}^{(2)}$, there exists $\gamma \in \Gamma$ such that $|\rho(\gamma)T| \geq \eps$. This suffices to establish the Proposition.

Choose $\eps = \frac12 s_{\min}$, where $s_{\min}$ is the minimum gap from Proposition \ref{prop:limits_exist_trans} for metric {\it geodesic} sequences given our domination constants. 
If $|T| \geq \eps$ then we may take $\gamma = \id$, so we may suppose that $|T| < \eps$. 

Choose $(6,20)$-metric quasigeodesic paths (the constants are from Proposition \ref{prop:reparam_proj_geod}) $(\gamma_i = g_1 \cdots g_{|\gamma_i|}), (\eta_i = h_1 \cdots h_{|\eta_i|}) \subset \Gamma$ such that 
% $|\gamma_i|_c, |\eta_i|_c \in [i, 2i]$ 
% (with uniform quasi-equivalence constants) and 
$\Xi_\rho(\gamma_i) \to \xi_1$, $\Xi_\rho(\eta_i) \to \xi_2$, and consider the sequence of matrices $(\dots, A_{-1}, A_0, A_1, \dots)$ given by $A_i = \rho(g_{i+1}^{-1})$ for $i \geq 0$ and $A_i = \rho(h_{|i|})$ for $i<0$.

By Lemma \ref{lem:geodrays_qg}$, (\dots, \eta_2, \eta_1, \id, \gamma_1, \gamma_2, \dots) =: x$ is a metric quasigeodesic.

If the sequence for $\xi_1$ is not eventually peripheral, then we may find an increasing sequence of 
$i_m > 0$
such that the shifted sequences
\[ \sigma^{i_m} x := \left( \sigma^{i_m} A_n := A_{n+i_m} \right)_{n\in\ints} \]
converge (as $m\to\infty$) to a metric geodesic sequence $\sigma^\infty x = (B_n)_{n\in\ints}$, i.e. $\displaystyle B_n = \lim_{m\to\infty} \sigma^{i_m} A_n$ for each $n \in \ints$. By construction, for any given $N$ we can find $m_0$ so that $\sigma^{i_m} A_n = B_n$ whenever $|n| \leq N$ and $m \geq m_0$.

By Proposition \ref{prop:limits_exist_trans}, $\sin \angle \left( E^u(\sigma^\infty x), E^s(\sigma^\infty x) \right) > 2\eps$.  Moreover, by Corollary \ref{cor:EC}, 
% (and Remark \ref{rmk:EC_ext}), 
for all large enough $m$ (given $\nu'$ and $c'$), $\sin \angle \left( E^*(\sigma^{i_m} x), E^*(\sigma^\infty x)\right) < \frac\eps2$ for $* \in \{u,s\}$, so that 
\[ \sin \angle \left( E^u(\sigma^{i_m} x), E^s(\sigma^{i_m} x) \right) > \eps .\]

Since the endpoints of $\sigma^{i^m} x$ are given by acting on the endpoints of $x$ by $A_{i_m-1} \cdots A_0 = \rho(g_1 \cdots g_{i_m})^{-1} = \rho(\gamma_{i_m}^{-1})$, this establishes that $|\rho(\gamma_{i_m}^{-1})(\xi_1,\xi_2)| \ge \eps$, as desired.

We argue similarly if the sequence for $\xi_2$ is not eventually peripheral.

If the sequences for both $\xi_1$ and $\xi_2$ are eventually peripheral, there is a positive lower bound on the (infimum of the) distance between these (over all shifts, as above):
if not, we can find $P,P' \in \mathcal{P}$ and a sequence of words $w_n \to \infty$ not starting with a letter from $P$ such that $d(\xi(P), w_n \xi(P')) < 2^{-n}$. Up to subsequence, the $w_n$ converge to some infinite geodesic such that $\displaystyle \lim_{n\to\infty} \Xi_\rho(w_n) = \xi(P)$; but now observe that this infinite geodesic cannot be eventually peripheral in both directions---these limit points are all distinct by hypothesis---, and by the arguments above neither can it be not eventually peripheral. We conclude, by contradiction, that said lower bound must in fact exist. 
\end{proof} \end{prop}

\subsection{Peripherals are maximal parabolics}

\begin{lem} \label{lem:nonperi_limit}
For any non-peripheral $\gamma \in \Gamma$, $\displaystyle \lim_{n \to \infty} \Xi_\rho(\gamma^n)$ is the top eigenline of $\rho(\gamma)$ 

\begin{proof}
Recall that $\rho(\gamma)$ is necessarily proximal (Proposition \ref{lem:nonperi_proximal}), so that the top eigenline is well-defined.

To show $\lim_{n \to \infty} \Xi_\rho(\gamma^n)$ is the top eigenline of $\rho(\gamma)$, we may apply Lemma \ref{lem:BPSA6} with $A = \rho(\gamma^n)$ and $L$ the top eigenline; then $d(L, \Xi_\rho(\gamma^n)) \leq C_\gamma e^{-\mu_\gamma n}$ for positive constants $C_\gamma, \mu_\gamma$ depending only on $\rho(\gamma)$; in particular, as $n \to \infty$, this bound goes to zero, so that $\displaystyle \lim_{n\to\infty} \Xi_\rho(\gamma^n) = L$ as desired. 
\end{proof} 
\end{lem}

\spacer \begin{prop} \label{prop:para_peri}
The maximal parabolic subgroups of $\Gamma$ are precisely (conjugates of) peripheral subgroups.

\begin{proof}
Suppose $H$ is a maximal parabolic subgroup.

Observe that $H$ cannot contain non-peripheral elements. Indeed, suppose $\gamma \in \Gamma$ is non-peripheral. From Lemma \ref{lem:nonperi_proximal} and \ref{lem:nonperi_limit}, $\rho(\gamma)$ is proximal, and $\displaystyle \lim_{n \to \infty} \Xi_\rho(\gamma^n)$ is the top eigenline of $\rho(\gamma)$. Similarly, $\rho(\gamma^{-1})$ is proximal, and $\displaystyle \lim_{n \to \infty} \Xi_\rho(\gamma^{-n})$ is the bottom eigenline of $\rho(\gamma)$. These are distinct (by proximality), and are both fixed by $\gamma$, so $\gamma \notin H$.

Hence every $\gamma \in H$ is peripheral. 

Now, from the unique limits hypothesis in Definition \ref{defn:peri_conds}, for any peripheral subgroup $P$, $\displaystyle \lim_{n \to \infty} \Xi_\rho(\eta_n) = \xi_\rho(P)$ for any sequence $(\eta_n) \subset P$, and so $P$ fixes $\xi_\rho(P)$. By Lemma \ref{lem:ns_dyn_all}, $P$ fixes no other point $\beta \in \Lambda_{rel}$: any such $\beta$ is at some definite distance $\eps(\beta) >0$ from $\xi(P)$, and hence by Lemma \ref{lem:ns_dyn_all}, sufficiently long words in $P$ must move $\beta$ off of itself.
Hence every peripheral subgroup $P$ is parabolic, and extends to some maximal parabolic subgroup $\hat{P}$. 

Suppose $\hat{P} \smallsetminus P \neq \varnothing$, so that $\hat{P}$ also contain some non-identity element $q$ of some other peripheral subgroup $Q \neq P$. By the torsionfree assumption, $\hat{P} \cap Q$ contains arbitrarily large powers of $q$. 
By the same argument as in the previous paragraph, this implies that $Q \subset \hat{P}$.  
But this contradicts the first part of the uniform transversality hypothesis which stipulates that $\xi_\rho(P) \neq \xi_\rho(Q)$.

Hence we must have $\hat{P} = P$, i.e. 
the maximal parabolic subgroups are exactly the peripheral subgroups, as desired.
\end{proof} \end{prop}
% Do not need to specifically show these are bounded---see statements of Yaman and Gerasimov's theorems.

It follows from the above that the parabolic points in $\Lambda_{rel}$ are precisely the peripheral fixed points.

\subsection{Summary of argument}
\begin{proof}[Proof of Theorem \ref{thm:reldom_relhyp}]

Consider a representation $\rho: \Gamma \to \GL(d,\real)$ which is 1-dominated relative to a prescribed collection of peripheral subgroups $\mathcal{P}$, such that $\Gamma$ contains at least one non-peripheral element.

$\rho$ induces an action of $\Gamma$ on the space of lines $\proj(\real^d)$. Consider $\Lambda_{rel} \subset \proj(\real^d)$. It is non-empty, compact and $\Gamma$-invariant (Proposition \ref{prop:BPS311}), and perfect (Proposition \ref{prop:Mrel_perfect}.) 

The diagonal action of $\Gamma$ on $\Lambda_{rel}^{(3)}$ is properly discontinuous (Proposition \ref{prop:new318}) and the diagonal action on $\Lambda_{rel}^{(2)}$ is cocompact (Proposition \ref{prop:geomfin}.) 

Moreover the maximal parabolic groups are precisely the peripheral subgroups; by Theorem \ref{thm:Gerasimov} and since conical limit points cannot be parabolic 
% (Tukia)
these are all bounded, and in particular the stabiliser of each bounded parabolic point is finitely-generated (Proposition \ref{prop:para_peri}.)

We summarize all of this in a statement that will be used again in the next section:
\spacer \begin{prop} \label{prop:geomfin_conv}
Given a representation $\rho: \Gamma \to \GL(d,\real)$ which is 1-dominated relative to $\mathcal{P}$, $\rho(\Gamma)$ acts on $\Lambda_{rel}$ as a geometrically-finite convergence group, with $\mathcal{P}^\Gamma$ as the set of maximal parabolic subgroups.
\end{prop}

Hence, by Theorem \ref{thm:Gerasimov}, $\Gamma$ is hyperbolic relative to $\mathcal{P}$.
\end{proof}

\section{Limit maps} \label{sec:limitmaps}
In this section, we prove that a relatively dominated representation $\rho: (\Gamma,\mathcal{P}) \to \GL(d,\real)$ gives us a pair of limit maps from the Bowditch boundary $\del(\Gamma,\mathcal{P})$ into projective space and its dual.
% (or, equivalently, into a suitable flag variety.)

In the case where $\mathcal{P} = \varnothing$, this recovers the limit maps from the Gromov boundary of the group into projective space and its dual that we obtain for an Anosov representation.

\spacer \begin{defn}
Suppose $\Gamma$ is hyperbolic relative to $\mathcal{P}$, and we have a pair of continuous maps $\xi: \del(\Gamma, \mathcal{P}) \to \proj(\real^d)$ and $\xi^*: \del(\Gamma, \mathcal{P}) \to \proj(\real^{d*})$.

$\xi$ and $\xi^*$ are said to be {\bf compatible} if $\xi(\eta) \subset \theta(\eta)$ as linear subspaces for all $\eta \in \del(\Gamma,\mathcal{P})$.

$\xi$ and $\xi^*$ are said to be {\bf transverse} if $\xi(\eta) \oplus \theta(\eta') = \real^d$ for all $\eta \neq \eta'$.

Given $\rho: \Gamma \to \GL(d,\real)$ such that $\rho(P)$ is a parabolic subgroup of $\GL(d,\real)$ for each $P \in \mathcal{P}$, $\xi$ and $\xi^*$ are said to be {\bf dynamics-preserving} if 
% they take attracting fixed points to corresponding attracting fixed points, i.e. 
\begin{enumerate}[(i)]
\item $\xi(\gamma^+) = (\rho(\gamma))^+$ and $\xi^*(\gamma^+)^\perp = (\rho^*(\gamma))^+$.
% Why is the $\perp$ there? See \ref{sub:dual_rep}
for all nonperipheral $\gamma \in \Gamma$, where $\gamma^+ := \lim_{n\to\infty} \gamma^n \in \del(\Gamma,\mathcal{P})$ and $\rho(\gamma)^+$ is the attracting eigenline for $\rho(\gamma)$,
% or the eigenline associated to the top Jordan block if $\gamma \in P$ is peripheral.
and
\item If $\del P \in \del(\Gamma,P)$ is the unique point associated to $P \in \mathcal{P}$, then $\xi(\del P)$ is the parabolic fixed point associated to $\rho(P)$.
\end{enumerate}
\end{defn}

\spacer \begin{thm} \label{thm:limitmaps}
Given $\rho: \Gamma \to \GL(d,\real)$ 1-dominated relative to $\mathcal{P}$, we have well-defined, $\rho(\Gamma)$-equivariant, continuous maps $\xi_\rho: \del(\Gamma, \mathcal{P}) \to \proj(\real^d)$ and $\xi^*_\rho: \del(\Gamma, \mathcal{P}) \to \proj(\real^{d*})$ which are dynamics-preserving, compatible, and transverse.

\begin{proof}
Recall that if $\rho: \Gamma \to \GL(d,\real)$ is 1-dominated relative to $\mathcal{P}$, then $\Gamma$ is hyperbolic relative to $\mathcal{P}$ by Theorem \ref{thm:reldom_relhyp}. Moreover, as noted in Proposition \ref{prop:geomfin_conv}, $\rho(\Gamma) \actson \Lambda_{rel}$ as a geometrically-finite convergence group, with $\mathcal{P}^\Gamma$ as the set of maximal parabolic subgroups.

Yaman's criterion (Theorem \ref{thm:Yaman}) then gives us an equivariant homeomorphism 
\[ \xi_\rho: \del(\Gamma,\mathcal{P}) \to \Lambda_{rel} \subset \proj(\real^d) .\]
By looking at the action of $\rho(\Gamma)$ on the dual vector space (recall \S\ref{sub:dual_rep} and in particular Proposition \ref{prop:dualrep_reldom}), we similarly obtain an equivariant homeomorphism 
\[ \xi^*_\rho: \del(\Gamma,\mathcal{P}) \to \Lambda_{rel}^* \subset \Gr_{d-1}(\real^d) .\]

Equivariance then combines with the other properties of our limit set $\Lambda_{rel}$ to imply that $\xi_\rho$ and $\xi^*_\rho$ are dynamics-preserving. Here we state the arguments for $\xi_\rho$; via the dual representation $\rho^*$ they also imply the claim for $\xi^*_\rho$.

For non-peripheral elements $\gamma$, the attracting eigenline $\rho(\gamma)^+$ is contained in $\Lambda_{rel}$ (Lemma \ref{lem:nonperi_limit}). Every point in $\proj(\real^d)$---outside a hyperplane given by the orthogonal complement of $\rho(\gamma)^+$---is attracted to $\rho(\gamma)^+$ under the action of $\rho(\gamma)$. By the transversality properties of $\Lambda_{rel}$, there exist points of $\Lambda_{rel}$ outside of this hyperplane, since said hyperplane is equal to the attracting hyperplane of $\rho^*(\gamma^{-1})$, and by Corollary \ref{cor:lim_transverse} any point of $\Lambda_{rel}$ other than $\rho(\gamma^{-1})^+$ is transverse to this.

Hence, by equivariance, we have that $\xi_\rho(\gamma^n \zeta) = \rho(\gamma^n) \xi_\rho(\zeta) \to \rho(\gamma)^+$ as $n \to \infty$, for an open set of $\zeta \in \Lambda_{rel}$, and so $\displaystyle \xi_\rho(\gamma^+) = \xi_\rho \left( \lim_{n\to\infty} \gamma^n \right) = \rho(\gamma)^+$.

For peripheral elements $\eta \in P$, the associated limit line $\xi_\rho(P)$ is 
contained in $\Lambda_{rel}$ by the unique limits assumption. Since $\xi$ is a homeomorphism, there is some $\zeta \in \del(\Gamma,\mathcal{P})$ such that $\xi_\rho(\zeta) = \rho(\eta)^+$. 
By equivariance, $\xi_\rho(\eta^n \zeta) = \rho(\eta^n) \xi_\rho(\zeta) \to \rho(\eta)^+$ as $n \to \infty$. Hence $\displaystyle \xi_\rho(\eta^+) = \xi_\rho \left( \lim_{n\to\infty} \eta^n \right) = \rho(\eta)^+$.

To verify that $\xi_\rho$ and $\xi^*_\rho$ are compatible and transverse, we will show that $\xi_\rho, \xi^*_\rho$ satisfy
\begin{align*}
\xi_\rho(x) & = \lim_{n \to \infty} \Xi_\rho(\gamma_n) &
\xi^*_\rho(x) & = \lim_{n \to \infty} \Xi^*_\rho(\gamma_n) 
\end{align*}
for $(\gamma_n) \in \Gamma$ any projected geodesic in $\Gamma$ such that $\gamma_n \to x$, and $\Xi_\rho$ and $\Xi^*_\rho$ as in \S\ref{sec:reldom_relhyp}.

To see this, we note that if $x = \gamma^+ \in \del(\Gamma,\mathcal{P})$ is a proximal limit point, then $\xi_\rho(x)$ is the top eigenline of $\rho(\gamma)$ since $\xi$ is dynamics-preserving, and by Lemma \ref{lem:nonperi_limit} this is equal to $\lim_{n\to\infty} \Xi_\rho(\gamma^n)$. If $x = \del P \in \del(\Gamma,\mathcal{P})$ is a parabolic limit point, then by the dynamics-preserving property $\xi_\rho(x) = \xi_\rho (\eta^+)$ for any $\eta \in P$, and by the unique limits hypothesis $\xi_\rho(x) = \xi_\rho (\eta^+) = \lim_{n\to\infty} \Xi_\rho(\eta_n)$ for any sequence $\eta_n \to \infty$ in $P$.

More generally, given $x \in \del(\Gamma,\mathcal{P})$ that is not a peripheral fixed point , suppose $(\gamma_n)$ is a sequence (along a metric quasigeodesic path) such that no $\gamma_n$ ends in a peripheral letter and $\gamma_n \to x$. Pick any peripheral element $\eta \in \bigcup \mathcal{P}$.

Then, writing $x_n := \displaystyle \lim_{m\to\infty} \gamma_n \eta^m$, we have \[ \displaystyle \lim_{n\to\infty} x_n = \lim_{n\to\infty} \lim_{m\to\infty} \gamma_n \eta^m = \lim_{n\to\infty} \gamma_n = x \] 
(once $n$ and $m$ are large enough, by Lemma \ref{lem:wordsumcomp} the sequences involved may be taken to be uniform quasigeodesics.)

By continuity, $\displaystyle \xi_\rho(x) = \lim_{n\to\infty} \xi_\rho(x_n)$; we then have 
\[ \xi_\rho(x) = \lim_{n\to\infty} \xi_\rho(x_n) = \lim_{n\to\infty} \lim_{m\to\infty} \Xi_\rho(\gamma_n \eta^m) = \lim_{n\to\infty} \Xi_\rho(\gamma_n) \]
where the last equality follows from Corollary \ref{cor:EC} (because the $\gamma_n\eta^m$ may be taken to be uniform quasigeodesics) and the triangle inequality: 
\begin{align*}
d(\Xi(\gamma_n), \xi(x)) & \leq 
d\left( \Xi(\gamma_n), \Xi(\gamma_n \eta^m) \right) + 
d\left( \Xi(\gamma_n \eta^m), \xi(x_n) \right) + 
d\left( \xi(x), \xi(x_n) \right)
\\ 
 & \leq \hat{C} e^{-\hat\mu n} + \hat{C} e^{-\hat\mu m}  + d\left( \xi(x), \xi(x_n) \right)
\end{align*}
and all of the terms that appear in the last line can be made arbitrarily small by taking ($m$ and then) $n$ sufficiently large.

We have written the argument above for $\xi_\rho$; the argument for $\xi^*_\rho$ is entirely analogous.

The compatibility of $\xi_\rho$ and $\xi^*_\rho$ then follows since $\Xi_\rho(\gamma_n) \subset \Xi^*_\rho(\gamma_n)$ for all $n$ by definition;  the transversality of $\xi_\rho$ and $\xi^*_\rho$ follows from Corollary \ref{cor:lim_transverse}. 
\end{proof}
\end{thm}

\begin{rmk}
We may alternatively prove this by defining the limit maps using 
\begin{align*}
\xi_\rho(x) & = \lim_{n \to \infty} \Xi_\rho(\gamma_n) &
\xi^*_\rho(x) & = \lim_{n \to \infty} \Xi^*_\rho(\gamma_n) 
\end{align*}
for $(\gamma_n) \in \Gamma$ any projected geodesic in $\Gamma$ such that $\gamma_n \to x$, as in \cite{GGKW}, and directly showing, using arguments similar to those above and earlier in the paper, that these maps satisfy the desired properties.
From the analysis above these will turn out to be equivalent to the limit maps supplied by Yaman's criterion.
\end{rmk}

\section{Examples} \label{sec:eg}

For a start, we observe that dominated representations are relatively dominated relative to $\mathcal{P} = \varnothing$, since in that case we have $|\cdot|_c = |\cdot|$.
We will now show that geometrically finite subgroups of $\SO(1,d)$ and geometrically finite convex projective holonomies, in the sense of \cite{CM12}, give examples of relatively dominated representations. 

\subsection{In rank one}

In rank one, the relatively dominated condition coincides with the more classical notion of geometric finiteness. Here we will illustrate the particular example of geometrically finite real hyperbolic manifold holonomies; the arguments for the more general case are similar.

\spacer \begin{eg} \label{eg:hyp_hol}
Let $M$ be a geometrically finite hyperbolic $d$-manifold, $\Gamma = \pi_1 M$, and $\rho: \Gamma \to \PSO(d,1) \subset \PSL(d+1,\real)$ be its holonomy representation.

In this case we know that $\Gamma$ is hyperbolic relative to the cusp stabilizers $\mathcal{P}$, and that the relative Cayley graph, and in fact the cusped space, quasi-isometrically embeds into $\HH^d$.

\begin{proof}[Proof of quasi-isometry]
This may be verified directly using hyperbolic geometry: 
%by the Milnor-\v{Sv}arc lemma, the relative Cayley graph is quasi-isometric to $C$ with a system of horoballs removed. By Example \ref{eg:horoball_qi}, each horoball in this system of horoballs is quasi-isometric to a combinatorial horoball. Since $M$ has finitely many cusps, there are finitely many conjugacy classes of horoballs, and we may make a uniform choice of constants for all of these quasi-isometries.
%In fact, by examining the proofs more carefully, we may assert that there exist constants $k \geq 1$ and $c \geq 0$ such that
%\[ \frac1k |\gamma|_c \leq d_{\HH^n}(o, \gamma \cdot o) \leq k|\gamma|_c + c .\]
%To prove this: 
 the quasi-isometric embedding of the Cayley graph is still given by the orbit map. This sends the ends of each coset $\gamma P$ of a cusp subgroup $P$ to a single point $\xi \in \del\HH^n$, and we may extend the orbit map to a quasi-isometric embedding of the combinatorial horoball over $\gamma P$ (the 0-simplices of which we address as elements of $P \times \ints_{\geq 0}$) to a quasi-horoball based at $\xi$ as follows:
\begin{itemize}
\item for each $p \in \gamma P$, let $\eta_p: [0,\infty) \to \HH^n$ be the geodesic ray from the image of $p$ to $\xi$;
\item send $(p, n)$ to $\eta_p(\lambda n)$ with $\lambda = e^2/2$ (the normalization constant needed so that the exponential decay factor between levels of the combinatorial horoballs matches the exponential decay factor between their images in $\HH^n$.) 
\end{itemize}

Call this map $\phi$. To check that this is indeed a quasi-isometric embedding, or more precisely a quasi-isometry to $C \subset \HH^n$ where $C$ is the convex hull of the limit set of $\Gamma$, we invoke the following argument of Cannon and Cooper:

\spacer \begin{lem}[\cite{cannon_cooper}, Lemma 4.2] \label{lem:cannon_cooper}
Given two spaces $X$, $Y$ with path metrics $d_X$, $d_Y$, $\phi: X \to Y$ is a quasi-isometry if it satisfies the following three conditions: \begin{enumerate}[(i)]
\item (quasi-onto) for some $\eps> 0$,  $Y \subset N(\phi(X), \eps)$ (the $\eps$-neighborhood of $\phi(X)$);
\item (Lipschitz) for some $L > 0$ and all $x_1, x_2 \in X$, $d_Y(\phi(x_1), \phi_2(x)) \leq L d_X(x_1, x_2)$; and
\item (uniformly non-collapsing) for each $R > 0$ there exists an $r > 0$ such that if $d_X(x_1, x_2) > r$ then $d_Y(\phi(x_1), \phi_2(x_2)) > R$. 
\end{enumerate}
\end{lem}

Let $p_1, \dots, p_k$ be parabolic fixed points belonging to different conjugacy classes, and $N_1, \dots, N_k$ be a system of disjoint horoballs based at $p_1, \dots, p_k$ (resp.) in $\HH^n$ such that $\bigcup \{\gamma N_i : \gamma\in \Gamma; i=1, \dots, n\} =: \mathcal{N}$ fills out a family of disjoint open horoballs in $\HH^n$, and 
\[ \phi(\Gamma) \subset \HH^n \setminus \mathcal{N} =: \mathcal{Q} .\]
($\mathcal{Q}$ is the ``thick part'', or ``truncated hyperbolic space''.)

% Let $D$ be a fundamental domain for the action of the fundamental group acting via the holonomy representation. We remark that $D \cap Q$ is compact.

To verify condition (i) here: let $y$ be a point of $C \subset \HH^n$. Then either there exists some $i\in\{1,\dots,n\}$ and $\gamma \in \Gamma$ such that $y \in \gamma N_i$, or $y \in \mathcal{Q}$.
% there exists $\gamma \in \Gamma$ such that $x \in \gamma(D \cap Q)$. 
In the latter case, 
\[ d_{\HH^n}(y, \phi(X^{(0)})) \leq \diam (\mathcal{Q} / \Gamma) < \infty .\]
In the former case, consider the horoball $\gamma N_i$, which has center $\gamma p_i =: p$. As noted in Example \ref{eg:horoball_qi}, $y$ is within distance $\delta$ of 
% the horosphere $H$ centered at $p$ containing the image $\phi(\Gamma_p \times \{j\})$ of level $j$ of a combinatorial horoball, for some $j \in \nats$. (Here $\Gamma_p = \gamma P_i \gamma^{-1}$ is the maximal parabolic subgroup of $\Gamma$ fixing $p$.) Moreover, every point of $H$ is a uniform distance $\delta$ from 
a vertex of the combinatorial horoball for $\gamma P_i$, where $\gamma P_i \gamma^{-1}$ is the maximal parabolic subgroup of $\Gamma$ fixing $p$, 
where $\delta$ may be chosen independent of $i$ and $H$.

Hence, condition (i) of the Lemma is satisfied with $\eps \geq \max\{\diam (\mathcal{Q}/\Gamma), \delta + 1 \} < \infty$.

For condition (ii): by Milnor-\v{S}varc, $\phi$ is a quasi-isometry between the Cayley graph and the truncated hyperbolic space $\HH^n \setminus N$. As noted in Example \ref{eg:horoball_qi}, $\phi$ is a quasi-isometry between the system of combinatorial horoballs and the system of horoballs $N$. In both cases, in fact, it is not difficult to show that the quasi-isometry in question is Lipschitz, in the latter case with uniform constants across the entire system of horoballs. This, together with the triangle inequality, gives us that $\phi$ is Lipschitz as a map from all of $X$ to $\HH^n$.

%distances between endpoints of vertical edges are multiplied by a fixed constant $\lambda$ by $\phi$. Endpoints of horizontal edges at level $j$ are $2^{-j}$ apart in the cusped space. The horospherical distance between image endpoints is $2^{-j}$ times the the horospherical distance between the corresponding image endpoints at level 0, but at level 0 (i.e. on the Cayley graph) the ratio between hyperbolic distance and horospherical distance is uniformly bounded since the hyperbolic distance (between images of adjacent vertices) is uniformly bounded. 

For condition (iii): suppose, on the contrary, that there exists $R > 0$ such that for every positive integer $m$, there exist points $x_m, w_m \in X$ such that $d(x_m, w_m) \geq m$, but $d(\phi(x_m), \phi(w_m)) \leq R$.

Since $\phi$ is a quasi-isometry between the system of combinatorial horoballs and the system of horoballs removed from hyperbolic space, there exists $r_0 > 0$ such that if $x$ and $w$ are points in the same combinatorial horoball,and $d(x,w) > r_0$, then $d(\phi(x), \phi(w)) > R$.

Suppose $m \geq (L+1) r_0$, where $L$ is the Lipschitz constant from (ii); without loss of generality suppose $L \geq 1$. Choose a geodesic path from $x_m$ to $w_m$ in $X$. If this geodesic path has a connected subpath of length at least $L r_0$ in a combinatorial horoball, then by the previous paragraph $d(x_m, w_m) \geq R$. Otherwise the geodesic path has a connected subpath of length at least $r_0$ with both endpoints in the Cayley graph. Then, by the same computation as in Example \ref{eg:horoball_qi}, 
\begin{align*}
d(\phi(x_m), \phi(w_m)) & \geq 2 \log d_{\mathcal{Q}}(\phi(x_m), \phi(w_m)) \\ 
 & \geq L \log |x_m^{-1} w_m| \geq \frac L2 (|x_m^{-1} w_m|_c -1) \geq \frac L2 (r_0 - 1) .
\end{align*} 
In particular, if we suppose (without loss of generality---choose $r_0$ to be larger if not) $\frac L2 (r_0 - 1) \geq R$, then we have a contradiction.

This verifies the hypotheses of Lemma \ref{lem:cannon_cooper}, and hence $\phi$ is a quasi-isometry as desired.
\end{proof}

% (see \cite{cannon_cooper}, where the arguments are done for $n=3$ but generalize to arbitrary dimension.)

The quasi-isometric embedding of the relative Cayley graph immediately gives us both lower and upper domination inequalities (D$^\pm$), since $\frac{\sigma_1}{\sigma_2}(\rho(\gamma)) = \frac12 \frac{\sigma_1}{\sigma_{d+1}}(\rho(\gamma))$ for any $\gamma \in \Gamma$, 
and there exists a basepoint $o \in \HH^d$ so that $d(o,\rho(\gamma) \cdot o)=\log \frac{\sigma_1}{\sigma_{d+1}}(\rho(\gamma))$ for all $\gamma \in \Gamma$.

The unique limits condition is satisfied since each cusp stabilizer is parabolic; the quadratic gaps condition is satisfied in the peripherals since, by a direct computation,
\[ \left|\log \frac{\sigma_1}{\sigma_2} \left( \rho(\eta) \right) - 2 \log n \right| = \left|d(o, \rho(\eta)\cdot o) - 2 \log n \right| \leq C_\eta  \]
for any parabolic element $\eta$, where $C_\eta$ is a constant depending on $\eta$. 
Conjugation changes this by a fixed additive constant, and we may take a uniform choice of such constant. The quadratic gaps condition is then satisfied in full, due to the following argument:
\spacer \begin{defn}
We say $\rho: (\Gamma,\mathcal{P}) \to \PGL(d,\real)$ {\bf admits good limit maps} if
\begin{itemize}
\item $\xi_\rho: \del(\Gamma,\mathcal{P}) \to \proj(\real^d)$ given by $\displaystyle\lim_{n\to\infty} \gamma_n \mapsto \lim_{n \to \infty} \Xi_\rho(\gamma_n)$ and
\item $\xi^*_\rho: \del(\Gamma,\mathcal{P}) \to \proj(\real^d)^*$ given by $\displaystyle \lim_{n\to\infty} \gamma_n \mapsto \lim_{n \to \infty} \Xi^*_\rho(\gamma_n)$ 
\end{itemize}
are well-defined, continuous, $\rho(\Gamma)$-equivariant, compatible, dynamics-preserving and transverse.
\end{defn}

We note that in our case $\rho$ admits good limit maps, with the image of $\xi_\rho$ being, up to conjugation in $\PSL(d+1,\real)$, the limit set in the boundary of the Beltrami--Klein projective ball model of hyperbolic $d$-space in $\proj(\real^{d+1})$, and the image of $\xi_\rho^*$ consisting of hyperplanes tangent to the boundary.

\spacer \begin{prop} \label{prop:jump_quadgap}
Suppose $\rho: (\Gamma,\mathcal{P}) \to \PGL(d,\real)$ admits good limit maps, and the quadratic gaps condition is satisfied for peripheral elements $\eta \in \bigcup \mathcal{P}$. 

Then the peripherals satisfy the quadratic gaps condition in full. 
\end{prop}
\begin{proof}
Given a geodesic $\gamma\eta$ where $\eta$ is peripheral, 
Lemma \ref{lem:BPSA7} gives us
\[ \frac{\sigma_1}{\sigma_2}(\rho(\gamma\eta)) \geq \delta^2 \cdot \frac{\sigma_1}{\sigma_2}(\rho(\gamma)) \cdot \frac{\sigma_1}{\sigma_2}(\rho(\eta)) ,\]
where $\delta := \sin \angle (\Xi(\eta), \Xi^*(\gamma^{-1}))$; we then obtain the quadratic gaps condition for $\gamma\eta$ by using the transversality of the limit maps to obtain a uniform positive lower bound
on $\delta$ and observing that $\frac{\sigma_1}{\sigma_2}(\rho(\gamma)) \geq 1$. 
More precisely: suppose no such $\delta$ exists; then we have a sequence of metric quasigeodesics $\gamma_n \eta_n$, with $\eta_n$ peripheral such that $\sin \angle (\Xi(\eta_n), \Xi*(\gamma_n^{-1})) \geq 2^{-n}$. Up to subsequence, these converge to some bi-infinite metric quasigeodesic $\gamma_\infty \eta_\infty$ with $\sin \angle(\xi(\eta_\infty), \xi^*(\gamma_\infty^{-1})) = 0$; but this is in contradiction with the transversality of the limit maps.
\end{proof}

The uniform transversality condition is also satisfied due to the good limit maps, by the following
\spacer \begin{prop} \label{prop:trans_unitrans}
Suppose $\rho: (\Gamma, \mathcal{P}) \to \PGL(d,\real)$ admits good limit maps. Then the uniform transversality hypothesis from Definition \ref{defn:peri_conds} is satisfied.
\begin{proof}
By the transversality of the limit maps, $\gamma(g^{-1} v_1(P), h W_{d-1}(P')) > 0$. To obtain the {\it uniform} version of this hypothesis, suppose we have sequences $(\gamma_n), (\eta_n) \subset \Gamma$ and peripheral subgroups $P, P'$ such that $\angle (\gamma_n^{-1} v_1(P'), \eta_n W_{d-1}(P)) < 2^{-n}$. Up the subsequence, the $\gamma_n^{-1}$ converge to some infinite (projected quasi-)geodesic $\gamma^{-1}: \nats \to \Gamma$, and the $\eta_n$ to some infinite (projected quasi-)geodesic $\eta: \nats \to \Gamma$ and $\angle( \xi_\rho(\gamma^{-1}), \xi^*_\rho(\eta)) = 0$; but this contradicts transversality. 
\end{proof} \end{prop}
\end{eg}

\subsection{A higher rank example}

In higher rank, we have holonomies of geometrically-finite convex projective $n$-manifolds, in the sense of \cite{CM12}:
\spacer \begin{defn}[\cite{CM12}, D\'efinition 1.5 and Th\'eor\`eme 1.3]
Let $\Omega \subset \proj(\real^{d+1})$ be a strictly convex domain with $C^1$ boundary. A finitely-generated discrete subgroup $\Gamma \leq \Aut(\Omega)$ is {\bf geometrically finite} if the 1-neighborhood of the convex core $\overline{C(\Lambda_\Gamma) / \Gamma} \subset \Omega / \Gamma$ is of finite volume.
\end{defn}

\spacer \begin{prop} \label{prop:eg_convproj}
Let $M$ be a $d$-manifold and write $\Gamma = \pi_1 M$. Suppose $\rho: \Gamma \to \PGL(d+1,\real)$ is a geometrically-finite convex projective holonomy representation. Then $\rho$ is 1-dominated relative to its cusp stabilizers.
\begin{proof}
Let $\Omega := \tilde{M}$; this is a strictly convex domain in $\proj(\real^{d+1})$ with $C^1$ boundary, and hence $\delta$-hyperbolic given the Hilbert metric. $\Gamma$ is hyperbolic relative to its cusp stabilizers $\mathcal{P}$, and acts on its limit set $\Lambda_\Gamma \subset \del\Omega$ of accumulation points as a geometrically-finite convergence group (\cite{CM12}, Th\'eor\`eme 1.9.)

In fact $\Lambda_\Gamma$, as well as the dual limit set $\Lambda_\Gamma^* \subset \proj(\real^{d+1})^*$, may be equivariantly identified with $\del(\Gamma,\mathcal{P})$, giving us continuous, compatible, dynamics-preserving limit maps; in particular $\xi^*_\rho(x)$ is tangent to $\del\Omega$ at $\xi_\rho(x)$. This gives us the unique limits condition.
% and hence we have a relative asymptotic embedding. 
Since $\del\Omega$ is strictly convex and $C^1$, 
% we have the antipodality hypothesis.
these limit maps are transverse. This gives us, via Proposition \ref{prop:trans_unitrans}, the uniform transversality condition.

By \cite{CLT}, Theorem 0.5, all of the peripheral elements $\eta \in \bigcup \mathcal{P}$ have image $\rho(\eta)$ projectively equivalent to an element in the holonomy of a hyperbolic cusp; in particular (cf. Example \ref{eg:hyp_hol}), we have quadratic gaps in the peripheral subgroups, and hence, by Proposition \ref{prop:jump_quadgap}, the quadratic gaps condition in full.
% $\rho(\eta)$ has quadratic-gap Jordan structures. 

We now claim that the orbit map is a relative quasi-isometric embedding from $(\Gamma,d_c)$ into $(\Omega,d_\Omega)$, where $d_\Omega$ denotes the Hilbert metric on $\Omega$, and $d_\Omega(o, \gamma \cdot o) = \log\frac{\sigma_1}{\sigma_{d+1}}(\rho(\gamma))$ for all $\gamma \in \Gamma$. 

To establish this, we observe that 
\begin{itemize}
    \item since the cusps are projectively equivalent, and hence isometric, to hyperbolic cusps, we have a system of disjoint horoballs $\mathcal{N}$ of $\Omega$, with boundaries the images of cusp stabilizers, which is quasi-isometric to our system of combinatorial horoballs; 
    % following the computations in Examples \ref{eg:hyp_hol} / \ref{eg:horoball_qi};
% to reduce to the hyperbolic case. 
    \item the cocompact action of $\rho(\Gamma)$ on the compact core of $\tilde{M}$ as a geometrically-finite convex projective manifold gives, by the Milnor-\v{S}varc lemma, a quasi-isometry from $\Cay(\Gamma)$ with the word metric to the truncated domain $\Omega \setminus \mathcal{N}$.
\end{itemize}
Then we may apply the same argument as in Example \ref{eg:hyp_hol}, using Lemma \ref{lem:cannon_cooper}, to obtain our relative quasi-isometric embedding.

Finally, by \cite{CM14}, Proposition 7.2, Corollaire 7.3 and Lemme 7.6, there exists $\eps = \eps(\rho) > 0$ such that $\log \frac{\lambda_1}{\lambda_2}(\rho(\gamma)) \geq \eps \log \frac{\lambda_1}{\lambda_{d+1}}(\rho(\gamma))$ for all non-peripheral $\gamma \in \Gamma$: more precisely, Lemme 7.6 bounds the ratio ${\log \frac{\lambda_1}{\lambda_2}(\rho(\gamma))} \cdot \left( \log \frac{\lambda_1}{\lambda_{d+1}}(\rho(\gamma)) \right)^{-1}$ from below by an auxiliary quantity $\frac12 \chi(\gamma)$ (half the top Lyapunov exponent for the Hilbert geodesic flow corresponding to $\rho(\gamma)$); Proposition 7.2 and Corollaire 7.3 together give us $\eps > 0$ (coming from the H\"older regularity of the boundary $\del\Omega$) such that $\frac12 \chi(\gamma) > \left(1+\frac1\eps\right)^{-1}$

We may then show that there exists $\eps' = \eps'(\rho)>0$ such that $\log \frac{\sigma_1}{\sigma_2}(\rho(\gamma)) > \eps' \log \frac{\sigma_1}{\sigma_{d+1}}(\rho(\gamma)) + \hat{C}_\rho$ where $\hat{C}_\rho$ is some constant depending only on the representation; this last inequality. which suffices to establish the lower domination inequality (D\textsuperscript{-}), will follow from the inequality with the eigenvalue gaps, together with results of \cite{AMS} and \cite{Benoist1997} (as tied together in \cite{GGKW}, Theorem 4.12): 

Specifically, by \cite{CM12}, Th\'eor\`eme 7.28, we may assume that $\rho$ is strongly irreducible and Zariski-dense. Then \cite{GGKW}, Theorem 4.12
% \cite{AMS}, Theorem 5.17 
states that there is a finite subset $F \subset \Gamma$ 
such that for any $\gamma \in \Gamma$ there exists $f \in F$ 
such that
\[ \log \frac{\sigma_1}{\sigma_2}(\rho(\gamma)) \geq \log \frac{\lambda_1}{\lambda_2}(\rho(\gamma f)) - C_\rho  \]
where $C_\rho$ is some constant depending only on $\rho$, 
and similarly
\[ \log \frac{\lambda_1}{\lambda_{d+1}}(\rho(\gamma f)) \geq \log \frac{\sigma_1}{\sigma_{d+1}}(\rho(\gamma )) - C_\rho ,\]
and putting 
all of 
these inequalities together we obtain
\begin{align*} 
\log\frac{\sigma_1}{\sigma_2}(\rho(\gamma)) & \geq  \log \frac{\lambda_1}{\lambda_2}(\rho(\gamma f)) - C_\rho %\\& 
\geq \eps \log \frac{\lambda_1}{\lambda_{d+1}}(\rho(\gamma f)) - C_\rho \\
  & \geq \eps \log \frac{\sigma_1}{\sigma_{d+1}}(\rho(\gamma)) - (\eps+1) C_\rho
\end{align*}
as desired.
\end{proof}
\end{prop}

\section{Relation to Kapovich--Leeb} \label{sec:KL}
In \cite{KL}, Kapovich and Leeb develop a number of possible relative analogues of Anosov representations. Here we describe how some of these are related to the notion of relatively dominated subgroups described here.

The definitions in \cite{KL} are formulated in terms of discrete subgroups $\Gamma \leq G$ of semisimple Lie groups $G$; we reformulate them in terms of discrete and faithful representations, and in the specific case of $G = \SL(d,\real)$.

We also remark that the choice of a model Weyl chamber $\tau_{mod}$ in \cite{KL} is equivalent to the choice of a Cartan projection / set of roots, and in particular all of the definitions below are formulated in the specific case of the first and last simple roots $\left\{ \log \frac{\sigma_1}{\sigma_2}, \log \frac{\sigma_{d-1}}{\sigma_d} \right\}$.

Below, given a representation $\rho: \Gamma \to G$, we let $\Lambda_\Gamma$ denote the limit set of $\rho(\Gamma) \subset G$ in the flag variety $G/P_{1,d-1}$ corresponding to our chosen set of simple roots: a point in $G/P_{1,d-1}$ corresponds to a pair $(\xi, \xi^*) \in \proj(\real^d) \times \proj(\real^d)^*$ such that the line corresponding to $\xi$ is contained in the hyperplane represented by $\xi^*$. More specifically, $\Lambda_\Gamma$ is the closure of the set of accumulation points $\displaystyle (\xi, \xi^*) = \lim_{n\to\infty} \left( \Xi_\rho(\gamma_n)), \Xi^*_\rho(\gamma_n)) \right)$ for sequences $\gamma_n \to \infty$. 

\subsection{Relatively dominated implies relatively RCA}

\spacer \begin{defn}[\cite{KL}, Definition 7.6]
$\rho: \Gamma \to G$ is {\bf relatively RCA}  if \begin{itemize}
\item (regularity) $\log \frac{\sigma_1}{\sigma_2}(\rho(\gamma_n)) \to \infty$ for all sequences $(\gamma_n)_{n\in\nats}$ going to infinity in $\Gamma$.

\item (convergence) every point in $\Lambda_\Gamma$ is either a conical limit point or a bounded parabolic point, and the stabilizers of the bounded parabolic points are finitely generated.

\item (antipodality) $\Lambda_\Gamma$ is antipodal, i.e. every pair of points in the limit set (has a pair of lifts which) can be joined by a bi-infinite geodesic in $G/K$.
% https://arxiv.org/abs/1703.02160 : Definition 2.17; see also Example 2.19 in same paper.
\end{itemize}
\end{defn}

We remark that, roughly speaking, the relatively dominated condition (Definition \ref{defn:reldomrep}) may be seen as strengthening the regularity hypothesis while weakening the convergence and antipodality hypotheses. There is also a more subtle distinction involving the role of the intrinsic geometry of $\Gamma$, which we elaborate on more in the next subsection.

We also remark that projecting $\Lambda_\Gamma \subset \proj(\real^d) \times \proj(\real^d)^*$ to the first coordinate yields the limit set $\Lambda_{rel}$ from \S\ref{sec:reldom_relhyp} above.

\spacer \begin{defn}[\cite{KL}, Definition 7.1]
A subgroup $\Gamma \leq G$ is {\bf relatively asymptotically embedded} if it satisfies the regularity and antipodality conditions (as in the previous Definition), and admits a relatively hyperbolic structure $(\Gamma, \mathcal{P})$ such that there exists a $\Gamma$-equivariant homeomorphism $\del_\infty(\Gamma, \mathcal{P}) \to \Lambda_\Gamma$.
\end{defn}

\spacer\begin{thm}[\cite{KL}, Theorem 7.8] \label{thm:KL78}
$\rho$ is relatively RCA if and only if $\rho(\Gamma)$ is relatively asymptotically embedded.
\end{thm}

In particular, if $\rho: \Gamma \to G$ is relatively RCA, then $\Gamma$ is relatively hyperbolic. Below, we will use the notions of relative RCA and relative asymptotic embeddedness interchangeably.

\spacer \begin{thm} \label{thm:reldom_rRCA}
If $\rho: \Gamma \to G$ is relatively dominated, then $\rho(\Gamma)$ is relatively asymptotically embedded. 
\begin{proof}
Regularity is immediate from the lower domination inequality (D\textsuperscript{-}) and the quasi-equivalence of $|\gamma|_c$ and $\|a(\rho(\gamma)\|$ (Proposition \ref{prop:qiembed_rel}.)

Antipodality follows from transversality: given two points $\xi_\pm$ in the limit set, consider the associated hyperplanes $\theta_\pm$; then, by transversality we have a decomposition $\real^d = \xi_+ \oplus (\theta_+ \cap \theta_-) \oplus \xi_-$, which gives a bi-infinite geodesic joining the simplices associated to $(\xi_\pm,\theta_\pm)$ in the associated flag variety $G/P_{1,d-1}$---concretely, pick a diagonal matrix $A \in \SL(d,\real)$ respecting that decomposition, and consider the bi-infinite geodesic $\exp(tA)$.

Asymptotic embeddedness follows from Theorem \ref{thm:limitmaps} on the limit maps: more precisely, we can combine both limit maps from that Theorem into a single limit map $(\xi,\xi^*)$ into the flag manifold corresponding to our choice of $\tau_{mod}$, and this single limit map gives us our asymptotic embedding.
\end{proof} \end{thm}

\subsection{Uniform regularity and distortion, and equivalence of notions}

\spacer \begin{defn}[\cite{KL}, \S4.4.1]
$\Gamma$ is {\bf uniformly regular} if there exist constants $\mu,c>0$ such that $\log \frac{\sigma_1}{\sigma_2}(\rho(\gamma_n)) \geq \mu \|a(\rho(\gamma_n))\| - c$ for all $(\gamma_n) \subset \Gamma$ going to infinity.
\end{defn}

\spacer \begin{defn}
Suppose $\Gamma$ is hyperbolic relative to $\mathcal{P}$ and we have a representation $\rho: \Gamma \to G$. We say $\Gamma$ (or any subgroup $H \leq \Gamma$) is {\bf relatively undistorted by $\rho$} if $\rho$ induces (via any orbit map) a quasi-isometric embedding of the relative Cayley (sub)graph (cf. Proposition \ref{prop:qiembed_rel}) into the symmetric space, i.e. the cusped word-length $|\gamma|_c$ and the norm $\|a(\rho(\gamma))\|$ are quasi-equivalent for all $\gamma \in \Gamma$ (resp., for all $\gamma \in H$). 
\end{defn}

\spacer \begin{rmk}
Uniform regularity does not necessarily entail undistortedness: e.g. consider a hyperbolic mapping torus $\Gamma \subset \SO^+(1,3) \subset \SL(4,\real)$ which is abstractly isomorphic to $\pi_1 \Sigma_g \rtimes \ints$; the fiber groups (abstractly isomorphic to $\pi_1 \Sigma_g$) are exponentially distorted. $\Gamma$, being a geometrically finite subgroup of $\SO^+(1,3)$, is uniformly regular (and undistorted by the inclusion map); the fiber groups, being exponentially distorted subgroups, are not quasi-isometrically embedded and hence not undistorted by the inclusion map. However, they remain uniformly regular, since this is a condition purely on the Cartan projections and independent of word-length.
\end{rmk}

\spacer \begin{defn}
We say $\rho: \Gamma \to G$ is {\bf relatively uniform RCA and undistorted} if it satisfies the convergence and antipodality conditions, and moreover $\rho(\Gamma)$ is uniformly regular and $\Gamma$ is relatively undistorted by $\rho$.
\end{defn}

\spacer\begin{thm}[\cite{KL}, Theorem 8.25]
$\rho$ is relatively uniform RCA and undistorted if and only if it is relatively asymptotically embedded with uniformly regular peripheral subgroups and $\Gamma$ is relatively undistorted by $\rho$.
\end{thm}

\spacer \begin{rmk} \label{rmk:reldom_rURU}
We can in fact strengthen Theorem \ref{thm:reldom_rRCA} to say that if $\rho: \Gamma \to G$ is relatively dominated, then $\rho(\Gamma)$ is relatively uniform RCA and undistorted, since, via Proposition \ref{prop:qiembed_rel}, (D\textsuperscript{-}) is precisely the uniform regularity and undistortedness condition. 
\end{rmk}

\spacer \begin{rmk}
In the non-relative case, uniform regularity and undistortedness (URU) is equivalent to RCA \cite{KLP}. The proof goes through the notion of Morse subgroups and in particular requires some version of a higher-rank Morse lemma.
 \end{rmk}

\spacer \begin{thm} \label{thm:rRCA_reldom}
If $\rho: \Gamma \to G$ is such that $\rho(\Gamma)$ is relative uniform RCA and undistorted with peripherals also satisfying the quadratic gaps condition, then $\rho$ is relatively dominated.
\begin{proof}
Relative uniform RCA implies relative hyperbolicity of the source group (via Theorem \ref{thm:KL78}); 
this immediately gives us (RH).

As noted in Remark \ref{rmk:reldom_rURU}, (D\textsuperscript{-}) is exactly the uniform regularity and undistortedness condition.

It remains to check that the hypotheses in Definition \ref{defn:peri_conds} are satisfied. The quadratic gaps condition has been assumed. Upper domination follows from \cite{KL}, Corollary 5.13. Unique limits follow from the relative asymptotic embedding; by Proposition \ref{prop:trans_unitrans}, so does uniform transversality.
\end{proof}
\end{thm}

\section{Extending the definition}

As above let $\Gamma$ be a finitely-generated group which is hyperbolic relative to some finite collection $\mathcal{P}$ of finitely-generated subgroups satisfying (RH) (Definition \ref{defn:peri_preconds}.)

We say that a representation $\rho: \Gamma \to \PGL(d,\real)$ is $1$-relatively dominated (with domination constants $(\ubar{C},\ubar\mu,\bar{C},\bar\mu)$) if it is the composition of a $1$-relatively dominated representation $\hat\rho: \Gamma \to \GL(d,\real)$ (with the same domination constants) with the natural projection map $\pi: \GL(d,\real) \to \PGL(d,\real)$, or more generally if we can find a group $\hat\Gamma$, a 2-to-1 homomorphism $f: \hat\Gamma \to \Gamma$, and a $1$-dominated representation $\hat\rho: \hat\Gamma \to \GL(d, \real)$ such that $\pi \circ \hat\rho = \rho \circ f$  (cf. \cite{BPS}, Remark 3.4.)

Alternatively, we can continue to use Definitions \ref{defn:peri_conds} and \ref{defn:reldomrep}, since {\it ratios} of singular values remain unchanged under the reductions considered here, and we can continue to work with the same symmetric space and flag spaces. 

By considering the associated representations to $\GL(d,\real)$, we have that $\Gamma$ is hyperbolic relative to $\mathcal{P}$ in these cases as well (Theorem \ref{thm:reldom_relhyp}) and we have associated continuous, equivariant, dynamics-preserving, transverse limit maps (Theorem \ref{thm:limitmaps}.) By considering the associated representations, or by working directly with the hypotheses in Definitions \ref{defn:peri_conds} and \ref{defn:reldomrep}, the results from \S\ref{sub:df_prox} and \ref{sub:rel_qi_embed} continue to hold.

More generally, we may use the following standard fact from the representation theory of semisimple Lie groups:
\spacer \begin{thm}[cf. \cite{GW}, Proposition 4.3 and Remark 4.12]
Given $G$ a semisimple Lie group with finite center and $P$ a parabolic subgroup of $G$, there exists a finite dimensional irreducible representation $\phi = \phi_{G,P}: G \to \SL(V)$ such that $\phi(P)$ is the stabilizer (in $\phi(G)$) of a line in $V$.

$\phi$ induces maps $\beta: G/P \to \proj(V)$ and $\beta^*: G/Q \to \proj(V^*)$, where $Q$ is the opposite parabolic to $P$.

Moreover, if $P$ is non-degenerate, then $\ker \phi = Z(G)$ and $\phi$ is an immersion.
\end{thm}
For a construction, we refer the reader to \cite{GW}, \S4 (see also \cite{BCLS}, Theorem 2.12 and Corollary 2.13.) The irreducible representation $\phi_{G,P}$ is called a Pl\"ucker representation in \cite{BCLS}, or a Tits representation in \cite{BPS}. 

We now make the following
\spacer \begin{defn}
Given $\Gamma$ a finitely-generated group and $\mathcal{P}$ a finite collection of finitely-generated proper infinite subgroups satisfying (RH), $G$ a semisimple Lie group with finite center and $P$ a non-degenerate parabolic subgroup of $G$, we say that a representation $\rho: \Gamma \to G$ is $P$-dominated relative to $\mathcal{P}$ (with domination constants $(\ubar{C},\ubar\mu,\bar{C},\bar\mu)$) if $\phi_{G,P} \circ \rho: \Gamma \to \SL(V)$ is 1-dominated relative to $\mathcal{P}$ (with the same constants).
\end{defn}

Given a $P$-relatively dominated representation $\rho: \Gamma \to G$,  by applying Theorem \ref{thm:reldom_relhyp} to $\phi_{G,P} \circ \rho: \Gamma \to \SL(V)$, we have that $\Gamma$ is hyperbolic relative to $\mathcal{P}$ in these cases as well. By Theorem \ref{thm:limitmaps}, $\phi_{G,P} \circ \rho$ has associated continuous, equivariant, dynamics-preserving, transverse limit maps of $\del(\Gamma,\mathcal{P})$ into $\proj(V)$ and $\proj(V^*)$; we may compose these with $\beta^{-1}$ and $(\beta^*)^{-1}$ to obtain limit maps of $\del(\Gamma,\mathcal{P})$ into the flag varieties $G/P$ and $G/Q$. We may argue similarly to see that the results from \S\ref{sub:df_prox} and \ref{sub:rel_qi_embed} continue to hold. 

As a particular case of this, suppose $G = \SL(d,\real)$ and $P = P_k$ is the stabilizer of a $k$-plane in $G$. Then we may explicitly take $V = \bigwedge^k \real^d$ and $\phi_{G,P}: \SL(d,\real) \to \SL(V)$ to be the map given by the action of $\SL(d,\real)$ on the exterior product $V$ coming from the natural action $\SL(d,\real) \actson \real^d$. 

We note, very briefly, that 
\begin{align*}
\sigma_1(\bigwedge^k \rho(\gamma)) & = \sigma_1 \cdots \sigma_k (\rho(\gamma)), \\ \sigma_2(\bigwedge^k \rho(\gamma)) & = \sigma_1 \cdots \sigma_{k-1} \sigma_{k+1} (\rho(\gamma)),
\end{align*} 
and moreover $U_1(\bigwedge^k \rho(\gamma)) = U_k(\rho(\gamma))$ (in the sense that they represent the same $k$-dimensional subspace of $\real^d$) and
\begin{align*}
S_{D-1}(\bigwedge{}^k \rho(\gamma)) & = U_{D-1} (\wedge^k \rho(\gamma^{-1})) = \langle \theta \in \Gr_k(\real^d) : \theta \not\pitchfork S_{d-k}(\rho(\gamma)) \rangle
\end{align*}
(where $D := \binom{d}k = \dim \bigwedge^k \real^d$) 
and hence we may also equivalently and more directly define $P_k$-relatively dominated representations as in \S\ref{sec:reldom}, replacing $\frac{\sigma_1}{\sigma_2}$ with $\frac{\sigma_k}{\sigma_{k+1}}$ as appropriate, and similarly replacing projective space and its dual with the appropriate Grassmannians. 

\appendix
\section{Linear algebraic lemmas} \label{app:linalg}

We collect in this appendix various lemmas of quantitative linear algebra which are used in the proofs above and below, especially in sections \ref{sec:reldom_relhyp} and \ref{sec:limitmaps}. They appear in the order in which they are used above. These are elementary; many of them appear, with proof, in Appendix A of \cite{BPS}. 

Recall that, given $\xi, \eta \in \proj(\real^d)$ or $\Gr_{d-1}(\real^d)$, or more generally $\Gr_p(\real^d)$ for some $p$ between 1 and $d$, $d(\xi, \eta)$ will denote distance between $\xi$ and $\eta$ in the relevant Grassmannian.

Below, we say that $A \in \GL(d,\real)$ is {\bf $P_p$-proximal} if $\sigma_{p+1}(A) > \sigma_p(A)$. Recall that $U_p(A)$ is well-defined once $A$ is $P_p$-proximal.

\spacer \begin{lem} [\cite{GGKW}, Lemma 5.8; \cite{BPS}, Lemmas A.4, A.5] \label{lem:BPSA4A5}
Given $A, B \in \GL(d,\real)$ with $A$, $AB$ and $BA$ $P_p$-proximal, we have
\begin{align}
d(U_p(A), U_p(AB)) & \leq \frac{\sigma_1}{\sigma_d}(B) \frac{\sigma_{p+1}}{\sigma_p}(A) \label{eqn:A4} \\
d(BU_p(A), U_p(BA)) & \leq  \frac{\sigma_1}{\sigma_d}(B) \frac{\sigma_{p+1}}{\sigma_p}(A) \label{eqn:A5}
\end{align}
\end{lem}

\spacer \begin{lem} [\cite{BPS}, Lemma A.6] \label{lem:BPSA6}
Given any $P_p$-proximal $A \in \GL(d,\real)$ and any $p$-dimensional subspace $P \subset \real^d$, we have 
\[ d(A(P), U_p(A)) \leq \frac{\sigma_{p+1}}{\sigma_p}(A) \frac{1}{\sin \angle(P,S_{d-p}(A))} .\]
\end{lem}

\spacer \begin{lem} [\cite{BPS}, Lemma A.7] \label{lem:BPSA7}
Let $A, B \in \GL(d,\real)$. Suppose that $A$ and $AB$ are $P_p$-proximal, and let $\alpha := \angle (U_p(B), S_{d-p}(A))$. Then
\begin{align*}
\sigma_p(AB) & \geq (\sin \alpha) \, \sigma_p(A) \sigma_p(B) \\
\sigma_{p+1}(AB) & \leq (\sin \alpha)^{-1} \sigma_{p+1}(A) \sigma_{p+1}(B)
\end{align*}
\end{lem}

\begin{lem}[cf. \cite{QTZ}, Lemma A.24] \label{lem:QTZA24}
If $U_0$ and $V_0$ are complementary vector subspaces, and $U$ is the graph of $\Theta: U_0 \to V_0$, then we have
\[ s(U,V_0) \geq \frac{s(U_0,V_0)}{\|\id \oplus \Theta\|} .\]
\begin{proof}
Choose a vector $u \in U$ achieving the minimum gap $s(U,V_0)$. Scale it so that if we decompose $u$ into its $U_0$ and $V_0$ components, its $U_0$ component $u_0$ is a unit vector. From the law of sines,
\[ \frac{1}{s(U,V_0)} = \frac{\|u\|}{\sin \angle (u_0, V_0)} \leq \frac{\|u\|}{s(U_0,V_0)} \leq \frac{\|\id \oplus \Theta\|}{s(U_0,V_0)} \]

\begin{figure}[ht]
    \centering
    \includegraphics[width=0.363\textwidth]{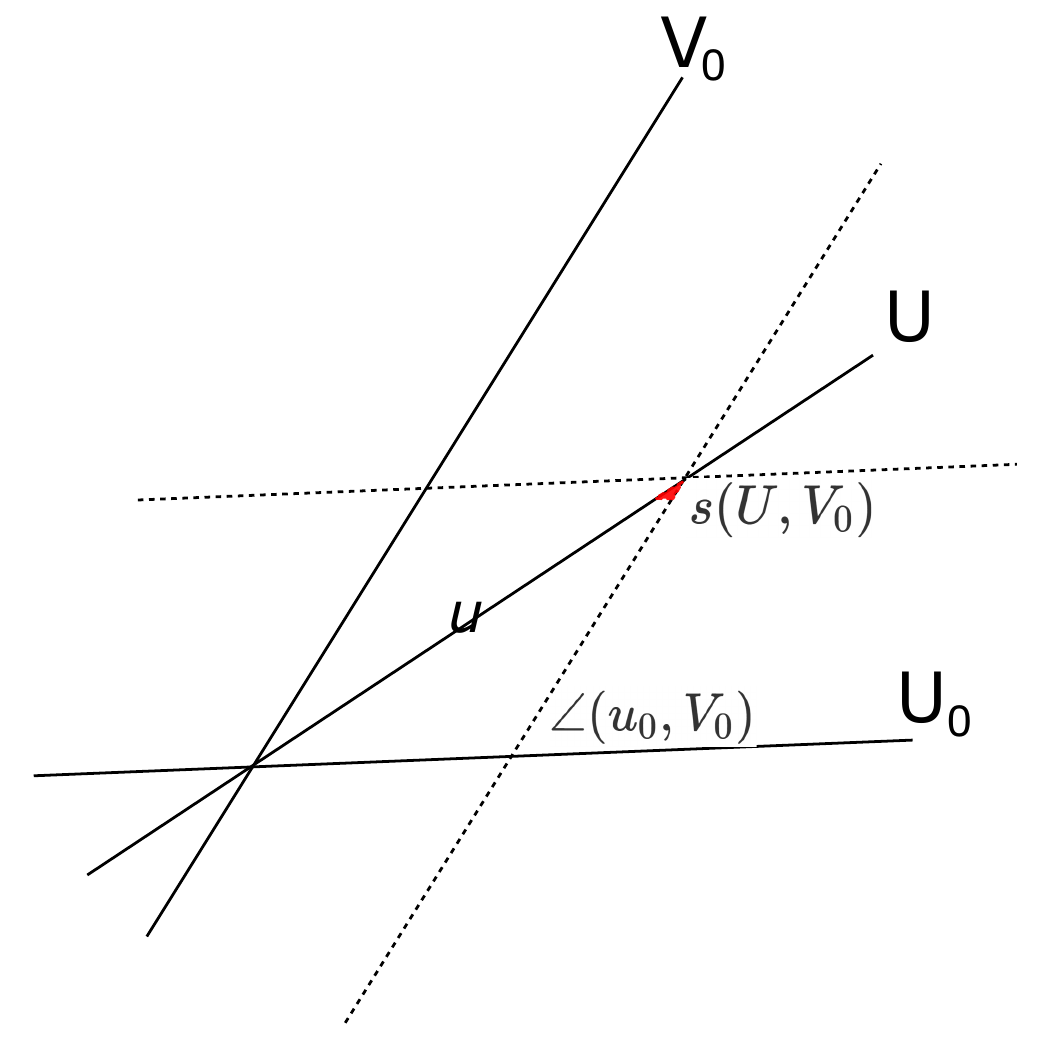}
    % \caption{}
    \label{fig:QTZ_lemA24}
\end{figure}
whence we have the desired inequality (see also illustration above.)
\end{proof}
\end{lem}

There is also the following slightly easier version
\spacer \begin{lem}[cf. \cite{QTZ}, Lemma A.24] \label{lem:QTZA24b}
If $U_0$ and $V_0$ are complementary vector subspaces, and $U$ is the graph of $\Theta: U_0 \to V_0$, then we have
\[ s(U,V_0) \leq \frac{1}{\|\id \oplus \Theta\|} .\]
\begin{proof}
Pick a vector $u \in U$ so that if we decompose $u$ into its $U_0$ and $V_0$ components, its $U_0$ component $u_0$ is a unit vector, and $\|u\| = \|\id \oplus \Theta\|$. By the law of sines, 
\[ \frac{\|\id\oplus\Theta\|}{\sin \angle (u_0, V_0)} = \frac{\|u\|}{\sin \angle (u_0, V_0)} = \frac{1}{\sin \angle (u,V_0)}  \]
but now $$\|\id \oplus \Theta\| \leq \frac{\|\id\oplus\Theta\|}{\sin \angle (u_0, V_0)} = \frac{1}{\sin \angle (u, V_0)} \leq \frac{1}{s(U, V_0)} ,$$ whence the desired inequality.
\end{proof}
\end{lem}

\section{A local version of Quas--Thieullen--Zarrabi} \label{app:QTZ}

The purpose of this appendix is to prove Theorem \ref{thm:QTZ}:
\spacer \begin{thm}[Theorem \ref{thm:QTZ}]
Let $(A_k)_{k\in\ints} \subset \GL(d,\real)$ be a sequence of matrices such that
there exists constants $C \geq 1$ and $\mu, \mu' \geq 0$, with $\frac1\mu \log 3C > 1$, such that the following axioms are satisfied: 
\begin{itemize}
\item (SVG-BG) 
% the sequence admits a locally uniform singular value gap at index 1 in the sense of Bochi--Gourmelon's domination condition i.e. 
for all $k \in \ints$ and all $n \geq 0$, 
\begin{align*}
\frac{\sigma_2}{\sigma_1} (A_{k+n-1} \cdots A_k) & \leq C e^{-n\mu} 
\end{align*}

\item (EC) 
% the approximate slow and fast spaces form uniformly Cauchy sequences, i.e. 
for all $k \in \ints$ and all $n \geq 0$, 
\begin{align*}
d(S_{d-1}(A_{k+n-1} \cdots A_k), S_{d-1}(A_{k+n} \cdots A_k)) & \leq C e^{-n\mu} ,\\
d(U_1(A_{k-1} \cdots A_{k-n}), U_1(A_{k-1} \cdots A_{k-(n+1)})) & \leq C e^{-n\mu} .
\end{align*}

\item (FI)\textsubscript{back}: 
for all $k \leq 0$ and $n, m \geq 0$
\begin{align*}
\frac{\sigma_1(A_{k+n-1}  \cdots A_{k-m})}{\sigma_1(A_{k+n-1} \cdots A_k) \cdot \sigma_1(A_{k-1} \cdots A_{k-m})} & \geq C^{-1} e^{-m \mu'}
\end{align*}
\end{itemize}

Then \begin{enumerate}[(i)]
\item for each $k \in \ints$ in the sequence we have a splitting $E^{u} \oplus E^{s}$ of $\real^d$ given by
\begin{align*}
E^{u}(k) & := \lim_{n\to\infty} U_1(A_{k-1} \cdots A_{k-n}) %\label{eqn:Ecu} 
\\
E^{s}(k) & := \lim_{n\to\infty} S_{d-1}(A_{k+n-1} \cdots A_k) %\label{eqn:Ecs}
\end{align*}
which is equivariant in the sense that $A_k E^*(k) = E^*(k+1)$ for all $k \in \ints$ and $* \in \{u,s\}$;
\item moreover, for all $k \leq 0$, we have a uniform lower bound $s_{min} = s_{min}(C,\mu,\mu')$ on the gap $s(E^{u}(k), E^{s}(k)) := \sin \angle (E^{u}(k),E^{s}(k))$ given by
\[ s(E^{u}(k), E^{s}(k)) \geq s_{min} := \frac23 (3e)^{-2r} \exp\left( -\frac{3/2}{1-e^{-\mu}} \right) C^{-(1+2r)} ,\]
where $r := \frac{\mu'}\mu$.
\end{enumerate}
\end{thm}

This statement is a mild generalization of a specific case of the main results of \cite{QTZ}; it is the particular statement which is needed above.

Here we are working with finite-dimensional real vector spaces, and hence many of the technical difficulties in \cite{QTZ}, which works in the more general case of Banach spaces, are significantly lightened.

We also deal only with the specific case where the singular value gap/s are at $p=1$ and $p=d-1$, and $A_k \in \SL(d,\real)$; these assumptions are natural in the application we have here.

\spacer \begin{rmk}
We can also follow the arguments of \cite{QTZ} to obtain a domination statement:
\begin{enumerate}[(i)]
\setcounter{enumi}{2}
% the splitting gives us a definite growth gap without going out to infinity: more precisely, 
\item there exists $n_{\min}$ depending only on $C, \mu, \mu'$ such that for all $k \leq 0$ and $n \geq n_{\min}$ with $k+n \leq 0$, 
\[ \frac{\left\| A_{n+k-1} \cdots A_k |_{E^{s}(k)} \right\|}{\mathbf{m}(A_{n+k-1} \cdots A_k |_{E^{u}(k)})} \leq \frac{16C}{9s_{min}^2} e^{-n\mu} .\]
\end{enumerate}
where $\mathbf{m}(A)$ denotes the bottom singular value of $A \in \GL(d,\real)$. We will not include the proof here, since we do not use this conclusion above.
\end{rmk}

We introduce some notation which will be useful below: write
\begin{itemize}
    \item $A(k,n)$ for the product $A_{k+n-1} \cdots A_k$, 
    \item $\sigma_i(k,n)$ as shorthand for $\sigma_i(A(k,n))$,
    \item $\tilde{U}(k,n) := U_1(A(k-n,n)) = A(k-n,n) U(k-n,n)$ and $V(k,n) = S_{d-1}(k,n)$.
\end{itemize}
We remark that, with these notations, we have
\begin{itemize}
    \item $U(k,n) \perp V(k,n)$;
    \item $\displaystyle E^u(k) = \lim_{n\to\infty} \tilde{U}(k,n)$ and 
    $\displaystyle E^s(k) = \lim_{n\to\infty} V(k,n)$.
\end{itemize}

\subsection{Existence and equivariance of limits}

It is immediate from (EC) that the limits $E^{u}(k)$ and $E^{s}(k)$ exist. In fact, we have the following uniform convergence estimates:

\spacer \begin{lem} \label{lem:QTZ_ECest}
For every $k, N \in \ints$, 
\begin{align*}
d( V(k,N), E^{s}(k)) & \leq \frac{Ce^{-N\mu}}{1 - e^{-\mu}} \\ d(\tilde{U}(k,N), E^{u}(k)) & \leq \frac{Ce^{-N\mu}}{1 - e^{-\mu}} .
\end{align*}
\begin{proof}
Immediate from by the triangle inequality and (EC).
\end{proof}
\end{lem}

Equivariance follows from using Lemma \ref{lem:BPSA4A5}, whence
\begin{align*}
E^{u}(k) & = A_{k-1} \cdots A_0 \cdot \lim_{n\to\infty} U_1\left( A_{-1} \cdots A_{-n} \right) \\
 & = A_{k-1} \cdots A_0 \cdot E^{u}(0) \qquad \mbox{for }k > 0 \\
E^{u}(0) & = A_{-1} \cdots A_k \cdot E^{u}(k) \\
\mbox{i.e. }E^{u}(k) & = A_k^{-1} \cdots A_{-1}^{-1} \cdot E^{u}(0) \qquad \mbox{for } k < 0
\end{align*}
and similarly 
\begin{align*}
E^{s}(0) & = A_{0}^{-1} \cdots A_{k-1}^{-1} \cdot \lim_{n\to\infty} U_{d-1} \left( A_k^{-1} \cdots A_{k+n-1}^{-1} \right) \\
\mbox{i.e. }E^{s}(k) & = A_{k-1} \cdots A_0 \cdot E^{s}(0) \qquad \mbox{for } k > 0 \\
E^{s}(0) & = A_{-1} \cdots A_k \cdot E^{s}(k) \\
\mbox{i.e. }E^{s}(k) & = A_k^{-1} \cdots A_{-1}^{-1} \cdot E^{s}(0) \qquad \mbox{for } k < 0
\end{align*}

\subsection{Proof of splitting}

The proof will involve, essentially, carefully refined versions of arguments that can be used to give the Raghunathan estimates \cite{Raghu1979}. Here we formulate these arguments in a series of lemmas, then assemble them into a proof of statement (ii), from which (i) follows. 

We follow the argument in \cite{QTZ} \S3, writing things out more concretely for our specific finite-dimensional, invertible case. We have supplied specific references to the corresponding / closely analogous lemmas in \cite{QTZ}, in the hope that the reader interested in also reading the result there may find these helpful.

For the next five lemmas (through Lemma \ref{lem:QTZ38}), fix $N$ sufficiently large so that 
\begin{equation}
\sum_{n\geq N} e^{-n\mu} = \frac{e^{-N\mu}}{1-e^{-\mu}} \leq \frac1{3C}
\label{ineq:suff_large_N*}
\end{equation} 

The following lemma tells us that whenever $m$ and $n$ are sufficiently large, $A(k,n)$ expands vectors in $U(k,m)$ at least $\frac23 \sigma_1(k,n)$. 
More precisely, we have
\spacer \begin{lem}[\cite{QTZ}, Lemma 3.3] \label{lem:QTZ33}
For every $n, m \geq N$ and $k \in \ints$, 
we have
\[ \forall u \in U(k,m) \quad \|A(k,n) u\| \geq \frac23 \sigma_1(k,n) \cdot \|u\| \]
\begin{proof}
From (EC) and our choice of $N$, we have, arguing as in the proof of Lemma \ref{lem:QTZ_ECest},
\[ d(U(k,n), U(k,m)) = d(V(k,n), V(k,m)) \leq \frac13 .\]

Given any unit vector $u \in U(k,m)$, write $u = v+w$ where $v \in U(k,n)$ and $w \in V(k,n) \perp v$. By the properties of the singular-value decomposition, $A(k,n)u = A(k,n)v + A(k,n)w$ is still an orthogonal decomposition, and we have
\begin{align*}
\|A(k,n) u\| & \geq \|A(k,n) v\| = \sigma_1(k,n) \cdot \cos \angle \left(U(k,m), U(k,n)\right) \\
 & = \sigma_1(k,n) \sqrt{1 - d (U(k,n), U(k,m))^2} \geq \sigma_1(k,n) \cdot \frac23
\end{align*} 
as desired.
\end{proof} \end{lem}

Recall $s(V,W)$ denotes the minimal gap $\inf \{\sin \langle (v,W) : v \in V, \|v\|=1\}$ between the subspaces. We now use the (FI) hypothesis to prove a lemma which states that whenever $m$ and $n$ are sufficiently large, we have a lower bound on the gap between the approximate fast space and the slow space. More precisely, we have
\spacer 
\begin{lem}[\cite{QTZ}, Lemma 3.4] \label{lem:QTZ34}
For all $k \leq 0$ and $m \geq N$,
\[ s(A(k-N,N)U(k-N,m), E^{s}(k)) \geq \frac23 C^{-1} e^{-N\mu'} \]

\begin{proof}
Write $W_{k,m} := A(k-N,N) U(k-N,m)$. 

Let $w \in W_{k,m}$ be a unit vector, and (given any $n \geq N$) write $w = w_1 + w_2$ where $w_1 \in U(k,n)$ and $w_2 \in V(k,n)$. Since $w = A(k-N,N)u$ for some $u \in U(k-N,m)$, we have, from Lemma \ref{lem:QTZ33} and the properties of the singular-value decomposition,
\begin{align*}
\|A(k,n)w\| & = \|A(k-N,N+n) u\| \geq \frac23 \cdot \sigma_1(k-N,N+n) \|u\| \quad\mbox{ and}\\
\|w\| & \leq \sigma_1(k-N,N)\|u\|
\end{align*}
so
\begin{align*}
\|A(k,n)w\| & \geq \frac23 \cdot \frac{\sigma_1(k-N,N+n)}{\sigma_1(k-N,N)} \|w\|
\end{align*}
On the other hand we also have
\begin{align*}
\|A(k,n) w_1\| = \sigma_1(k,n) \|w_1\| & & \mbox{and } & &
\|A(k,n) w_2\| \leq \sigma_2(k,n) \|w_2\| 
\end{align*}
or, together,
\begin{align*}
\|A(k,n) w\| & \leq \sigma_1(k,n) \left( \|w_1\| + \frac{\sigma_2}{\sigma_1}(k,n) \|w_2\| \right) 
\end{align*}

Combining the two estimates of $\|A(k,n)w\|$ we obtain
\begin{align*}
\|w_1 + w_2\| = \|w\| & \leq \frac32 \frac{\sigma_1(k-N,N) }{\sigma_1(k-N,N+n)} \|A(k,n)w\| \\
 & \leq \frac32 \frac{\sigma_1(k-N,N) \cdot \sigma_1(k,n)}{\sigma_1(k-N,N+n)} \left( \|w_1\| + \frac{\sigma_2}{\sigma_1}(k,n) \|w_2\| \right) .
\end{align*}
By property (FI)\textsubscript{back} we have 
\[ \frac{\sigma_1(k-N,N+n)}{\sigma_1(k-N,N) \cdot \sigma_1(k,n)} \geq C^{-1} e^{-N\mu'} \]
and using this and (SVG-BG) on the last inequality we further obtain
\begin{align*}
\|w_1 + w_2\| = \|w\| & \leq \frac32 C  e^{N\mu'} \left( \|w_1\| + C e^{-n\mu} \|w_2\| \right) \\ & = \frac32 Ce^{N\mu'} \|w_1\| \left( 1 + Ce^{-n\mu} \frac{\|w_2\|}{\|w_1\|} \right) .
\end{align*}
Now we claim that $\frac{\|w_2\|}{\|w_1\|} = \frac{\sqrt{1-\|w_1\|^2}}{\|w_1\|} = \sqrt{\|w_1\|^{-2} - 1}$ is uniformly bounded above by some upper bound $B$ that depends only on the constants $C$ and $\mu'$. If not, $\|w_1\|$ gets arbitrarily close to zero; in particular, it can be made smaller than $(3Ce^{N\mu'})^{-1}$. Then $1 = \|w\| \leq \frac12+ \frac32 C^2 e^{N\mu'-n\mu}  < 1$ for all large enough $n$, which is a contradiction.

With this upper bound in hand, we then have
\[ s(W_{k,m}, V(k,n)) \geq \frac{\|w_1\|}{\|w\|}
\geq \frac23 C^{-1} e^{-N\mu'} \left(1 + C e^{-n\mu} B \right)^{-1} .\]
and we conclude by letting $n \to \infty$, since $\displaystyle \lim_{n\to\infty} V(k,n) = E^s(k)$.
\end{proof}
\end{lem}

This does not quite suffice, since as $N \to \infty$ these lower bounds go to zero, and so {\it a priori} we could still have the minimal gap between the fast and slow spaces collapsing to zero. Onwards we push ...
The idea is to do some kind a multiplicative block analysis, using Lemma \ref{lem:QTZ34} to control each block, and using the subsequent lemma/s to control the remaining exponential terms.
This we will achieve using, on the one hand, a lemma which controls expansion on the slow spaces:
\spacer 
\begin{lem}[\cite{QTZ}, Lemma 3.5] \label{lem:QTZ35}
For all $n \geq N$ and $k \leq 0$, 
\[ \left\|A(k,n) \big|_{E^{s}(k)} \right\| \leq \frac23 \cdot \sigma_1(k,n) \cdot e^{-(n-N)\mu}  .\]
\begin{proof}
Let $w \in E^{s}(k)$, and write $w = w_1+w_2$ where $w_1 \in U(k,n)$ and $w_2 \in V(k,n)$. Note we have
\[ d( V(k,n), E^{s}(k) ) \leq \sum_{m \geq n} C e^{-m\mu} \leq \frac{Ce^{-n\mu}}{1-e^{-\mu}} \leq \frac13 e^{-(n-N)\mu} \]
by Lemma \ref{lem:QTZ_ECest} and our choice of $N$. Then
\begin{align*}
\|A(k,n) w_1\| & \leq \sigma_1(k,n) \|w_1\| \leq \sigma_1(k,n) \cdot \frac13 e^{-(n-N)\mu} \|w_2\| \\
\|A(k,n) w_2\| & \leq \sigma_2(k,n) \|w_2\| \\
\end{align*}
and putting these two together we obtain
\begin{align*}
\|A(k,n) w\| & \leq \|A(k,n) w_1 \| + \|A(k,n) w_2\| \\
& \leq \sigma_1(k,n) \left(\frac13 e^{-(n-N)\mu} + \frac{\sigma_2}{\sigma_1}(k,n) \right) \|w_2\| \\
 & \leq \sigma_1(k,n) \left(\frac13 e^{-(n-N)\mu} + Ce^{-n\mu} \right) \|w\| \\
 & \leq \frac23 e^{-(n-N)\mu} \cdot  \sigma_1(k,n) \|w\|
\end{align*}
as desired.
\end{proof}
\end{lem}

On the other hand, we have the following lemma which gives us some control on 
the slow space components of images of approximate fast spaces

\spacer \begin{lem}[\cite{QTZ}, Lemma 3.7] \label{lem:QTZ37}
Let $N$ be sufficiently large. 

\begin{enumerate}[(i)]
\item Given $w \in \real^d$ a unit vector, write $w = w_1 + w_2$ where $w_1 \in U(k-nN,nN)$ and $w_2 \in E^{s}(k-nN)$. 
Then we have $\|w_2\| \leq \frac32$ 
\item The operator $\Gamma_{-n}: U(k-nN,nN) \to E^{s}(k-nN)$ whose graph is $$W_{n+1} := A(k-(n+1)N, N) \, U(k-(n+1)N, (n+1)N)$$
satisfies $\|\Gamma_{-n}\| \leq \frac94 C e^{N\mu'}.$
\end{enumerate}

\begin{proof}
By Lemma \ref{lem:QTZ_ECest} and our choice of $N$, $d( V(k-nN, nN), E^{s}(k-nN)) < \frac13$; basic trigonometry then implies $\frac{\|w_2|_{U(k-nN,nN)}\|}{ \|w_2|_{V(k-nN,nN)}\|} \leq \frac{1/3}{\sqrt{1-(1/3)^2}} = \frac{1}{2\sqrt{2}}$.

Hence, from the orthogonal decomposition $w_2 = w_2|_{V(k-nN,nN)} + w_2|_{U(k-nN,nN)}$, we get
\begin{align*}
\|w_2\| & \leq \left( 1 + \frac{1}{2\sqrt{2}} \right) \left\|w_2 |_{V(k-nN, nN)} \right\|
\end{align*}
and since $w_2|_{V(k-nN,nN)} = w|_{V(k-nN,nN)}$ we have $\|w_2|_{V(k-nN,nN)}\| \leq 1$, so in fact
\begin{align*}
\|w_2\| & \leq 1 + \frac{1}{2\sqrt{2}} < \frac32
\end{align*}

For (ii):
applying Lemma \ref{lem:QTZA24b} to the operator $\Gamma_{-n}: U(k-nN,nN) \to E^{s}(k-nN)$ gives us
\[ \|\id \oplus \Gamma_{-n}\| \leq \frac{1}{s(W_{n+1},E^{s}(k-nN))} \leq \frac32 C e^{N\mu'} \]
where the last inequality follows from Lemma \ref{lem:QTZ34}, which gives $s(W_{n+1}, E^{s}(k-nN)) \geq \frac23 C^{-1} e^{-N\mu'}$.

Now we observe that $\Gamma_{-n} = q_{-n} \circ (\id \oplus \Gamma_{-n})$ where $q_{-n}$ is projection to $E^s(k-nN)$ parallel to $U(k-nN,nN)$. We observe that we may rewrite statement (i) as the assertion that $\|q_{-n}\| \leq \frac32$. We put all of this together to obtain
\[ \|\Gamma_{-n}\| \leq \|q_{-n}\| \|\id \oplus \Gamma_{-n}\| \leq \frac94 Ce^{N\mu'} \]
as desired.
\end{proof}
\end{lem}

Now we can put everything together:
\spacer \begin{lem}
[\cite{QTZ}, Lemma 3.8] 
\label{lem:QTZ38}
% Let $N$ be sufficiently large.
For every $n \geq 1$,
\[ s\left(  \tilde{U}(k,nN), E^{s}(k) \right) \geq \frac23 C^{-1} e^{-N\mu'}
\prod_{j=1}^{n-1} \left( 1 + \frac32 D e^{N\mu'} j^{-3} \right)^{-1} .\]

\begin{proof}
From Lemma \ref{lem:QTZA24} we have
\[ s\left( \tilde{U}(k,nN), E^{s}(k) \right) \geq \frac{s( \tilde{U}(k,N), E^{s}(k))}{\|\id \oplus \Xi_n\|} \]
where $\Xi_n: \tilde{U}(k,N) \to E^{s}(k)$ is such that $\tilde{U}(k,nN)$ is the graph of $\Xi_n$. Since $\tilde{U}(k,N) = A(k-N,N) \, U(k-N,N)$, we have $s(\tilde{U}(k,N), E^{s}(k)) \geq \frac23 C^{-1} e^{-N\mu'}$ from Lemma \ref{lem:QTZ34}
and it remains to bound $\|\id \oplus \Xi_n\|$.

Write $A_{-n} := A(k-nN, N) = \left[ \begin{array}{cc} a_{-n} & 0 \\ c_{-n} & d_{-n} \end{array} \right]$ where 
\begin{align*}
a_{-n}&: U(k-nN,nN) \to U(k-(n-1)N,(n-1)N), \\
c_{-n}&: U(k-nN,nN) \to E^{s}(k-(n-1)N), \\
d_{-n}&: E^{s}(k-nN) \to E^{s}(k-(n-1)N)
\end{align*}
and the 0 in the upper-right corner comes from the equivariance of the slow spaces; here we adopt the notational convention $U(k,0) :=  \tilde{U}(k,N)$.

Then $A_{-n}^n := A_{-1} \cdots A_{-n} = A(k-nN, nN) := \left[ \begin{array}{cc} a_{-n}^n & 0 \\ c_{-n}^n & d_{-n}^n \end{array} \right]$. 
Now we have
\begin{align*}
A_{-(n+1)}^{n+1} & = \left[ \begin{array}{cc} a_{-n}^n & 0 \\ c_{-n}^n & d_{-n}^n \end{array} \right] \left[ \begin{array}{cc} a_{-(n+1)} & 0 \\ c_{-(n+1)} & d_{-(n+1)} \end{array} \right] 
\end{align*}
and examining in particular the bottom-left entry of this product, we have
\begin{align*}
c_{-(n+1)}^{n+1} & = c_{-n}^n a_{-(n+1)} + d_{-n}^n c_{-(n+1)}. \end{align*}
Since $a^{n+1}_{-(n+1)} = a^n_{-n} a_{-(n+1)}$,
\begin{align}
c_{-(n+1)}^{n+1} (a_{-(n+1)}^{n+1})^{-1} & = c_{-n}^n (a_{-n}^n)^{-1} + d_{-n}^n c_{-(n+1)} (a_{-(n+1)})^{-1} (a_{-n}^n)^{-1} . \label{eqn:block_mat_mult}
\end{align}
Now, firstly, we observe that $\Xi_n = c_{-n}^n (a_{-n}^n)^{-1}$, since from the block structure of $A_{-n}^n$ we see that $c_{-n}^n (a_{-n}^n)^{-1}$ maps from $U(k,0) = \tilde{U}(k,N)$ 
% to $U(k-nN,nN)$, and then 
to $E^{s}(k)$ with graph $A(k-nN,nN) \, U(k-nN,nN) = \tilde{U}(k,nN)$.

Secondly, we write $c_{-(n+1)} (a_{-(n+1)})^{-1} =: \Gamma_{-n}: U(k-nN,nN) \to E^{s}(k-nN)$ (see observation 1 below), and note that (\ref{eqn:block_mat_mult}) combined with the triangle inequality, give us (writing $\id$ for the identity on the appropriate complementary subspace, so that $\id \oplus \Xi_n$ is a linear endomorphism of $\real^d$)
\begin{align} 
\|\id \oplus \Xi_{n+1}\| \leq \|\id \oplus \Xi_n\| \left( 1 + \frac{\left\| d^n_{-n} \right\| \|\Gamma_{-n}\| \left\| (a^n_{-n})^{-1} \right\|}{\|\id \oplus \Xi_n\|} \right) . \label{ineq:block_iterative_bd}
\end{align}
To bound the last quantity that appears, we observe that 
\begin{enumerate}
\item $\Gamma_{-n}$ is precisely the operator from Lemma \ref{lem:QTZ37}(ii): $c_{-(n+1)} (a_{-(n+1)})^{-1}$ maps from $U(k-nN,nN)$ 
% to $U(k-(n+1)N,(n+1)N)$, and then 
to $E^{s}(k-nN)$ with graph $A(k-(n+1)N,N) \, U(k-(n+1)N, (n+1)N)$.

Hence, from Lemma \ref{lem:QTZ37}, $\|\Gamma_{-n}\| \leq \frac94 Ce^{N\mu'}$
\item We have
\begin{align*}
\frac{\left\| (a^n_{-n})^{-1} \right\|}{\|\id \oplus \Xi_n\|} & \leq \left( \sigma_1(k-nN, nN) \right)^{-1} 
\end{align*}
since $(a^n_{-n})^{-1} = \left( A^n_{-n} |_{U(k-nN,nN)} \right)^{-1} \circ (\id \oplus \Xi_n)$ {\footnotesize (easier to see by writing $a_n^{-n}$ as composition of $\left( A^n_{-n} |_{U(k-nN,nN)} \right)^{-1}$ with projection onto $\tilde{U}(k,N)$ parallel to $E^{s}(k)$)} and $\|\left( A^n_{-n} |_{U(k-nN,nN)} \right)^{-1}\| = \left( \sigma_1(k-nN,nN) \right)^{-1}$

\item From Lemma \ref{lem:QTZ35},
\begin{align*}
\left\| d^n_{-n} \right\| & \leq \frac23 \cdot \sigma_1(k-nN, nN) e^{-(n-1)N\mu}
\end{align*}
\end{enumerate}
Combining the bounds from these three observations, we obtain
\begin{align*}
\frac{\left\| d^n_{-n} \right\| \|\Gamma_{-n}\| \left\| (a^n_{-n})^{-1} \right\|}{\|\id \oplus \Xi_n\|} & \leq \frac23 \cdot  \frac{\sigma_1(k-nN,nN)}{\sigma_1(k-nN,nN)} \cdot e^{-(n-1)N\mu} \cdot \frac94 Ce^{N\mu'} \\
 & = \frac32 C \cdot e^{-(n-1)N\mu} e^{N \mu'}
\end{align*}
Since $\Xi_1 \equiv 0$, $\|\id \oplus \Xi_1\| = 1$. We then use this together with (\ref{ineq:block_iterative_bd}), as in \cite{QTZ}, to obtain the iterative bound 
\begin{align*}
\|\id \oplus \Xi_n\| & \leq \prod_{j=0}^{n-2} \left( 1 + \frac32 C e^{N(\mu'-j\mu)} \right)
\end{align*}
and we are done.
\end{proof}
\end{lem}

An elementary argument, done in \cite{QTZ}, gives us control over the infinite product that appears as we take $n \to \infty$:
\spacer \begin{lem}[\cite{QTZ}, Lemma 3.10] \label{lem:QTZ310}
Fix constants $C, N, \mu'$ and $\mu > 0$. Then
\[ \prod_{j=0}^\infty \left[ 1+\frac32 C^{-1} e^{N(\mu'-j\mu)} \right] \leq 
\exp\left(\frac32 C^{-1} e^{-N\mu'} \cdot (1-e^{-N\mu})^{-1} \right) < \infty .\]

\begin{proof}
Write $a_j := \frac32 C^{-1} e^{N\mu'} e^{-jN\mu}$. If $\mu > 0$, then $\sum_{j=1}^\infty a_j$ converges, and hence so does our infinite product $\prod_{j=1}^\infty (1+a_j)$.

In particular, we have $\prod_{j=1}^\infty (1+a_j) = \exp\left( \sum_{j=1}^\infty \log(1+a_j) \right) \leq \exp (\sum_{j=1}^\infty a_j)$ since $a_j > 0$. Now observe $\sum_{j=1}^\infty a_j = \frac32 C^{-1} e^{N\mu'} \sum_{j=0}^\infty e^{-jN\mu}$, and $\sum_{j=0}^\infty e^{-jN\mu} = (1-e^{-N\mu})^{-1}$. 
\end{proof}
\end{lem}

Now for the final assembly:
\begin{proof}[Proof of splitting]
From Lemma \ref{lem:QTZ38} and Lemma \ref{lem:QTZ310}, 
we have
\begin{align*} 
s(\tilde{U}(k, nN), E^{s}(k)) & \geq \frac23 C^{-1} e^{-N\mu'} \prod_{j=0}^{n-2} \left[ 1 + \frac32 C e^{N(\mu'-j\mu)} \right]^{-1} \\ 
 & \geq \frac23 C^{-1} \exp \left(-\frac32 C^{-1}e^{-N\mu'} (1-e^{-N\mu})^{-1} - N\mu' \right)  .
 \end{align*}

Now recall that $N$ satisfies (\ref{ineq:suff_large_N*}), i.e. $N \geq \frac1\mu \left( \log 3C - \log(1-e^{-\mu}) \right) > \frac1\mu \log 3C$. 
Pick $N \leq \frac2\mu \log 3C$,  
so that $e^{-N\mu'} \geq (3C)^{-2r}$ where $r := \frac{\mu'}\mu$.
Such a choice of $N$ exists from our hypothesis that $\frac1\mu \log 3C > 1$.
Then
\begin{align*} 
s(\tilde{U}(k, nN), E^{s}(k))
 & \geq \frac23 C^{-1} \exp \left(-\frac32 \frac{C^{-(1+2r)} 9^{-r}}{1-e^{-N\mu}} - 2r \log 3C \right) \\
 & \geq \frac23  (3e)^{-2r} C^{-(1+2r)} \exp \left( -\frac{3^{1-2r} C^{-(1+2r)}}{2(1-e^{-N\mu})} \right) \\
 & \geq \frac23 (3e)^{-2r} \exp\left( -\frac{3/2}{1-e^{-\mu}} \right) C^{-(1+2r)}.
\end{align*}

Finally, using the fact that $\tilde{U}(k, nN) \to E^{u}(k)$ as $n \to\infty$, we are done.
\end{proof}

% \bibliographystyle{alpha}
% \AtNextBibliography{\footnotesize}
\printbibliography

\end{document}